\documentclass[10pt, reqno,amsmath,amsthm,amssymb,amscd]{amsart}
\usepackage{amsmath}
\usepackage{mathrsfs,amssymb, amscd,amsmath,amsthm}
\usepackage[enableskew,vcentermath]{youngtab}
\usepackage{multicol}\multicolsep=0pt
\usepackage{tikz}

\newcommand{\clr}{rgb:black,1;blue,4;red,1}

\newcommand{\bdot}{ node[circle, draw, fill=\clr, thick, inner sep=0pt, minimum width=4pt]{}}

\newcommand{\ob}[1]{\mathsf{#1}}

\newcommand{\B}{\mathcal{B}}
\newcommand{\CB}{\mathcal{CB}}
\newcommand{\AB}{\mathcal{AB}}

\newcommand{\lcap}{
\begin{tikzpicture}[baseline = 3pt, scale=0.5, color=\clr]
        \draw[-,thick] (1,0) to[out=up, in=right] (0.53,0.5) to[out=left, in=right] (0.47,0.5);
        \draw[-,thick] (0.49,0.5) to[out=left,in=up] (0,0);
\end{tikzpicture}
}
\newcommand{\lcup}{
\begin{tikzpicture}[baseline = 6pt, scale=0.5, color=\clr]
        \draw[-,thick] (1,1) to[out=down, in=right] (0.53,0.5) to[out=left, in=right] (0.47,0.5);
        \draw[-,thick] (0.49,0.5) to[out=left,in=down] (0,1);
\end{tikzpicture}
}

\newcommand{\swap}{
\begin{tikzpicture}[baseline = 3pt, scale=0.5, color=\clr]
        \draw[-,thick] (0,0) to[out=up, in=down] (1,1);
        \draw[-,thick] (1,0) to[out=up, in=down] (0,1);
\end{tikzpicture}
}

\newcommand{\xdot}{
\begin{tikzpicture}[baseline = 3pt, scale=0.5, color=\clr]
\draw[-,thick] (0,0) to[out=up, in=down] (0,1);
\draw(0,0.5) \bdot;
\end{tikzpicture}
}

 \providecommand{\og}{``}
\providecommand{\fg}{''} \providecommand{\smfandname}{and}

\usepackage{amssymb}
\usepackage{mathrsfs,amssymb}
\usepackage{multicol}\multicolsep=0pt
\usepackage[enableskew,vcentermath]{youngtab}

\usepackage[sort]{cite}
\usepackage{xcolor,graphicx}

\def\crulefill{\leavevmode\leaders\hrule height 1pt\hfill\kern 0pt}
\long\def\QUERY#1{%
\leavevmode\newline%
\noindent$\star\star\star$\thinspace\textsf{Comment/Query}\crulefill\newline%
   \space #1\newline\hbox to 120mm{\crulefill}$\star\star\star$\newline}
\newtheorem{Theorem}{Theorem}[section]%[chapter] theorem number will %continue
\newtheorem{Lemma}[Theorem]{Lemma}
\newtheorem{Cor}[Theorem]{Corollary}
\newtheorem{Prop}[Theorem]{Proposition}

 \theoremstyle{definition}

\newtheorem{Defn}[Theorem]{Definition}

\newtheorem{rem}[Theorem]{Remark}

\newtheorem{Assumption}[Theorem]{Assumption}

\numberwithin{equation}{section}
\theoremstyle{definition}
\makeatletter
\def\enumerate{\begingroup\ifnum\@enumdepth>3\@toodeep\else
      \advance\@enumdepth\@ne
      \edef\@enumctr{enum\romannumeral\the\@enumdepth}%
      \topsep\z@\parskip\z@
      \list{\csname label\@enumctr\endcsname}
        {\@nmbrlisttrue\let\@listctr\@enumctr
         \parsep\z@\itemsep\z@\topsep\z@
         \setcounter{\@enumctr}{0}
         \def\makelabel##1{\hss\llap{\rm ##1}}
       }\fi}

\makeatother

\let\bar=\overline
\let\epsilon=\varepsilon
\def\({\big(}
\def\){\big)}

\def\0{\underline{0}}

\DeclareMathOperator{\End}{End}

\def\hf{\frac{1}{2}}

\def\Hom{\text{Hom}}

{\catcode`\|=\active
  \gdef\set#1{\mathinner{\lbrace\,{\mathcode`\|"8000%
                                   \let|\midvert #1}\,\rbrace}}
  \gdef\seT#1{\mathinner{\Big\lbrace\,{\mathcode`\|"8000%
                                   \let|\midverT #1}\,\Big\rbrace}}
}
\def\midvert{\egroup\mid\bgroup}
\def\midverT{\egroup\,\Big|\,\bgroup}

\def\Set[#1]#2|#3|{\Big\{\ #2\ \Big| \
           \vcenter{\hsize #1mm\centering #3}\Big\}}

\def\wt{\widetilde}

\def\Hom{{\rm Hom}}

\def\mfg{{\mathfrak g}}

\def\Set{{\rm Set}}

\newcommand{\C}{\mathcal{C}}

\def\wt{{\text{\rm wt}}}

\def\Hom{\text{Hom}}%
\def\textsf#1{{\textit{#1}}}%

\begin{document}
\title{Representations of  Brauer category and  categorification }
\author{Hebing Rui, Linliang Song}
\address{H.R.  School of Mathematical Science, Tongji University,  Shanghai, 200092, China}\email{hbrui@tongji.edu.cn}
\address{L.S.  School of Mathematical Science, Tongji University,  Shanghai, 200092, China}\email{llsong@tongji.edu.cn}

\thanks{H. Rui is supported  partially by NSFC (grant No.  11571108).  L. Song is supported  partially by NSFC (grant No.  11501368). }
%\date{\today}
\sloppy \maketitle

\begin{abstract}We study  representations of the  locally unital and locally finite dimensional  algebra $B$ associated to the Brauer category $\B(\delta_0)$  with defining parameter $\delta_0$  over an   algebraically closed   field $K$ with characteristic $p\neq 2$. The Grothendieck group $K_0(B\text{-mod}^\Delta)$  will be used to  categorify  the integrable  highest weight $\mathfrak {sl}_{K}$-module $ V(\varpi_{\frac{\delta_0-1}{2}})$ with the fundamental weight $\varpi_{\frac{\delta_0-1}{2}}$ as its highest weight,  where  $B$-mod$^\Delta$  is a subcategory of $B$-lfdmod in which each object has a finite $\Delta$-flag, and $\mathfrak {sl}_{K}$ is either $\mathfrak{sl}_\infty$ or $\hat{\mathfrak{sl}}_p$
depending on whether $p=0$   or $2\nmid p$. As $\mfg$-modules, $\mathbb C\otimes_{\mathbb Z} K_0(B\text{-mod}^\Delta)$  is isomorphic to $ V(\varpi_{\frac{\delta_0-1}{2}})$, where $\mfg$ is a Lie subalgebra of $\mathfrak {sl}_{K}$ (see Definition~\ref{liesub}). When $p=0$, standard $B$-modules and
projective covers of simple $B$-modules correspond to monomial basis and  so-called quasi-canonical basis of $V(\varpi_{\frac{\delta_0-1}{2}}) $, respectively.
 \end{abstract}

\section{Introduction}\label{affb}

Brauer algebras were introduced in \cite{B} so as to study decomposition of tensor product of natural modules for orthogonal and symplectic groups over $\mathbb C$. They are associative  algebras defined over a commutative ring $K$ containing the multiplicative identity $1$.
When $K$ is the complex field,  decomposition numbers of Brauer algebras  have been computed in \cite{CV, ES1} via Kazhdan-Lusztig polynomials associated to the Weyl group of type $D$.  See also \cite{RS2} for Birman-Murakami-Wenzl algebras where Kazhdan-Lusztig polynomials associated to affine Weyl groups of classical types come into play.
Unlike Ariki's remarkable work   in \cite{Ari},  Lie  algebra actions were not available on the  Grothendieck groups of Brauer algebras and  Birman-Murakami-Wenzl algebras at that time.

  Recently, the authors\cite{RS3} introduced  (affine) Brauer categories and cyclotomic Brauer categories.
A special case of their Brauer category is isomorphic to the  Brauer diagram category $\B(\delta_0)$ in \cite{LZ}. The additive Karoubi envelope of
  $\B(\delta_0)$(which  is known as the Deligne category $\underline{Rep~ } O_{\delta_0}$\cite{De1, De2}) has already been studied recently so as  to classify the indecomposable direct  summands of tensor powers of the standard representation of orthosymplectic supergroup\cite{CH}.
Motivated by Reynolds' work on oriented Brauer categories~\cite{Re}  and Brundan's work on oriented skein  categories~\cite{Br},
we  consider the infinite dimensional associative algebra $B$ associated to $\B(\delta_0)$ over an   algebraically closed   field $K$. This is
   a locally unital and locally finite dimensional  $K$-algebra which  has a triangular decomposition. In \cite{CZ}, this  triangular structure of $B$ was considered  so as to study certain truncations of $B$. Let $B$-lfdmod be the category of all locally finite dimensional  left $B$-modules. We show that $B$-lfdmod is an upper finite fully  stratified category in the sense of \cite{BS}.  It contains  standard objects
  $
\{\Delta(\lambda)\mid    \lambda\in \Lambda_p\}$ and proper standard objects $
\{\bar\Delta(\lambda)\mid    \lambda\in \Lambda_p\}$,  where $\Lambda_p$ is the set of all $p$-regular partitions.
Any simple $B$-module can be realized as the simple head $L(\lambda)$ of some $\bar\Delta(\lambda)$,  and   simple $B$-modules can be classified in this way.

Consider the subcategory $B\text{-mod}^\Delta$ of $B$-lfdmod such that each object has a finite  $\Delta$-flag. The Grothendieck group $K_0(B\text{-mod}^\Delta)$ has two basis $\{[\Delta(\lambda)]\mid \lambda\in \Lambda_p\}$ and $\{[P(\lambda)]\mid \lambda\in \Lambda_p\}$, where $P(\lambda)$ is the projective cover of $L(\lambda)$ for any $\lambda\in \Lambda_p$. We show that certain endofunctors of $B$-lfdmod induce
a $\mathfrak g$-module structure on $[K_0(B\text{-mod}^\Delta)]:=\mathbb C \otimes_{\mathbb Z} K_0(B\text{-mod}^\Delta) $, where  $\mfg$ is a Lie subalgebra of $\mathfrak{sl}_K$ (see Definition~\ref{liesub}), and $\mathfrak{sl}_K$ is either $\mathfrak{sl}_\infty$ or $\hat  {\mathfrak{sl}}_p$ depending on whether $p$  is zero or odd. Further, as $\mfg$-modules,  $[K_0(B\text{-mod}^\Delta)]$  is isomorphic to the   integrable highest weight $\mathfrak{sl}_K$-module $V(\varpi_{\frac{\delta_0-1}{2}})$  with the fundamental weight $\varpi_{\frac{\delta_0-1}{2}}$ as its highest weight (see Theorem~\ref{cateofg}). When $p=0$ and $\delta_0\in \mathbb Z$, $V(\varpi_{\frac{\delta_0-1}{2}})$ can be realized as $\bigwedge^{\infty}_d\mathbb W$, the $d$-sector  of semi-infinite wedge space $\bigwedge^{\infty}\mathbb W$ in \cite{BW}, where $\mathbb W$ is the restricted dual of the natural $\mathfrak{sl}_{\infty}$-module and $d=\frac{\delta_0}{2}-1$. In this case, $\{[\Delta(\lambda)]\mid \lambda\in \Lambda_p\}$ corresponds to the monomial basis of $\bigwedge_d^{\infty}\mathbb W$ and $\{[P(\lambda)]\mid \lambda\in \Lambda_p\}$ corresponds to the so-called quasi-canonical basis   of $\bigwedge_d^{\infty}\mathbb W$ (see Theorem~\ref{isoophi} and Theorem~\ref{amisd}),  and each entry of  transition matrix between monomial basis and quasi-canonical basis can be expressed via parabolic Kazhdan-Lusztig polynomial associated to the Weyl group of type $D$ with maximal parabolic subgroup of type $A$.

There is a collection of mutually orthogonal idempotents in $B$ which induce exact idempotent truncation functors from $B$-lfdmod to the representation category of Brauer algebras (see Proposition~\ref{xeijdiec}). So,  representations of  all  Brauer algebras can be reflected in  $ B\text{-lfdmod}$. Via the above idempotent truncation functors  and the BGG type reciprocity in $B$-lfdmod (see Corollary~\ref{xijixs}(2)), the combinatorics of the representation theory of  Brauer algebras given in \cite{CV, ES1} are reformulated in the  quasi-canonical basis   of $\bigwedge_d^{\infty}\mathbb W$ via categorification.
 In other words, our result about the categorification theorem  on $K_0(B\text{-mod}^\Delta)$ can also be considered as
the generalization of the well-known categorification theorem related to all symmetric groups (c.f. \cite[Theorem~9.5.1]{Kle}). In a sequel,  the current results will be generalized so as to deal with Brauer-type algebras such as cyclotomic Nazarov-Wenzl algebras and (cyclotomic) Birman-Murakami-Wenzl algebras.
 Moreover, parallel to the existence of a categorical action of a Kac-Moody 2-category on the additive Karoubi envelope of oriented Brauer category and oriented skein category \cite{Br},  there should be a strong categorical action on the additive Karoubi envelope of Brauer category and its quantum analog, i.e, a 2-representation of the 2-category associated to certain quantum symmetric pairs \cite{BSWW}.

The content of this paper is organized as follows. In section~2, we study the locally unital algebra $B$ associated to the Brauer category $\B(\delta_0)$. In section~3, we study induction and restriction functors. Via them, we define Kashiwara operators on $[K_0(B\text{-mod}^\Delta)]$ in section~4,  and  prove that as $\mfg$-modules $[K_0(B\text{-mod}^\Delta)]$ is isomorphic to $V(\varpi_{\frac{\delta_0-1}{2}})$. Finally, we establish  Kazhdan-Lusztig theory for  $[K_0(B\text{-mod}^\Delta)]$ when $p=0$ and $\delta_0\in \mathbb Z$ in section~5. This will show that our two basis can be considered as monomial and quasi-canonical basis of $[K_0(B\text{-mod}^\Delta)]$.

\subsection*{Acknowledgements}The authors are grateful to the referee for his/her suggestions
on the presentation of the paper, especially
on updating  the  precise definition of upper finite fully stratified category in \cite{BS} and  some important references\cite{BS, BSWW,CH,CZ}.

\section{Brauer category and its associated algebra}
 Throughout,   we assume that $K$ is an algebraically closed field of characteristic  $p\neq 2$.
 The aim of this section is to study the representation theory of the locally unital  algebra associated to the Brauer category  $\B(\delta_0)$  with defining parameter $\delta_0\in K$.
 Before we give their definitions,  we recall the notion of a strict $K$-linear monoidal category as follows.

A strict monoidal category   $(\C,\otimes, \mathbf 1 )$ is a category equipped with a bifunctor $\otimes: \C\times \C\rightarrow\C$ and a unit object $\mathbf 1$ such that
$(\ob  a\otimes \ob b)\otimes \ob c=\ob a\otimes (\ob b\otimes \ob c)$ and $\mathbf 1\otimes \ob a=\ob a=\ob a\otimes \mathbf 1$ for any objects $\ob a,\ob b,\ob c$, and
$(f\otimes g)\otimes h=f\otimes (g\otimes h)$ and $1_{\mathbf 1}\otimes f=f=f\otimes 1_{\mathbf 1}$, for all morphisms $f,g,h$, where $1_{\ob a}:\ob a\rightarrow \ob a $ is the identity morphism.
A $K$-linear  category $\C$  is a category such that each hom-set $\Hom_\C ( \ob a, \ob b)$ has the $K$-linear structure, and composition of morphisms is $K$-bilinear. A strict $K$-linear monoidal category $\C$  is a strict monoidal category and a $K$-linear category  such that tensor product of morphisms is $K$-bilinear.

We need string calculus in a strict $K$-linear monoidal category $\C$  (c.f. \cite[\S1.6]{VT}).
%Suppose that $\C$ is a strict $K$-linear monoidal category.
For any two objects $\ob a,\ob b$ in  $\C$,  $\ob a\ob b$ represents  $\ob a\otimes \ob b$.
 A  morphism $g:\ob a\to \ob b$ is drawn as
    $$\begin{tikzpicture}[baseline = 12pt,scale=0.5,color=\clr,inner sep=0pt, minimum width=11pt]
        \draw[-,thick] (0,0) to (0,2);
        \draw (0,1) node[circle,draw,thick,fill=white]{$g$};
        \draw (0,-0.2) node{$\ob a$};
        \draw (0, 2.3) node{$\ob b$};
    \end{tikzpicture}
     \quad\text{ or simply as }
    \begin{tikzpicture}[baseline = 12pt,scale=0.5,color=\clr,inner sep=0pt, minimum width=11pt]
        \draw[-,thick] (0,0) to (0,2);
        \draw (0,1) node[circle,draw,thick,fill=white]{$g$};
    \end{tikzpicture}$$
 if there is no confusion on the objects. Note that $\ob a$ is  at the bottom while $\ob b$ is at the top.
 The composition  (resp., tensor product) of two morphisms  is given by vertical stacking (resp., horizontal concatenation):
\begin{equation}\label{com1}
      g\circ h= \begin{tikzpicture}[baseline = 19pt,scale=0.5,color=\clr,inner sep=0pt, minimum width=11pt]
        \draw[-,thick] (0,0) to (0,3);
        \draw (0,2.2) node[circle,draw,thick,fill=white]{$g$};
        \draw (0,0.8) node[circle,draw,thick,fill=white]{$h$};
        % \draw (0,-0.2) node{$\ob c$};
        % \draw (0.3,1.5) node{$\ob b$};
        % \draw (0, 3.3) node{$\ob a$};
    \end{tikzpicture}
    ~,~ \ \ \ \ \ g\otimes h=\begin{tikzpicture}[baseline = 19pt,scale=0.5,color=\clr,inner sep=0pt, minimum width=11pt]
        \draw[-,thick] (0,0) to (0,3);
        \draw[-,thick] (2,0) to (2,3);
        \draw (0,1.5) node[circle,draw,thick,fill=white]{$g$};
        \draw (2,1.5) node[circle,draw,thick,fill=white]{$h$};
        % \draw (0,-0.2) node{$\ob c$};
        % \draw (0, 3.3) node{$\ob d$};
        % \draw (2,-0.2) node{$\ob a$};
        % \draw (2, 3.3) node{$\ob b$};
    \end{tikzpicture}~.
\end{equation}
By definition,
$ (f\otimes g)\circ(k\otimes h)=(f\circ k)\otimes (g\circ h)$
for any $f\in\Hom_{\C}(\ob b,\ob c)$, $g\in\Hom_{\C}(\ob n,\ob t )$, $k\in\Hom_{\C}(\ob a,\ob b)$, $h\in\Hom_{\C}(\ob m,\ob n)$.
This is known as the {\em interchange law}, and it implies that diagrams that are ``rectilinearly isotropic" define the same morphism in $\C$. So
both  $(f\otimes g)\circ(k\otimes h)$ and $(f\circ k)\otimes (g\circ h)$  may be  drawn as
 
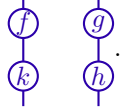
\begin{equation*}
    \begin{tikzpicture}[baseline = 19pt,scale=0.5,color=\clr,inner sep=0pt, minimum width=11pt]
        \draw[-,thick] (0,0) to (0,3);
        \draw[-,thick] (2,0) to (2,3);
        \draw (0,2.2) node[circle,draw,thick,fill=white]{$f$};
        \draw (2,2.2) node[circle,draw,thick,fill=white]{$g$};
        \draw (0,0.8) node[circle,draw,thick,fill=white]{$k$};
        \draw (2,0.8) node[circle,draw,thick,fill=white]{$h$};
    \end{tikzpicture}.
\end{equation*}

The affine Brauer category $\AB$ has been introduced in \cite{RS3}. It is the strict $K$-linear  monoidal category
  generated by a single object
 \begin{tikzpicture}[baseline = 10pt, scale=0.5, color=\clr]
                \draw[-,thick] (0,0.5)to[out=up,in=down](0,1.2);
    \end{tikzpicture}.   Therefore, the set of    objects in  $\AB$ is $\{
\begin{tikzpicture}[baseline = 10pt, scale=0.5, color=\clr]
                \draw[-,thick] (0,0.5)to[out=up,in=down](0,1.2);
    \end{tikzpicture}^{\otimes m} \mid m\in \mathbb N\}$.
     To simplify the notation, we denote
$ \begin{tikzpicture}[baseline = 10pt, scale=0.5, color=\clr]
                \draw[-,thick] (0,0.5)to[out=up,in=down](0,1.2);
                    \end{tikzpicture}^{\otimes m}$
by $\ob m$. So, the unit object in $\AB$  is  $\ob 0$, while the object
\begin{tikzpicture}[baseline = 10pt, scale=0.5, color=\clr] \draw[-,thick] (0,0.5)to[out=up,in=down](0,1.2);
    \end{tikzpicture}
    is  $\ob 1$. The morphisms in $\AB$ are  generated by   four  morphisms  $U=\lcup:\ob0\rightarrow \ob2$, $ A=\lcap:\ob2\rightarrow\ob0$, $ S=\swap: \ob2\rightarrow\ob2$   and  $X=\xdot:\ob1\rightarrow\ob1$,  subject to \eqref{B1}--\eqref{AB2} as follows:
     \begin{equation}\label{B1}
        \begin{tikzpicture}[baseline = 10pt, scale=0.5, color=\clr]
            \draw[-,thick] (0,0) to[out=up, in=down] (1,1);
            \draw[-,thick] (1,1) to[out=up, in=down] (0,2);
            \draw[-,thick] (1,0) to[out=up, in=down] (0,1);
            \draw[-,thick] (0,1) to[out=up, in=down] (1,2);
                    \end{tikzpicture}
        ~=~
        \begin{tikzpicture}[baseline = 10pt, scale=0.5, color=\clr]
            \draw[-,thick] (0,0) to (0,1);
            \draw[-,thick] (0,1) to (0,2);
            \draw[-,thick] (1,0) to (1,1);
            \draw[-,thick] (1,1) to (1,2);
        \end{tikzpicture}~
        ,\qquad
        \begin{tikzpicture}[baseline = 10pt, scale=0.5, color=\clr]
            \draw[-,thick] (0,0) to[out=up, in=down] (2,2);
            \draw[-,thick] (2,0) to[out=up, in=down] (0,2);
            \draw[-,thick] (1,0) to[out=up, in=down] (0,1) to[out=up, in=down] (1,2);
        \end{tikzpicture}
        ~=~
        \begin{tikzpicture}[baseline = 10pt, scale=0.5, color=\clr]
            \draw[-,thick] (0,0) to[out=up, in=down] (2,2);
            \draw[-,thick] (2,0) to[out=up, in=down] (0,2);
            \draw[-,thick] (1,0) to[out=up, in=down] (2,1) to[out=up, in=down] (1,2);
        \end{tikzpicture}~,
    \end{equation}
    \begin{equation}\label{B2}
        \begin{tikzpicture}[baseline = 10pt, scale=0.5, color=\clr]
            \draw[-,thick] (2,0) to[out=up, in=down] (2,1) to[out=up, in=right] (1.5,1.5) to[out=left,in=up] (1,1);
            \draw[-,thick] (1,1) to[out=down,in=right] (0.5,0.5) to[out=left,in=down] (0,1) to[out=up,in=down] (0,2);
        \end{tikzpicture}
        ~=~
        \begin{tikzpicture}[baseline = 10pt, scale=0.5, color=\clr]
            \draw[-,thick] (0,0) to (0,1);
            \draw[-,thick] (0,1) to (0,2);
        \end{tikzpicture}
        ~=~
        \begin{tikzpicture}[baseline = 10pt, scale=0.5, color=\clr]
            \draw[-,thick] (2,2) to[out=down, in=up] (2,1) to[out=down, in=right] (1.5,0.5) to[out=left,in=down] (1,1);
            \draw[-,thick] (1,1) to[out=up,in=right] (0.5,1.5) to[out=left,in=up] (0,1) to[out=down,in=up] (0,0);
        \end{tikzpicture}~,
    \end{equation}

\begin{equation}\label{B3}
    \begin{tikzpicture}[baseline = 5pt, scale=0.5, color=\clr]
        \draw[-,thick] (0,1) to[out=down,in=left] (0.5,0.35) to[out=right,in=down] (1,1);
    \end{tikzpicture}
    ~=~
    \begin{tikzpicture}[baseline = 5pt, scale=0.5, color=\clr]
        \draw[-,thick] (0,1) to[out=down,in=up] (1,0) to[out=down,in=right] (0.5,-0.5) to[out=left,in=down] (0,0) to[out=up,in=down] (1,1);
    \end{tikzpicture}
    ~,\qquad
    \begin{tikzpicture}[baseline = 5pt, scale=0.5, color=\clr]
        \draw[-,thick] (0,0) to[out=up,in=left] (0.5,0.65) to[out=right,in=up] (1,0);
    \end{tikzpicture}
    ~=~
    \begin{tikzpicture}[baseline = 5pt, scale=0.5, color=\clr]
        \draw[-,thick] (0,0) to[out=up,in=down] (1,1) to[out=up,in=right] (0.5,1.5) to[out=left,in=up] (0,1) to[out=down,in=up] (1,0);
    \end{tikzpicture}~,
    \end{equation}

    \begin{equation}\label{B4}
    \begin{tikzpicture}[baseline = 10pt, scale=0.5, color=\clr]
                \draw[-,thick] (1,0) to[out=up,in=right] (0.5,1.5) to[out=left,in=up] (0,1) to[out=down,in=up] (0,0);
        \draw[-,thick] (0.5,0) to[out=up,in=down] (1.3,1.5);
    \end{tikzpicture}~=~
    \begin{tikzpicture}[baseline = 10pt, scale=0.5, color=\clr]
         \draw[-,thick] (1,0) to[out=up,in=right] (0.5,1.5) to[out=left,in=up] (0.2,1) to[out=down,in=up] (0.2,0);
        \draw[-,thick] (0.7,0) to[out=up,in=down] (0,1.5);
    \end{tikzpicture}
    ~,\quad\quad\begin{tikzpicture}[baseline = 10pt, scale=0.5, color=\clr]
        \draw[-,thick]  (2,1.5) to[out=down, in=right] (1.5,0) to[out=left,in=down] (1,1.5);
        %\draw[-,thick] (1,1) to[out=up,in=right] (0.5,1.5) to[out=left,in=up] (0,1);
        \draw[-,thick] (0.7,0) to[out=up,in=down] (1.5,1.5);
    \end{tikzpicture}~=~
    \begin{tikzpicture}[baseline = 10pt, scale=0.5, color=\clr]
        \draw[-,thick]  (2,1.5) to[out=down, in=right] (1.5,0) to[out=left,in=down] (1,1.5);
        \draw[-,thick] (2.3,0)to[out=up,in=down](1.5,1.5);
    \end{tikzpicture}~,
\end{equation}

   \begin{equation}\label{AB1}
                \begin{tikzpicture}[baseline = 7.5pt, scale=0.5, color=\clr]
            \draw[-,thick] (0,0) to[out=up, in=down] (1,2);
            \draw[-,thick] (0,2) to[out=up, in=down] (0,2.2);
            \draw[-,thick] (1,0) to[out=up, in=down] (0,2);
            \draw[-,thick] (1,2) to[out=up, in=down] (1,2.2);
             \draw(0,1.9)\bdot;
        \end{tikzpicture}
        ~-~
        \begin{tikzpicture}[baseline = 7.5pt, scale=0.5, color=\clr]
            \draw[-,thick] (0,0) to[out=up, in=down] (1,2);\draw[-,thick] (0,0) to[out=up, in=down] (0,-0.2);
             \draw[-,thick] (1,0) to[out=up, in=down] (0,2);\draw[-,thick] (1,0) to[out=up, in=down] (1,-0.2);
                        \draw(1,0.1)\bdot;
        \end{tikzpicture}
        ~=~
       \begin{tikzpicture}[baseline = 10pt, scale=0.5, color=\clr]
          \draw[-,thick] (2,2) to[out=down,in=right] (1.5,1.5) to[out=left,in=down] (1,2);
            \draw[-,thick] (2,0) to[out=up, in=right] (1.5,0.5) to[out=left,in=up] (1,0);
        \end{tikzpicture}
        ~-~
        \begin{tikzpicture}[baseline = 7.5pt, scale=0.5, color=\clr]
            \draw[-,thick] (0,0) to[out=up, in=down] (0,2);
            \draw[-,thick] (1,0) to[out=up, in=down] (1,2);
                   \end{tikzpicture}~,
    \end{equation}

  \begin{equation}\label{AB2}
               \begin{tikzpicture}[baseline = 10pt, scale=0.5, color=\clr]
            \draw[-,thick] (2,0) to[out=up, in=down] (2,1) to[out=up, in=right] (1.5,1.5) to[out=left,in=up] (1,1);
            \draw[-,thick] (1,1) to[out=down,in=right] (0.5,0.5) to[out=left,in=down] (0,1) to[out=up,in=down] (0,2);
           % \draw[fill=\clr,thick] (0.85,1)--(1,0.85)--(1.15,1)--(0.85,1);
           \draw( 1,1) \bdot;
        \end{tikzpicture}
        ~=~
        -\begin{tikzpicture}[baseline = 10pt, scale=0.5, color=\clr]
            \draw[-,thick] (0,0) to (0,2);
            \draw(0,1) \bdot;
            %\draw[fill=\clr,thick] (-0.15,0.97)--(0,1.2)--(0.15,1)--(-0.1,0.97);
                    \end{tikzpicture}
        ~=~
        \begin{tikzpicture}[baseline = 10pt, scale=0.5, color=\clr]
            \draw[-,thick] (2,2) to[out=down, in=up] (2,1) to[out=down, in=right] (1.5,0.5) to[out=left,in=down] (1,1);
            \draw[-,thick] (1,1) to[out=up,in=right] (0.5,1.5) to[out=left,in=up] (0,1) to[out=down,in=up] (0,0);
            %\draw[fill=\clr,thick] (0.85,1)--(1,0.85)--(1.15,1)--(0.85,1);
            \draw( 1,1) \bdot;
        \end{tikzpicture}~.
    \end{equation}

Each endpoint  at both rows of the above diagrams represents  $\ob 1$. If there is no endpoint at a row, then the object at this row is  $\ob 0$.
 %For example, the object  at the bottom (resp., top) row of $U$ is $\ob 0$ (resp., $\ob 2$).
For any $m>0$, the identity morphism $1_{\ob m}$  is drawn as the object itself. For example,  $\begin{tikzpicture}[baseline = 10pt, scale=0.5, color=\clr]
                \draw[-,thick] (0,0.5)to[out=up,in=down](0,1.2);\draw[-,thick] (0.5,0.5)to[out=up,in=down](0.5,1.2);
    \end{tikzpicture}$ represents  $\text{1}_{\ob 2}$.

    In \cite{RS3}, we have proved that  the Brauer category $\B$ is isomorphic to the subcategory of $\AB$, whose objects are the same as those for $\AB$ and whose morphisms are generated by   $A, U$ and $S$ satisfying \eqref{B1}--\eqref{B4}. So, one can identify $\B$ as a subcategory of $\AB$.
    For any $k\in\mathbb N$, define % $\Delta_k$ the crossing-free bubble with $k$ $\bullet$'s on the left boundary. Then
 \begin{equation}\label{defofdelta}
\Delta_k:= \begin{tikzpicture}[baseline = 5pt, scale=0.5, color=\clr]
        \draw[-,thick] (0.6,1) to (0.5,1) to[out=left,in=up] (0,0.5)
                        to[out=down,in=left] (0.5,0)
                        to[out=right,in=down] (1,0.5)
                        to[out=up,in=right] (0.5,1);
        \draw (0,0.5) \bdot;
        \draw (-0.4,0.5) node{\footnotesize{$k$}};
    \end{tikzpicture}= \lcap\circ (\xdot~ \begin{tikzpicture}[baseline = 5pt, scale=0.5, color=\clr]
     \draw[-,thick] (0,0.15) to (0,1.15); \end{tikzpicture} ~)^k\circ \lcup\in \End_{\AB}(\ob 0).
\end{equation}
  Thanks to our previous results on basis of morphism spaces in   $\AB$~\cite[Theorem~B]{RS3}, $\AB$  can be viewed as a $K[\Delta_0, \Delta_2, \Delta_4, \ldots]$-linear category with $\Delta_kg:=\Delta_k\otimes g$ for any $g\in\Hom_{\AB}(\ob m,\ob s)$.
   Later on, we always assume that $ \delta=\{\delta_i\in K\mid  i\in \mathbb N\}$.
Suppose that  $\Delta_j$'s are specialized at scalars $\delta_j$'s. It is proved in \cite{RS3} that
\begin{equation}\label{admomega}
2\delta_k=-\delta_{k-1}+\sum_{j=1}^{k}(-1)^{j-1}\delta_{j-1}\delta_{k-j}, \text{ for } k=1,3,\ldots.
    \end{equation}
So, $\delta$  is  admissible in the sense of \cite[Definition~2.10]{AMR}.
In \cite{RS3}, we define $\AB(\delta)$ to be the $K$-linear category (not a monoidal category) obtained from $\AB$ by specializing $\Delta_{2j}$ at $\delta_{2j}$, i.e.,
 $$\AB(\delta):=K\otimes _{K[\Delta_0, \Delta_2,\Delta_4,\ldots]}\AB$$ where $K$ is the  $K[\Delta_2, \Delta_4,\ldots]$-module on which $\Delta_{2j}$ acts  as $\delta_{2j}$ for any  $j\geq 1$.

Let  $f(t)=\prod_{1\leq i\leq a}(t-u_i)\in K[t]$, where $a\in \mathbb N\setminus\{0\}$. If    $\delta$ satisfies \eqref{admomega}, we define the \emph{cyclotomic Brauer category} $\CB(\delta)^f$ to be the quotient category $\AB/I$, where  $I$ is  the right tensor ideal of $\AB$ generated by $f(\xdot)$ and $\Delta_j-\delta_j$, $j\in\mathbb N$. Note that $\CB(\delta)^f$ is also not monoidal.
Let  $u$ be an indeterminate and $\mathbf u=\{u_1,u_2,\ldots,u_a\}$. Following  \cite[Definition~3.6,~Lemma~3.8]{AMR},   $\delta$ is called
$\mathbf u$-admissible    if
\begin{equation}\label{admc}
\sum_{i\geq 0} \frac{\delta_i}{ u^{i}}+u-\frac{1}{2}=(u-\frac{1}{2}(-1)^a)\prod_{i=1}^a \frac{u+u_i}{u-u_i}.
\end{equation}
In~\cite[Theorem~C]{RS3}, we proved that any morphism space of $\CB(\delta)^f$ is free (over an integral domain $K$ containing $1/2$)  with maximal rank if and only if $\delta$ is $\mathbf u$-admissible.
In the remaining part of the paper, we keep the  Assumption~\ref{ASS} as follows.
\begin{Assumption}\label{ASS}Fix $\delta_0\in  K$. Let $\delta=\{\delta_i\mid i\in\mathbb N\}$ such that $\delta_i=\delta_0(\frac{\delta_0-1}{2})^i$ for $i\geq1$.
Let   $f(t)=t-u_1$, where $u_1=\frac{\delta_0-1}{2}$.
\end{Assumption}

 Under the Assumption~\ref{ASS},   $\delta$ is $\mathbf u$-admissible in the sense of \eqref{admc}.
  Let $\B(\delta_0)$ be obtained from $\B$ by adding an additional condition $\Delta_0=\delta_0$. We have the functor $G:\B(\delta_0)\hookrightarrow \AB(\delta)\twoheadrightarrow \CB^f(\delta)$,  which is the composition of natural  embedding functor and canonical quotient  functor. Using basis theorem of morphism spaces in $\B(\delta_0)$ and
 $\CB^f(\delta)$\cite[Theorems ~A,C]{RS3}, we have that $G$ is an isomorphism of categories between   $\B(\delta_0)$ and  $\CB^f(\delta)$.      So, $\B(\delta_0)$  is both a subcategory and a quotient category  of $\AB(\delta)$.

 We are going to study the representation theory of  $\B(\delta_0)$. This is motivated by \cite{Re, Br} on locally unital algebras associated with  oriented Brauer category and oriented skein category.

 Suppose $m, s \in \mathbb N$. An $(m,s)$-\textsf{Brauer diagram} is a string diagram obtained by tensor product and composition of
  $A, U, S$  and the identity  morphism \begin{tikzpicture}[baseline = 10pt, scale=0.5, color=\clr]
                \draw[-,thick] (0,0.5)to[out=up,in=down](0,1.5);\end{tikzpicture}
     such that there are $m$ (resp., $s$) endpoints  on the bottom (resp., top) row of the  resulting diagram.  For example, the following is a $(7,3) $-Brauer diagram:
 \begin{equation}\label{example of m ,s}
\begin{tikzpicture}[baseline = 25pt, scale=0.35, color=\clr]
        %\draw[-,thick] (2,0) to[out=up,in=down] (0,5);
        \draw[-,thick] (5,0) to[out=up,in=down] (8,5);
         \draw[-,thick] (10,0) to[out=up,in=down] (6,5);
         %\draw[-,thick] (0,2)to (11,2);\draw[-,thick] (0,2.7)to (11,2.7);\draw[-,thick] (0,4)to (11,4);\draw[-,thick] (0,1)to (11,1);\draw[-,thick] (0,5.2)to (11,5.2);\draw[-,thick] (0,-0.2)to (11,-0.2);
          \draw[-,thick] (2.6,5) to (2.5,5) to[out=left,in=up] (1.5,4)
                        to[out=down,in=left] (2.5,3)
                        to[out=right,in=down] (3.5,4)
                        to[out=up,in=right] (2.5,5);
                         %\draw (-1,0.5) \bdot;\draw (-0.4,0.5) node{\footnotesize{$k$}};\draw (6.2,4) \bdot;\draw (6.8,4) node{\footnotesize{$h$}};
                        % \draw (3.5,4) \bdot;\draw (3.9,4) node{\footnotesize{$l$}};\draw (6.2,2.2) \bdot;
        \draw[-,thick] (-1,0) to[out=up,in=left] (1,1.5) to[out=right,in=up] (3,0);
        %\draw[-,thick] (2,5) to[out=down,in=left] (3,4) to[out=right,in=down] (4,5);
         \draw[-,thick] (2,0) to[out=up,in=left] (4,2.5) to[out=right,in=up] (6,0);
          \draw[-,thick] (12.6,4.5) to (12.5,4.5) to[out=left,in=up] (10.5,2.5)
                        to[out=down,in=left] (12.5,0.5)
                        to[out=right,in=down] (14.5,2.5)
                        to[out=up,in=right] (12.5,4.5);
                        \draw[-,thick] (11,0) to [out=up,in=down](12,1.7);   \draw[-,thick] (12,1.7)to [out=up,in=down](11,5);
           \end{tikzpicture}~, \end{equation}
where  loops are  called                           bubbles.
       In  $\B(\delta_0)$, each bubble is reduced to $\delta_0$. So,  we consider  $\mathbb{B}_{m,s}$,  the set of  all $(m,s)$-Brauer  diagrams without bubbles.
 We label  the  endpoints at the bottom (resp., top) row  of $d\in \mathbb B_{m,s}$ by  $1,2,\ldots, m$ (resp., $m+1, m+2,\ldots,m +s$) from left to right.
Two endpoints   are  paired if they    are connected  by a strand. Since each endpoint of $d$ is uniquely  connected with another  one, it  gives a partition of $\{1,2,\ldots, m+s\}$ into disjoint union of  pairs.
For example, the $(7,3)$-Brauer diagram in \eqref{example of m ,s} (forgetting the bubbles there)  gives the partition $\{\{1,3\}, \{2,5\}, \{4,9\},\{6,8\},\{7,10\}\}$.

A  strand of $d\in \mathbb{B}_{m,s}$ connecting the pairs on different rows (resp., the same row) is called a vertical  (resp.,  horizontal) strand. Moreover, A horizontal strand connecting the pairs on the top (resp., bottom) row is also called  a cup (resp., cap).
 We say that  $d, d'\in  \mathbb{B}_{m,s}$    are   \emph{equivalent} and write $d\sim d'$ if they  give  the same partition of $\{1,2,\ldots, m+s\}$ into disjoint union of pairs. So, the cardinality  of the set  $  \mathbb B_{m,s}/\sim$ is $(m+s-1)!!$ (resp., 0) if $m+s$ is even (resp., otherwise).
 Thanks to the following result, each equivalence class can be identified with any  Brauer diagram in it.

 \begin{Theorem}\label{basisofb}\cite{RS3} For any $m, s\in \mathbb N$, \begin{enumerate}\item $d=d'$ in $\B(\delta_0)$ if $d, d'\in \mathbb{B}_{m,s}$ and $d\sim d'$,
 \item  $\Hom_{\B(\delta_0)}(\ob m,\ob s)$  has $K$-basis   given by  $\mathbb{B}_{m,s}/\sim$.\end{enumerate}
 \end{Theorem}

 A locally unital algebra is a non-unital associative $K$-algebra $  H$ containing a distinguished collection
of mutually orthogonal idempotents $\{1_m\mid m\in J\}$, for some index set $J$, such that
$ H=\bigoplus_{m,n\in J}1_{m}H1_n$.

\textsf{Unless otherwise stated, all modules in this paper are assumed to be left modules}. Following \cite{Re},  we consider  $H$-mod, the  category of locally unital $H$-modules. In other words,  any object $V\in  H\text{-mod}$ is of form
$$V=\bigoplus_{m\in J}1_{  m}V,$$
 where nonzero spaces $1_{ m}V$ are called \textsf{weight spaces} of $V$. If  $\text{dim}1_mV<\infty$ for all $m\in J$, then
 $V$ is called locally finite dimensional. Let $H\text{-lfdmod}$ be the subcategory of   $H\text{-mod}$ consisting of all locally finite dimensional modules. Denote by
   $H\text{-pmod}$ the subcategory of  $H\text{-mod}$ consisting of all finitely generated projective modules.

A locally unital algebra $H$ is called locally finite dimensional if  each $1_{m}H1_n$ is of finite dimensional for all $m,n\in J$.
Thanks to Theorem~\ref{basisofb},  $B$ is a  locally unital and  locally finite dimensional algebra, where
\begin{equation}\label{Bal}
B:= \bigoplus_{m,n\in \mathbb N}\Hom_{\B(\delta_0)}(\ob m,\ob n).
\end{equation}
 The set  $\{1_{\ob m}\in\B(\delta_0) \mid m\in \mathbb N\}$  serves as the system of  mutually orthogonal idempotents of $B$ such that  $1_{\ob n} B1_{\ob m}=\Hom_{\B(\delta_0)}(\ob m,\ob n)$  for all $m,n\in\mathbb N$.
The multiplication  on $B$ is defined as   $$g h=\begin{cases} g\circ h, &\text{if $m=t$,}\\ 0, &\text{otherwise,}\\
\end{cases} \\ \quad \text{for any  $(g, h)\in \Hom_{\B(\delta_0)}(\ob m,\ob n) \times \Hom_{\B(\delta_0)}(\ob s,\ob t)$.}$$
Suppose $V\in B$-mod and $L$ is a  simple $B$-module.  Motivated by \cite{Re}, we  define  $$[V:L]=\text{sup}\sharp\{i\mid V_{i+1}/V_i\cong L\},$$ the supremum being taken over all filtrations by submodules $0=V_0\subset \cdots \subset V_n=V$. The integer
$[V:L]$ is called the composition multiplicity of $L$ in $V$, although  $V$ may not have a composition series.

 The results in Proposition~\ref{usdeks}  are special cases of those   for any
 locally unital and locally finite dimensional algebra in  \cite[Section~II.2]{Re}(see also \cite[Section~2.2]{BRUNDAN}).

\begin{Prop}\label{usdeks} Suppose $B$ is the locally unital and locally finite dimensional algebra in \eqref{Bal}.
\begin{enumerate}
\item [(1)] Every finitely generated $B$-module is locally finite dimensional.
\item [(2)] If $L$ is a simple $B$-module, then $\End_{B}(L)=K$.
\item [(3)] For any finitely generated $B$-module $ V$ and any locally finite dimensional $B$-module $W $, $\dim\Hom_{B}(V,W)<\infty$.
\item [(4)] Every finitely generated $B$-module can be  decomposed  into a direct sum of  finitely many indecomposable modules and this decomposition is unique (up to a permutation).
\item [(5)] Every finitely generated  $B$-module has a projective cover.
\item [(6)] If $V\in B$-{\rm lfdmod}, then $[V: L]=\dim \Hom_{B}(P,V)<\infty$  for any simple module $L$, where $P$
is the projective cover of $L$.
\end{enumerate}
\end{Prop}
By Proposition~\ref{usdeks}(1), $B$-pmod is also a  subcategory of $B$-lfdmod.
Motivated by \cite{Re}, we consider locally unital subalgebras of $B$ as follows:
\begin{enumerate}\item
 $\mathbb K:=\bigoplus_{m\in\mathbb N}  K 1_{\ob m}$,

    \item $B^0_m$, $m\in\mathbb N$: the $K$-span of all Brauer diagrams  on which there are  $m$ strings and there are  neither  cups nor caps,
\item $B^0:=\bigoplus_{m\in \mathbb N}B^0_m$, the $K$-span of all Brauer diagrams on which there are neither  cups nor caps,
 \item  $B^+$ : the $K$-span of all Brauer diagrams on which there are neither cups and nor crossings among vertical strands,   \item $B^-$:  the $K$-span of all Brauer  diagrams on which there are  neither  caps  nor crossings among vertical strands,
\item $B^\sharp$: the $K$-span of all Brauer diagrams on which there are no cups,
\item $B^\flat$: the $K$-span of all Brauer diagrams on which there are  no caps.\end{enumerate}

\begin{Lemma}\label{useifcats}For any $m,n\in\mathbb N$,
\begin{enumerate}
\item [(1)] $1_{\ob m }B^-1_{\ob n}\neq 0\text{ only if }$ $m=n+2k$, for some $k\in\mathbb N$,
\item [(2)] $1_{\ob m }B^+1_{\ob n}\neq 0\text{ only if }$ $n=m+2k$, for some $k\in\mathbb N$.
\end{enumerate}
\end{Lemma}
\begin{proof} It follows from  the definition of  $(m,s)$-Brauer diagrams that $1_{\ob m} B1_{\ob s}\neq 0$ only if $m+s$ is even. Now, (1)-(2) follow immediately from  Theorem~\ref{basisofb} and definitions of $B^+$ and $B^-$.\end{proof}
Thanks to Theorem~\ref{basisofb}, there are $\mathbb N$-gradings on both $B^\sharp$  and   $B^\flat$ such that
  $B^\sharp=\bigoplus _{d\in\mathbb N}B^\sharp[d]$ and $B^\flat=\bigoplus _{d\in\mathbb N}B^\flat[-d]$,
  where \begin{itemize}\item  $B^\sharp[d]$ is the subspace of $B^\sharp$ spanned by all   Brauer diagrams on which there are $d$  caps,
  \item  $B^\flat[-d]$ is the subspace of $B^\flat$ spanned by all  Brauer diagrams on which there are $d$  cups. \end{itemize}
  Moreover, it induces  a  grading on $B^+$ and $B^-$   and  $B^\flat[0]=B^\sharp[0]=B^0$. The following   triangular decomposition of $B$ was first given  in \cite[\S~8.3]{CZ}.
%If $V$ is a right $K$-module and $W$ is a left $K$-module, then
% $V\otimes_{K} W\cong \bigoplus_{m\in\mathbb N}V1_{\ob m}\otimes _K 1_{\ob m} W$ as $K$-spaces.
\begin{Lemma}\label{triangularde}
As $K$-spaces, we have \begin{enumerate}\item [(1)] $B\cong B^-\otimes _{\mathbb K}B^0\otimes _{\mathbb K}B^+$, \item  [(2)] $B^\flat\cong  B^-\otimes _{\mathbb K}B^0 $, \item [(3)] $ B^\sharp\cong B^0\otimes _{\mathbb K}B^+$.\end{enumerate}
\end{Lemma}
\begin{proof}
By Theorem~\ref{basisofb}, $B^-$ (resp., $B^+$) has a basis given by  the set of all Brauer diagrams on which there are  no caps (resp., cups)  and moreover, there are no crossings among vertical strings.
Similarly, $B^0$ has a basis given by the set of  Brauer diagrams on which there are neither  caps nor cups. Thanks to Theorem~\ref{basisofb} again, the product map gives a $K$-isomorphism between  $B^-\otimes _{\mathbb K}B^0\otimes _{\mathbb K}B^+$ and $B$. One can verify the second and third isomorphisms, similarly.
\end{proof}

Thanks to Theorem~\ref{basisofb} and \eqref{B1},  $B^0_m\cong K\mathfrak S_m$, the group algebra of the symmetric group on $m$
letters. Thus,
\begin{equation}\label{isomosym}
B^0\cong \bigoplus_{m\in\mathbb N}K \mathfrak S_m.
\end{equation}
Since $B^0$ is isomorphic to the  quotient of $B^\sharp$ by the two-sided ideal $\bigoplus_{d\geq1}B^\sharp[d]$, any  $B^0$-module can be considered as a $B^\sharp$-module. It gives  an exact functor $B^0$-mod $ \rightarrow B^\sharp$-mod. Composing this functor with $B\otimes _{B^\sharp}?$ yields the standardization functor
\begin{equation}\Delta: B^0\text{-mod}\rightarrow B\text{-mod}
\end{equation}
in the sense of \cite{Re}. See also \cite{CZ} for similar functors defined for certain truncations of $B$.
Similarly, there is an exact functor $B^0$-mod $ \rightarrow B^\flat$-mod. Composing this functor with
$\bigoplus_{m\in\mathbb N} \Hom_{B^\flat}(B1_{\ob m},?)$ yields  the costandardization functor
\begin{equation}
\nabla: B^0\text{-mod}\rightarrow B\text{-mod}
\end{equation}
in the sense of \cite{Re}, where the action of $a\in 1_{\ob m} B 1_{\ob n}$ on $f\in \Hom_{B^\flat}(B1_{\ob m'},V)$ is zero unless $n=m'$, in which case $a f \in \Hom_{B^\flat}(B1_{\ob m},V)$ such that $(af)(b)=f(ba)$, for all $b\in B 1_{\ob m}$.
\begin{Lemma}\label{ddheuhd}Keep the notations  above.
\begin{enumerate}
\item[(1)] Both  $\Delta$ and $\nabla$ are exact functors.
\item[(2)] Restricting  $\Delta$ and $\nabla$   to $B^0$-{\rm lfdmod} yields  two functors $\Delta: B^0\text{-\rm lfdmod}\rightarrow B\text{-\rm lfdmod}$ and $\nabla: B^0\text{-\rm lfdmod}\rightarrow B\text{-\rm lfdmod}$.
\end{enumerate}
\end{Lemma}
\begin{proof}
Thanks to Lemma~\ref{triangularde}, we have
\begin{equation}\label{siobia}
\begin{aligned}
1_{\ob m}B&\cong 1_{\ob m} B^-\otimes_{\mathbb K}B^\sharp\cong \bigoplus_{n\in \mathbb N}(1_{\ob n}B^\sharp)^{\oplus \dim (1_{\ob m}B^-1_{\ob n})}~ \text{ as right $B^\sharp$-modules,} \\
B1_{\ob m}&\cong B^\flat \otimes _{\mathbb K} B^+1_{\ob m}\cong \bigoplus_{n\in \mathbb N}(B^\flat 1_{\ob n})^{\oplus \dim(1_{\ob n}B^+1_{\ob m})}~ \text{ as left $B^\flat$-modules}.
\end{aligned}
\end{equation}
By Lemma~\ref{useifcats}, there are only finitely many $n$ such that   $\dim 1_{\ob m}B^-1_{\ob n}\neq 0$ and  $\dim 1_{\ob n}B^+1_{\ob m}\neq 0$ for any fixed $m$.
So, $1_{\ob m}B$ (resp., $B 1_{\ob m}$) is finitely generated and projective as right $B^\sharp$-modules (resp., as left $B^\flat$-modules). Therefore, (1) is proven.

Suppose that $V\in  B^0\text{-lfdmod}$. For any $m\in \mathbb N$, we have
\begin{equation}\label{kswskwskwiskwks}
\begin{aligned}1_{\ob m}\Delta(V)&=1_{\ob m}B\otimes_{B^\sharp}V=1_{\ob m}B^-\otimes_{B^\sharp} V =\oplus _{n\in \mathbb N}1_{\ob m}B^-1_{\ob n}\otimes_{B^\sharp} 1_{\ob n}V,\\
1_{\ob m}\nabla(V)&= \Hom_{B^\flat}(B1_{\ob m},V)=\oplus_{n\in\mathbb N}\Hom_{B^\flat}( B^\flat1_{\ob n},V)^{\dim 1_{\ob n}B^+1_{\ob m}} =\oplus_{n\in\mathbb N}(1_{\ob n}V)^{\dim 1_{\ob n}B^+1_{\ob m}}.
\end{aligned} \end{equation}
Thanks to Lemma~\ref{useifcats}, we see that $\dim 1_{\ob m}\Delta(V)<\infty$ and  $\dim 1_{\ob m}\nabla(V)<\infty$. Hence $\Delta(V)\in B\text{-lfdmod}$
  and $\nabla(V)\in B\text{-lfdmod}$.
\end{proof}
Now, we briefly  recall the representation theory of the symmetric group $\mathfrak S_m$.
Let $K\mathfrak S_m$-mod be the category of finite dimensional $\mathfrak S_m$-modules.
For $p>0$, a partition $\lambda$ is called $p$-regular if $\sharp\{j\mid \lambda_j=k\}<p$ for any $k>0$.
Let  $\Lambda_p(m)$ be   the set of all $p$-regular partitions  of  $m$. If $p=0$, then $\Lambda_p(m)$ is $\Lambda(m)$,
 the set of all  partitions of $m$. We always identify  a partition with its Young diagram.
The content of any box $x$ in $\lambda$ is $c(x)\equiv j-i (\text{mod } p)$ if $x$ is in the $i$th row and $j$th column of $\lambda$.
 Define
  $\Lambda=\bigcup_{m\in\mathbb N}\Lambda(m)$ and  $\Lambda_p=\bigcup_{m\in\mathbb N}\Lambda_p(m)$. When  $p=0$, $\Lambda_0=\Lambda$.

For any  $\lambda\in \Lambda(m)$, there is an $S(\lambda)$, called the \textsf{classical Specht module}\cite{GJ}.
If $\lambda\in \Lambda_p(m)$, then the head $D(\lambda)$  of  $S(\lambda)$ is  simple.
It is well-known that $\{D(\lambda)\mid \lambda\in \Lambda_p(m)\}$ is a complete set of pair-wise   non-isomorphic simple  $K \mathfrak S_m$-modules. For any $\lambda\in\Lambda(m)$, there is a $Y(\lambda)$, called  the \textsf{Young module}. If $\lambda\in\Lambda_p(m)$, then
   $  Y(\lambda)  $ is both the   projective cover and the injective hull of $D(\lambda)$. When $p=0$, $K\mathfrak S_m$ is semisimple and  $Y(\lambda)=S(\lambda)=D(\lambda)$ for all $\lambda\in \Lambda(m)$.

Suppose  $V\in B$-mod.
Following \cite{Re}, the  subspace $1_{\ob m}V$  is called the {\em {shortest word space}} $V^s$ of $V$ if $m$ is the minimal integer such that $1_{\ob m}V\neq0$.
\begin{Prop}\label{hhxdexxx}
For any $V\in B$-mod,  $V^s$ is  a $B^\sharp$-module on which  $ \bigoplus _{d\in \mathbb N\setminus 0}B^\sharp[d]$ acts trivially.
\end{Prop}
\begin{proof}
%The first statement is obvious.
Suppose $V^s=1_{\ob m}V$. For any $d>0$,  $B^\sharp[d]V^s=B^\sharp[d]1_{\ob m}V=1_{\ob{m-2d}}B^\sharp[d]1_{\ob m}V\subseteq 1_{\ob{m-2d}}V$.  Since $m$ is the minimal integer such that $1_{\ob m}V\neq0$, we have $B^\sharp[d]V^s=0$ for $d>0$. So $ \bigoplus _{d\in \mathbb N\setminus 0}B^\sharp[d]$ acts trivially on $V^s$.
Obviously, $1_{\ob m}V$ is a $1_{\ob m} B 1_{\ob m}$-module and hence a $K \mathfrak S_m$-module.   So,  $V^s$ is    both a $B^0$-module and a  $B^\sharp$-module.
\end{proof} Motivated by \cite{Re},  we define
  \begin{equation}\label{stahhss}
  \begin{aligned}
  &\Delta(\lambda):= \Delta(Y(\lambda)), ~\nabla(\lambda):= \nabla(Y(\lambda)), ~\lambda\in\Lambda_p,\\
 &{\bar\Delta}(\lambda):= \Delta(D(\lambda)),\  {\bar\nabla}(\lambda):= \nabla(D(\lambda)), \ \lambda\in\Lambda_p,\\
  &\tilde \Delta(\lambda):=\Delta(S(\lambda)),  ~\lambda\in\Lambda.
  \end{aligned}
  \end{equation}
  $\Delta(\lambda)$ (resp., $\nabla(\lambda)$) is called the standard (resp., costandard) module. Similarly,  $\bar \Delta(\lambda)$
(resp., $\bar\nabla(\lambda)$) is called  the proper standard (resp., proper costandard) module. By Lemma~\ref{ddheuhd}(2), all  modules in \eqref{stahhss}
are objects in $B$-lfdmod.
  If $p=0$, then $Y(\lambda)=S(\lambda)=D(\lambda)$ and hence
    $\Delta(\lambda)=\tilde \Delta(\lambda)={\bar\Delta}(\lambda)$ for all $\lambda\in\Lambda$.

\begin{Theorem} \label{hd1} Keep the notations above.
 \begin{itemize}
 \item[(1)] $\bar \Delta(\lambda)$ has a unique maximal proper submodule, for all $\lambda\in\Lambda_p$.
  \item [(2)] $L(\lambda)^s\cong D(\lambda)$ as $B^0$-modules, where $L(\lambda)$ is the simple head of $\bar\Delta(\lambda)$, for all $\lambda\in\Lambda_p$.
\item [(3)]  $\{L(\lambda)\mid \lambda\in\Lambda_p\}$
is a complete set of pair-wise  non-isomorphic simple $B$-modules.\end{itemize}
\end{Theorem}
\begin{proof}Suppose $\lambda\in \Lambda_p(m)$.
 By Lemma~\ref{useifcats} and \eqref{kswskwskwiskwks},
$1_{\ob n}\bar \Delta(\lambda)=1_{\ob n}B^-1_{\ob m}\otimes _{B^\sharp}D(\lambda)\neq 0$ only if
$n=m+2d$ for some $d\in\mathbb N$,
 and hence $\bar \Delta(\lambda)^s=1_{\ob m}\otimes _{\mathbb K}D(\lambda)$. This shows that  $\bar \Delta(\lambda)$ is generated by any non-zero element in $1_{\ob m}\bar\Delta(\lambda)$ and any proper submodule of it  lies in  the subspace $\bigoplus_{d>0}1_{\ob {m+2d}}\bar\Delta(\lambda)$, forcing  the direct sum of all proper submodules is proper.  Then  $\text{rad}\bar\Delta(\lambda)$ is the unique maximal submodule of $\bar\Delta(\lambda)$ such that
$
\text{rad}\bar\Delta(\lambda)\subset\bigoplus_{d>0}1_{\ob {m+2d}}\bar\Delta(\lambda)$.
Moreover, since $\bar \Delta(\lambda)\twoheadrightarrow L(\lambda)$, we have  $D(\lambda)\cong 1_m \bar \Delta(\lambda)\twoheadrightarrow 1_m L(\lambda)$. So,  $ L(\lambda)^s\cong D(\lambda)$ as $B^0$-modules. This proves (1) and (2).

Suppose $L$ is a simple $B$-module such that  $ L^s=1_{\ob m}L$. Then $1_{\ob m}L$ is a $B^0$-module on which $B^0$ acts  via  $K\mathfrak S_m$. Let  $L^0$ be a  simple $B^0$-submodule of $1_{\ob m}L$ such that $L^0\cong D(\lambda)$ for some $\lambda\in\Lambda_p(m)$. By Proposition~\ref{hhxdexxx} and  Frobenius reciprocity,
there is   a  surjective  $B$-module homomorphism $\bar\Delta(\lambda)\twoheadrightarrow  L$. Thanks to (1)-(2),    $L\cong L(\lambda)$.
Now, suppose  $L(\lambda)\cong L(\mu)$. Then  $L(\lambda)^s\cong L(\mu)^s$ as  $B^0$-modules and hence   $D(\lambda)\cong D(\mu)$. Thus   $\lambda=\mu$.
This completes the proof of (3).
\end{proof}

If $\lambda=\emptyset$, then $D(\emptyset)=S(\emptyset)=Y(\emptyset)=K$, forcing
$\tilde\Delta(\emptyset)=\bar\Delta (\emptyset)=\Delta(\emptyset)$. We will freely use this fact later on.

\begin{Lemma}\label{omehga0}
If $\delta_0=0$, then $L(\emptyset)\cong K$,  the trivial module of $B$.
\end{Lemma}
\begin{proof}
Let $N=\bigoplus_{d>0}1_{\ob {2d}}\bar\Delta(\emptyset)= \bigoplus _{d>0}1_{\ob {2d}}B^-1_{\ob 0}\otimes _{\mathbb K}K$.
Note that $1_{\ob m}B1_{\ob n}\neq 0$ only if $m+n$ is even.
So, $1_{\ob m}B1_{\ob n}N\neq 0$ only if $m=2k$ and $n=2d$ for some $d\in \mathbb N\setminus \{0\}$ and $k\in\mathbb N$.
We have $1_{\ob {2k}}B1_{\ob {2d}}N\subset  1_{\ob {2k}}B^-1_{\ob 0}\otimes _{\mathbb K}K$.
If $k>0$, then  $1_{\ob {2k}}B1_{\ob {2d}}N\subset N$.
If $k=0$, then $1_{\ob {0}}B1_{\ob {2d}}N\subset 1_{\ob {0}}B1_{\ob {2d}} B^-1_{\ob 0}\otimes  K=0$, provided that $\delta_0=0$ (i.e. any bubble is reduced to zero in $\B(0)$).
 So,  $N$ is a proper $B$-submodule, forcing  $N=\text{rad} \bar\Delta(\emptyset)$. Therefore,
  $L(\emptyset)=1_{\ob 0} \bar \Delta(\emptyset)=  K$.
\end{proof}

By the definition of $\B(\delta_0)$, there is an obvious isomorphism $\tau: \B(\delta_0)\rightarrow \B(\delta_0)^{\text{op}}$ such that $\tau(\lcap)=\lcup$, $\tau(\lcup)=\lcap$ and it fixes  objects and the generator $\swap$.
Then for any Brauer diagram $b$,     $\tau(b) $ is  obtained by flipping  $b$ vertically.
Moreover,  $\tau$ induces  a $K$-linear  anti-involution on $B$ which fixes $\mathbb K$, preserves $B^0$, and
 swaps $B^+$ with $B^-$. For any $V\in  B$-lfdmod, let
 \begin{equation}
 V^*=\bigoplus_{m\in\mathbb N}\Hom_{K}(1_{\ob m}V,K).
 \end{equation}
 Then  $V^*$ is a $B$-module such that for any $a\in  B$, and $f\in V^*$,
$(af)(v)=f(\tau(a)v)$.
This induces an exact contravariant duality functor $*$  on $B$-lfdmod.
Similarly, we have the contravariant duality functor on $B^0$-lfdmod.  Thanks to  \eqref{isomosym}, this is just the usual duality functor  for symmetric groups.
The counterpart of the Lemma~\ref{key1}  and its proof for oriented skein category can be found in \cite[Lemma~5.5]{Br}.
\begin{Lemma}\label{key1}
 $*\circ \Delta\cong \nabla\circ *$ on the category of finite dimensional $B^0$-modules.
\end{Lemma}
\begin{proof} For any finite dimensional $B^0$-module $V$, there is  a $B$-homomorphism
$\phi: \Delta(V)^*\rightarrow \nabla(V^*)$, sending  $ \alpha$ to $ \bar \alpha  $
where $\bar\alpha(f)(v)=\alpha(\tau(f)\otimes v)$ for all $v\in V$ and  $f\in B1_{\ob m}$.
The  inverse of $\phi$ is the homomorphism
$ \nabla(V^*)\rightarrow\Delta(V)^*$, sending $ \beta$ to $ \tilde \beta$,
where $\tilde \beta(f\otimes v)=\beta(\tau(f))(v)$.
\end{proof}

\begin{Lemma}\label{isodual} Suppose $\lambda,\mu\in\Lambda_p$. Then   \begin{itemize} \item [(1)]
$\Delta(\lambda)^*\cong \nabla(\lambda)$, $\bar\Delta(\lambda)^*\cong \bar \nabla(\lambda)$ and $L(\lambda)\cong L(\lambda)^*$, \item [(2)] $[\bar\Delta(\lambda):L(\mu)]=[\bar\nabla(\lambda):L(\mu)]$.\end{itemize}
\end{Lemma}
\begin{proof}It is well-known that  $D(\lambda)^*\cong D(\lambda)$ and $Y(\lambda)^*\cong Y(\lambda)$ as $B^0$-modules (c.f.\cite[Theorem~11.2.1]{Kle}). Thanks to Lemma~\ref{key1}, we have isomorphisms in (1) except the last one.
Since $ L(\lambda)^s=D(\lambda)\cong D(\lambda)^*=  (L(\lambda)^*)^s$, we have the last isomorphism. Finally, since $\ast$ is an exact functor,  (2) follows from (1). \end{proof}

 A $B$-module $V$  has a finite $\Delta$-flag if it has a finite filtration such that its  sections are  isomorphic to $\Delta(\lambda)$ for various  $\lambda\in \Lambda_p$. Let $B$-mod$^{\Delta}$ be full subcategory of $B$-mod
consisting of all modules with a finite $\Delta$-flag. By Lemma~\ref{ddheuhd}(2),  $\Delta(\lambda)\in B\text{-lfdmod}$ for any $\lambda\in\Lambda_p$. So, $B$-mod$^{\Delta}$ is a subcategory
of $ B\text{-lfdmod}$.
%Thanks to Lemma~\ref{triangularde},   the restriction of $ \Delta(\lambda)$ to $B^\flat$ is isomorphic to
%$B^\flat\otimes _{B^0} Y(\lambda)$.
The following result  can
be proved by arguments similar to those  in \cite[Lemma~5.6]{Br}.

\begin{Lemma}\label{deltp}$\{\Delta(\lambda)\mid \lambda\in\Lambda_p\}$  is a complete set of all non-isomorphic  indecomposable projective  $B^\flat$-modules.
\end{Lemma}
\begin{proof}Thanks to Lemma~\ref{triangularde},
we have  $ \Delta(\lambda)\cong B^\flat\otimes _{B^0} Y(\lambda)$ as $B^\flat$-modules. Since $\{Y(\lambda)\mid \lambda\in \Lambda_p\}$  is a complete set of  non-isomorphic indecomposable projective
$B^0$-modules, any  $ \Delta(\lambda)$ is projective and each indecomposable projective $B^\flat$-module
is isomorphic to a summand of some $\Delta(\lambda)$. Suppose  $\lambda\in \Lambda_p(m)$. Then
  \begin{equation}
 \begin{aligned}
 \End_{B^\flat}(B^{\flat}\otimes _{B^0}Y(\lambda))&\cong\Hom_{B^0}(Y(\lambda),B^\flat\otimes _{B^0}Y(\lambda) )\\
 &\cong\Hom_{B^0}(1_{\ob m}Y(\lambda),1_{\ob m}B^\flat\otimes _{B^0}Y(\lambda) )\\
 &\overset{Lem.~\ref{useifcats}}\cong \End_{B^0}(Y(\lambda))\cong K.
 \end{aligned}
 \end{equation}
So, $\Delta(\lambda)$ is indecomposable and the result follows.
\end{proof}
%\begin{Lemma}\label{extis}
%For any $\lambda, \mu \in \Lambda_p$ and $d\in \mathbb N$,
%$\dim\text{Ext}^d_{B}(\Delta(\lambda), \bar\nabla(\mu))=\delta_{d,0}\delta_{\lambda,\mu}$. \end{Lemma}
%\begin{proof} The result can be proved by arguments in the proof of \cite[Lemma~V.2.1]{Re}. The reason is that his arguments depend on the %triangular decomposition of the algebra associated to the oriented Brauer category. In our case, our algebra $B$ associated to $\B(\delta_0)$ has %the  triangular decomposition, too. \end{proof}

%The following is a direct corollary of the above lemma.
\begin{Prop} \label{ext}
Suppose  $V\in B$-{\rm mod}$^{\Delta}$ and $\lambda,\mu\in\Lambda_p$.
\begin{enumerate}
\item[(1)] For any  $d\geq0$,
$\dim\text{Ext}^d_{B}(\Delta(\lambda), \bar\nabla(\mu))=\delta_{d,0}\delta_{\lambda,\mu}$.
\item [(2)]$(V:\Delta(\lambda))=\dim \Hom_B(V, \bar \nabla(\lambda))$,  where $(V:\Delta(\lambda))$ is  the   multiplicity
of $\Delta(\lambda)$ in a $\Delta$-flag of $V$. In particular, it   is  independent of the choice of a $\Delta$-flag of $V$.
\end{enumerate}
\end{Prop}
\begin{proof} (1) can be proven by
arguments in the proof of \cite[Lemma~V.2.1]{Re}.  We remark that  Reynolds'  arguments depend on the triangular decomposition of the algebra associated to the oriented Brauer category. In our case, our algebra $B$ associated to $\B(\delta_0)$ has the  triangular decomposition, too (see Lemma~\ref{triangularde}). (2) follows from standard arguments by using (1).
\end{proof}

\begin{Lemma}\label{deltafildire}
A $B$-module $V$ is in $B$-{\rm mod}$^{\Delta}$ if and only if it is finitely generated and projective as a $B^\flat$-module. So, $B$-{\rm mod}$^{\Delta}$ is Karoubian.
\end{Lemma}
\begin{proof}Mimicking  arguments in the proof of \cite[Lemma~5.8]{Br} yields the result, immediately. The only difference is that we have to use    Lemma~\ref{deltp} to replace its counterpart in \cite{Br}.
\end{proof}

The proof of following result follows the arguments in \cite[Proposition~ V.2.4]{Re}.
\begin{Prop}\label{dektss} For any $\lambda\in\Lambda_p(m)$, let $P(\lambda)$ be the projective cover of $L(\lambda)$. Then
 \begin{itemize} \item[(1)] $P(\lambda)\in B$-{\rm mod}$^{\Delta}$ such that $(P(\lambda): \Delta(\mu))\neq 0$ only if  $\mu\in\Lambda_p(m-2d)$, for some    $d\geq 0$,
  \item [(2)] $(P(\lambda):\Delta(\lambda))=1$ and $\Delta(\lambda)$ appears as the top section of $P(\lambda)$,
   \item [(3)] $(P(\lambda): \Delta(\mu))\neq 0$ and $\mu\neq \lambda$  only if  $\mu\in\Lambda_p(m-2d)$ for some    $d> 0$.
  \end{itemize} \end{Prop}

\begin{proof}
Thanks to Lemma~\ref{useifcats}, we have
 $B^\sharp [k]\otimes _{B^0}Y(\lambda)\neq0$ only if $0\leq k\leq \lceil m/2\rceil $. In this case,   $1_{\ob n}B^\sharp [k]\otimes _{B^0}Y(\lambda)\neq0$ only if $n=m-2k$.
So $B^0$ acts on the  finite dimensional $B^0$-module  $ B^\sharp [k]\otimes _{B^0}Y(\lambda)$ via $K\mathfrak S_{m-2k}$.
 As a   $B^0$-module, $B^\sharp \otimes _{B^0}Y(\lambda)$ has  a  filtration
$$0\subset B^\sharp _{\lceil m/2\rceil}\otimes _{B^0}Y(\lambda)\subset \cdots \subset B^\sharp _{1}\otimes _{B^0}Y(\lambda) \subset B^\sharp _{0}\otimes _{B^0}Y(\lambda)=B^\sharp \otimes _{B^0}Y(\lambda), $$
where $B^\sharp _i=\oplus_{k\geq i}B^\sharp [k]$. Note that $B^\sharp B^\sharp[k]\subset B^\sharp[l]$ with $l\geq k$. So, the above filtration is a  $B^\sharp$-module filtration and each section
$B^\sharp [k]\otimes _{B^0}Y(\lambda)$ is a $B^\sharp $-module such that $ \oplus _{d\in \mathbb N\setminus\{0\}}B^\sharp[d]$ acts trivially.
This shows that  $\hat P(\lambda)$   has a finite filtration with sections $ \Delta(B^\sharp [k]\otimes _{B^0}Y(\lambda) )$, $0\leq k\leq \lceil m/2\rceil$, where $\hat P(\lambda):=B\otimes _{B^0}Y(\lambda)\cong B\otimes _{B^\sharp}B^\sharp \otimes _{B^0}Y(\lambda)$.  The top section is $\Delta(B^\sharp [0]\otimes _{B^0}Y(\lambda))=\Delta(Y(\lambda))=\Delta(\lambda)$.

Suppose $k>0$.
Since $Y(\lambda)$ is a projective $B^0$-module,  $ B^\sharp [k]\otimes _{B^0}Y(\lambda)=B^\sharp [k]1_{\ob m}\otimes _{B^0}Y(\lambda)$,    which is   a direct summand of  $B^\sharp [k]1_{\ob m}$. By Lemma~\ref{triangularde},
 $$B^\sharp [k]1_{\ob m}\cong\bigoplus _{n\in\mathbb N}B^01_{\ob n}\otimes _K 1_{\ob n} B^+[k]1_{\ob m}\overset{Lem.~\ref{useifcats}} \cong (B^01_{\ob {m-2k}})^{\dim 1_{\ob {m-2k}}B^+[k]1_{\ob m}}= (K \mathfrak S_{m-2k})^{\dim 1_{\ob {m-2k}}B^+[k]1_{\ob m}},$$
 where $B^+[k]$ is the subspace of $B^+$ spanned by all Brauer diagrams with $k$ caps.
So, $B^\sharp [k]1_{\ob m}$ is a projective $B^0$-module and hence $ B^\sharp [k]\otimes _{B^0}Y(\lambda)$ is a finite dimensional projective $B^0$-module. In fact, it is  a projective  $K \mathfrak S_{m-2k}$-module, which is  a direct sum of some $Y(\mu)$'s with $\mu\in\Lambda_p(m-2k)$. This proves (1)-(3) if $P(\lambda)$ is replaced by $\hat P(\lambda)$.
Since $Y(\lambda)$ is a direct summand of $B^0$,  $\hat P(\lambda)$ is a projective $B$-module.
We have $$\begin{aligned} \dim \Hom_{B}(\hat P(\lambda),L(\lambda))& =\dim\Hom_{B^0}(Y(\lambda), \oplus_{m\in\mathbb N}\Hom_B(B1_{\ob m},L(\lambda)))\\ & =\dim\Hom_{B^0}(Y(\lambda),L(\lambda)) =\dim\Hom_{B^0}(Y(\lambda),1_{\ob m}L(\lambda))\\ & =\dim\Hom_{B^0}(Y(\lambda),D(\lambda))=1.\end{aligned} $$
So, it gives an epimorphism from $\hat P(\lambda)$ to $L(\lambda)$. Since $\hat P(\lambda)$ is  projective, and  $P(\lambda)$ is the projective cover of $L(\lambda)$,  $P(\lambda)$ has to be a direct summand of $\hat P(\lambda)$.
Thanks to Lemma~\ref{deltafildire}, we have (1)-(3) for  $P(\lambda)$.
\end{proof}
%\begin{rem} In  Theorem~\ref{fullystratified}, we prove that $B$-lfdmod is an upper finite fully stratified category. Consequently,
%Proposition~\ref{dektss} can be obtained directly (without using Lemma~\ref{deltafildire}) from the fact that $B$-lfdmod is an upper finite fully stratified category (see details i and Remark~\ref{highestweight}) by using the general arguments for upper finite fully stratified category in \cite{BS}. Their condition $(\widehat{P\Delta})$ in  \cite[Remark~3.8]{BS} for $B$-lfdmod to be upper finite fully stratified is exactly  to check the property of  $\hat P(\lambda)$ in the proof Proposition~\ref{dektss}. However, this is automatically satisfied by \cite[Corollary~5.31]{BS}  and the fact that  $B$
 %has an upper finite triangular decomposition $(B^-,B^0,B^+)$(see Lemmas~\ref{useifcats}--\ref{triangularde}) in the sense of \cite[Definition~5.24]{BS}. Moreover, Proposition ~\ref{ext} follows immediately since it is a standard fact in an upper finite fully stratified category\cite{BS}. We thank the referee for providing this alternative proof for Propositions ~\ref{ext},~\ref{dektss}. \end{rem}

Let $K_0(B\text{-pmod})$ be the split   Grothendieck group of $B\text{-pmod}$. Then
$K_0(B\text{-pmod})$ is the free abelian group on the isomorphism classes of finitely generated projective $B$-modules modulo
$[M]-[N]-[L]$ whenever $M\cong N\oplus L$. Let $K_0(B\text{-mod}^\Delta)$ be the Grothendieck group of $B\text{-mod}^\Delta$.

\begin{Cor}\label{xijixs}Keep the notations above.
\begin{enumerate}
\item [(1)] There is an embedding $B\text{-\rm pmod} \hookrightarrow B\text{-\rm mod}^\Delta$ and   $K_0(B\text{-\rm pmod})\cong K_0(B\text{-\rm mod}^\Delta)$.
\item [(2)] $(P(\lambda):\Delta(\mu))=[\bar\Delta(\mu):L(\lambda)]$ for any $\lambda,\mu\in\Lambda_p$.
\item [(3)] Suppose that $\mu\in\Lambda(m)$. Then $[\tilde \Delta(\mu): L(\lambda)]=\sum_{\nu\in\Lambda_p(m)} [S(\mu): D(\nu)](P(\lambda):\Delta(\nu))$. In particular, $[\tilde \Delta(\mu): L(\lambda)]\neq 0$ only if $ \mu\in \Lambda(m)$ and $\lambda\in \Lambda_p(m+2d)$ for some $d,m\in\mathbb N$.
\end{enumerate}
\end{Cor}
\begin{proof}  Thanks to Proposition~\ref{ext}(1),
 $\{[\Delta(\lambda)]\mid \lambda\in\Lambda_p\}$ is a $\mathbb Z$-basis of $K_0(B\text{-mod}^\Delta)$.
By Proposition~\ref{dektss},   the transition matrix between $\{[P(\lambda)]\mid \lambda\in\Lambda_p \}$ and $\{[\Delta(\lambda)]\mid \lambda\in\Lambda_p\}$ is invertible. This proves (1).

Thanks to Lemma~\ref{isodual}(2) and Proposition~\ref{usdeks}(6),  we have  $$[\bar\Delta(\mu):L(\lambda)]=[\bar\nabla(\mu):L(\lambda)]=\dim\Hom_B(P(\lambda),\bar\nabla(\mu)).$$
On the other hand, $P(\lambda)\in B$-mod$^{\Delta}$ (see Proposition~\ref{dektss}(1)).
 Using the  Hom functor $\Hom_B(?, \bar \nabla(\mu))$    and   Proposition~\ref{ext}(1), we have $\dim\Hom_B(P(\lambda),\bar\nabla(\mu))= (P(\lambda):\Delta(\mu))$, proving (2).

 Since $\Delta$ is exact,   $\tilde\Delta(\mu)$ has a $\bar\Delta$-flag such that each section is of form $\bar\Delta(\nu)$ for some $\nu\in\Lambda_p(m)$ and the multiplicity of
 $\bar\Delta(\nu)$ (denoted by $(\tilde\Delta(\mu): \bar\Delta(\nu))$) in this flag is $ [S(\mu): D(\nu)]$. So,

$$\begin{aligned}~\quad
 [\tilde \Delta(\mu): L(\lambda)] &=\sum_{\nu\in\Lambda_p(m)} (\tilde\Delta(\mu): \bar\Delta(\nu))[\bar\Delta(\nu):L(\lambda)]\\
&=\sum_{\nu\in\Lambda_p(m)} [S(\mu): D(\nu)][\bar\Delta(\nu):L(\lambda)]\\
&\overset{(2)}=\sum_{\nu\in\Lambda_p(m)} [S(\mu): D(\nu)](P(\lambda):\Delta(\nu)).\end{aligned}$$
The remaining statement of (3) follows from Proposition~\ref{dektss}(1).
 \end{proof}

%\begin{Cor}
%
%\end{Cor}
%\begin{proof} Thanks to Proposition~\ref{dektss}, we have a short exact sequence $ 0\rightarrow N\rightarrow P(\lambda)\rightarrow \Delta(\lambda)\rightarrow 0$ where $N$ has a $\Delta$-flag with  sections $\Delta(\mu)$, $\mu\neq \lambda$. Applying  $\Hom_B(?, \bar \nabla(\mu))$  and using \eqref{ext}  yields the isomorphism
%$\Hom_B(P(\lambda), \bar \nabla(\mu))\cong \Hom_B (N, \bar \nabla(\mu))$.  Comparing dimensions and using Lemma~\ref{isodual}(2), we have   $(P(\lambda):\Delta(\mu))=[\bar\nabla(\mu):L(\lambda)]=[\bar\Delta(\mu):L(\lambda)]$.
%Finally, by Theorem~\ref{hd1}, we have
% $(P(\lambda):\Delta(\lambda))=[\bar\Delta(\lambda):L(\lambda)]=1$.\end{proof}

\begin{Cor}\label{ijxxexeu} Suppose  $\lambda\in\Lambda_p(m)$. Then   $P(\lambda)$ has a finite $\tilde\Delta$-flag  such that $\mu\in \coprod_{d=0}^{\lceil m/2\rceil} \Lambda(m-2d)$ if $\tilde\Delta(\mu)$ appears as a section. Further, the  multiplicity of $\tilde\Delta(\mu)$ in this flag is equal to $[\tilde\Delta(\mu):L(\lambda)]$.
\end{Cor}
\begin{proof} Since $Y(\lambda)$ has a Specht filtration, $\Delta(\lambda)$ has a finite $\tilde\Delta$-flag. Thanks to
Proposition ~\ref{dektss},  $P(\lambda)$ has a finite $\tilde\Delta$-flag, too. Denote by $(P(\lambda):\tilde \Delta(\mu))$
the multiplicity of $\tilde\Delta(\mu)$ in this flag of $P(\lambda)$. Then

 \begin{equation}\label{uxhuexuexeu}
 \begin{aligned} (P(\lambda):\tilde \Delta(\mu))& =\sum_{\nu } (P(\lambda): \Delta(\nu)) (\Delta(\nu):\tilde \Delta(\mu))
 =\sum_{\nu } [\bar \Delta(\nu): L(\lambda)](Y(\nu): S(\mu))\\ & =\sum_\nu [\bar \Delta(\nu): L(\lambda)] [S(\mu): D(\nu)]
 = \sum_\nu [\bar \Delta(\nu): L(\lambda)](\tilde \Delta(\mu): \bar \Delta(\nu))\\ & =[\tilde \Delta(\mu): L(\lambda)],\\
 \end{aligned}\end{equation}
 where $\nu$ is over the set  $ \coprod_{d=0}^{\lceil m/2\rceil} \Lambda_p(m-2d)$ (see Proposition~\ref{dektss}(1)).
 If $[S(\mu): D(\nu)]\neq0$ in \eqref{uxhuexuexeu} for some  $\nu\in\Lambda_p(n)$, then $\mu\in \Lambda(n)$. So, $\mu\in \coprod_{d=0}^{\lceil m/2\rceil} \Lambda(m-2d)$.
  \end{proof}

\section{Induction and Restriction functors  }

The aim of this section is to study the induction and restriction functors on $B^0$-mod and  $B$-mod.
The methods we use in this section are parallel to   those  for the oriented Brauer category and oriented skein  category in \cite{Br,Re}. However,  the resulting combinatorics is completely different.

 For all admissible $i,m$, we  define
\begin{equation}\label{jdijdijxs}
\begin{aligned}
%U_i 1_{\ob m}&=1_{\ob m-2} U_i=\begin{tikzpicture}[baseline = 25pt, scale=0.35, color=\clr]
%\draw[-,thick](-2,1.1)to[out= down,in=up](-2,3.9);
%         \draw(-1.5,1.1) node{$ \cdots$}; \draw(-1.5,3.9) node{$ \cdots$};
%         \draw[-,thick](-0.5,1.1)to[out= down,in=up](-0.5,3.9);
%        \draw[-,thick] (0,4) to[out=down,in=left] (0.5,3.5) to[out=right,in=down] (1,4);
%         %\draw[-,thick] (0,1) to[out=up,in=left] (0.5,1.5) to[out=right,in=up] (1,1);
%         \draw(0,4.5)node{\tiny$i$}; \draw(1.3,4.5)node{\tiny$i+1$};
%         \draw[-,thick](2,1.1)to[out= down,in=up](2,3.9);
%         \draw(3.5,1.1) node{$ \cdots$}; \draw(3.5,3.9) node{$ \cdots$};
%         \draw[-,thick](4.5,1.1)to[out= down,in=up](4.5,3.9);
%           \end{tikzpicture},~~   A_i 1_{\ob m}=1_{\ob m-2}A_i=\begin{tikzpicture}[baseline = 25pt, scale=0.35, color=\clr]
%\draw[-,thick](-2,1.1)to[out= down,in=up](-2,3.9);
%         \draw(-1.5,1.1) node{$ \cdots$}; \draw(-1.5,3.9) node{$ \cdots$};
%         \draw[-,thick](-0.5,1.1)to[out= down,in=up](-0.5,3.9);
%        %\draw[-,thick] (0,4) to[out=down,in=left] (0.5,3.5) to[out=right,in=down] (1,4);
%         \draw[-,thick] (0,1) to[out=up,in=left] (0.5,1.5) to[out=right,in=up] (1,1);
%         \draw(0,0.4)node{\tiny$i$}; \draw(1.3,0.4)node{\tiny$i+1$};
%         \draw[-,thick](2,1.1)to[out= down,in=up](2,3.9);
%         \draw(3.5,1.1) node{$ \cdots$}; \draw(3.5,3.9) node{$ \cdots$};
%         \draw[-,thick](4.5,1.1)to[out= down,in=up](4.5,3.9);           \end{tikzpicture}\\
      X_i 1_{\ob m}&=1_{\ob m}X_i= \begin{tikzpicture}[baseline = 25pt, scale=0.35, color=\clr]
       \draw[-,thick](0,1.1)to[out= down,in=up](0,3.9);
       \draw(0.5,1.1) node{$ \cdots$}; \draw(0.5,3.9) node{$ \cdots$};
       \draw[-,thick](1.8,1.1)to (1.8,3.9);
       %\draw[-,thick] (2,1) to[out=up, in=down] (3,4);
       \draw(3,4.5)node{\tiny$i$}; %\draw(3.2,4.5)node{\tiny$i+1$};
        \draw[-,thick] (3,1) to[out=up, in=down] (3,4);
         \draw[-,thick](4,1.1)to[out= down,in=up](4,3.9);\draw (3,2.5) \bdot;
         \draw(4.5,1.1) node{$ \cdots$}; \draw(4.5,3.9) node{$ \cdots$};
         \draw[-,thick](5.5,1.1)to[out= down,in=up](5.5,3.9);
           \end{tikzpicture},~~
           S_i 1_{\ob m}=1_{\ob m}S_i= \begin{tikzpicture}[baseline = 25pt, scale=0.35, color=\clr]
       \draw[-,thick](0,1.1)to[out= down,in=up](0,3.9);
       \draw(0.5,1.1) node{$ \cdots$}; \draw(0.5,3.9) node{$ \cdots$};
       \draw[-,thick](1.5,1.1)to[out= down,in=up](1.5,3.9);
       \draw[-,thick] (2,1) to[out=up, in=down] (3,4);
       \draw(2,4.5)node{\tiny$i$}; \draw(3.2,4.5)node{\tiny$i+1$};
        \draw[-,thick] (3,1) to[out=up, in=down] (2,4);
         \draw[-,thick](4,1.1)to[out= down,in=up](4,3.9);
         \draw(4.5,1.1) node{$ \cdots$}; \draw(4.5,3.9) node{$ \cdots$};
         \draw[-,thick](5.5,1.1)to[out= down,in=up](5.5,3.9);
           \end{tikzpicture}.
           \end{aligned}
           \end{equation}
           where  $i$ labels the $i$th string from the left to the right. The object at the bottom of each diagram is $\ob m$.
The symbols $X_i,S_i$ are ambiguous until an object is specified by multiplying some well-defined morphism in $\AB(\delta)$.
If there is no confusion on the objects, we also write $X_i$ and $S_i$ for $X_i1_{\ob m}$ and $S_i1_{\ob m}$, respectively.

Fix a positive integer $m$. The group algebra of the symmetric group $K \mathfrak S_{m}$ can be identified with the subalgebra of $\End_{\B}(\ob m)$ generated by $S_i1_{\ob m}$, $1\leq i\leq m-1$. For the simplification of notation, we denote $S_i1_{\ob m}$ by $S_i$, the usual transposition of $(i, i+1)$ under the above identification. Let
\begin{equation}
\text{ind}^m_{m-1}: K \mathfrak S_{m-1}\text{-mod}\rightarrow K \mathfrak S_{m}\text{-mod}, \quad \text{res}^m_{m-1}: K \mathfrak S_{m}\text{-mod}\rightarrow K \mathfrak S_{m-1}\text{-mod}
\end{equation}
be the induction and restriction functors with respect to the natural embedding $K \mathfrak S_{m-1}\hookrightarrow K \mathfrak S_{m} $, which sends $S_i\mapsto S_i$ for all $1\le i\le m-2$.
So,
\begin{equation}\label{reinfu}
\text{ind}^m_{m-1}= K \mathfrak S_{m}\otimes _{K \mathfrak S_{m-1} }? \text{ and $\text{res}^m_{m-1} = K \mathfrak S_{m}\otimes _{K \mathfrak S_{m} }?$,}\end{equation}
where $K \mathfrak S_m$ is viewed as  both a  $(K \mathfrak S_{m},  K \mathfrak S_{m-1})$-bimodule and a $(K \mathfrak S_{m-1},  K \mathfrak S_{m})$-bimodule, respectively. It is  well-known that these two functors are biadjoint pairs. Further, the Jucys-Murphy element
\begin{equation}\label{Jusss}
L_m=\sum_{1\leq i<m}(i,m)\end{equation}
 of $K\mathfrak S_m$ centralizes the subalgebra $K \mathfrak S_{m-1}$. Multiplying  $L_m$ on the right (resp., left) of $K \mathfrak S_{m}$  gives an endomorphism of the $(K \mathfrak S_{m},  K \mathfrak S_{m-1})$-bimodule (resp., $(K \mathfrak S_{m-1},  K \mathfrak S_{m})$-bimodule) $K \mathfrak S_m$. Let $ (K \mathfrak S_m)_i$ be the generalized $i$-eigenspace  of $L_m$ on  $K \mathfrak S_m$.
For any $i\in K$, define
$i\text{-ind}^m_{m-1}= (K\mathfrak S_{m})_i\otimes_{K \mathfrak S_{m-1} }?$ and $i\text{-res}^m_{m-1} = (K \mathfrak S_{m})_i\otimes _{K \mathfrak S_{m} }?$.
It is well-known that
\begin{equation}\label{bra1} \text{ind}^m_{m-1}=\bigoplus_{i\in\mathbb Z1_K}i\text{-ind}^m_{m-1}, \quad \text{res}^m_{m-1}=\bigoplus_{i\in\mathbb Z1_K}i\text{-res}^m_{m-1}. \end{equation}

%Thanks to   $f\otimes g$ can be represented by  $fg$  for any two Brauer diagrams $f$ and $g$.
\begin{Lemma}\label{actionbimo} Let $E$ and $F$ be  the endo-functors  of $B^0$-mod such that  $E:=B^0_L\otimes _{B^0} ?$ and   $F:= B^0_R\otimes _{B^0} ?$, where both  $ B^0_L$ and $B^0_R$ are    $(B^0,B^0)$-bimodule such that
\begin{itemize}\item [(1)] $1_{\ob m} B^0_L1_{\ob n}=1_{\ob {m+1}}B^01_{\ob n} $, and $1_{\ob m} B^0_R1_{\ob n}=1_{\ob {m}}B^01_{\ob {n+1}} $,
  \item[(2)]   $B^0$ acts on   the right (resp., left )  of $B^0_L$  (resp.,  $B^0_R$) via    the usual multiplication,
  \item [(3)]
$a\cdot f= (a ~\begin{tikzpicture}[baseline = 10pt, scale=0.5, color=\clr]
                \draw[-,thick] (0,0.5)to[out=up,in=down](0,1.2);
    \end{tikzpicture})\circ f$,   and $g\cdot a=g\circ (a~\begin{tikzpicture}[baseline = 10pt, scale=0.5, color=\clr]
                \draw[-,thick] (0,0.5)to[out=up,in=down](0,1.2);
    \end{tikzpicture})$, for all   $ f\in B^0_L, g\in  B^0_R$ and  $a\in B^0 $, where  $ab$ is  $a\otimes b$ for any two  Brauer diagrams $a$ and $b$ (see \eqref{com1}).
\end{itemize} Then
 $E=\bigoplus_{m\in\mathbb N} \text{res}^{m+1}_{m}$ and $ F=\bigoplus_{m\in\mathbb N} \text{ind}^{m+1}_{m} $.
    \end{Lemma}
\begin{proof}The results follow immediately from \eqref{reinfu} and the definitions of $E$ and $F$.\end{proof}

%\begin{Lemma} As   endo-functors of $B^0$-Mod, we have
 %$E=\bigoplus_{m\in\mathbb N} \text{res}^{m+1}_{m}$ and $ F=\bigoplus_{m\in\mathbb N} \text{ind}^{m+1}_{m} $.
% \end{Lemma}

Under the identification between $K\mathfrak S_m$ and $1_{\ob m}B^01_{\ob m}$, the transposition $(k,l)$ is the morphism in $ \End_{\B(\delta_0)}(\ob m)$
such that its diagram is the crossing of the $k$th and $l$th strands, for all $1\leq k,l\leq m$. For example, $S_i1_\ob{m}=(i,i+1)$.

 \begin{Lemma}\label{bim123} Suppose $\delta_0\in K$. Define $X^0_L: B^0_L\rightarrow B^0_L$ and $X^0_R: B^0_R\rightarrow B^0_R$ such that
 \begin{itemize}\item [(1)]  the restriction of $X^0_L$ to  $ 1_{\ob {m+1}}B^0$ is given by  the left
multiplication  of  $\frac{\delta_0-1}{2}1_{\ob {m+1}}+ L_{m+1} \in $ $ \End_{\B(\delta_0)}(\ob m+1)$,
 \item [(2)]  the restriction of $X^0_R$ to   $B^0 1_{\ob {m+1}}$ is given by  the right
multiplication  of  $\frac{\delta_0-1}{2}1_{\ob {m+1}}+ L_{\ob m+1}$.
\end{itemize}
Then both   $X^0_L$ and $X^0_R$ are $(B^0,B^0)$-bimodule endomorphisms.
 \end{Lemma}
 \begin{proof}
Thanks to
 $ (a~\begin{tikzpicture}[baseline = 10pt, scale=0.5, color=\clr]
                \draw[-,thick] (0,0.5)to[out=up,in=down](0,1.2);
    \end{tikzpicture}) \circ L_{m+1}=L_{m+1}\circ (a~\begin{tikzpicture}[baseline = 10pt, scale=0.5, color=\clr]
                \draw[-,thick] (0,0.5)to[out=up,in=down](0,1.2);
    \end{tikzpicture})$,  $\forall a\in 1_{\ob m}B^01_{\ob m}$, we have the results, immediately.
 \end{proof}

We define  $B^0_{L, i}$ (resp.,  $B^0_{R, i}$)  to be the generalized $i$-eigenspace of $X^0_L$ (resp., $X^0_R$) on $B^0_L$ (resp., $B^0_R$).  The following result follows from Lemma~\ref{actionbimo}, \eqref{bra1} and  well-known results on symmetric groups.

\begin{Lemma}\label{filteration}
For $i\in  K$, let $E_i=B^0_{L, i}\otimes _{B^0} ?$ and   $F_i:= B^0_{R, i}\otimes _{B^0} ?$.  Then
 \begin{enumerate}\item [(1)] $E_{\frac{\delta_0-1}{2}+i}=\bigoplus _{m\in\mathbb N}i\text{-res}^{m+1}_m $ and $F_{\frac{\delta_0-1}{2}+i}=\bigoplus _{m\in\mathbb N}i\text{-ind}^{m+1}_m $,
  \item [(2)] $E=\bigoplus _{i\in I_0}E_i$ and  $F=\bigoplus _{i\in I_0}F_i$, where $I_0=\frac{\delta_0-1}{2}+\mathbb Z1_  K \subset   K$,
\item [(3)] $E_i S(\lambda)$ has a multiplicity-free filtration with sections $S(\mu)$, where  $ \mu\in\Lambda(m-1)$
is obtained by removing  a box of content $i-\frac{\delta_0-1}{2}$ from  $ \lambda\in\Lambda(m)$,
\item [(4)] $F_i S(\lambda)$ has a multiplicity-free filtration with sections $S(\mu)$, where   $ \mu\in\Lambda(m+1)$
is obtained by adding   a box of content $i-\frac{\delta_0-1}{2}$ to $\lambda\in\Lambda(m)$.
\end{enumerate}
\end{Lemma}

\begin{Lemma}\label{bimoudisomah} Let  $B_L$ and $B_R$ be  $(B,B)$-bimodules   such that \begin{itemize}\item  $1_{\ob m} B_L1_{\ob n}=1_{\ob {m+1}}B1_{\ob n} $ and $1_{\ob m} B_R1_{\ob n}=1_{\ob {m}}B1_{\ob {n+1}} $, for all $m,n\in\mathbb N$,
\item
$a\cdot f= (a ~\begin{tikzpicture}[baseline = 10pt, scale=0.5, color=\clr]
                \draw[-,thick] (0,0.5)to[out=up,in=down](0,1.2);
    \end{tikzpicture})\circ f$,   and $g\cdot a=g\circ (a~\begin{tikzpicture}[baseline = 10pt, scale=0.5, color=\clr]
                \draw[-,thick] (0,0.5)to[out=up,in=down](0,1.2);
    \end{tikzpicture})$, for all   $ f\in B_L, g\in  B_R$ and  $a\in B $.\end{itemize}
Then there is a $(B,B)$-isomorphism  $B_L\cong B_R$.
\end{Lemma}

\begin{proof} The required isomorphism is $\phi: B_L\rightarrow B_R$ such that  $\phi (f)= (1_{\ob m} ~ \lcap)\circ (f~\begin{tikzpicture}[baseline = 10pt, scale=0.5, color=\clr]
                \draw[-,thick] (0,0.5)to[out=up,in=down](0,1.2);
    \end{tikzpicture})$, for any $f\in 1_{\ob {m+1}}B$. Thanks to  \eqref{B2}, its inverse is
     $\phi^{-1}: B_R\rightarrow B_L$   such that $\phi^{-1}(g) =(g~\begin{tikzpicture}[baseline = 10pt, scale=0.5, color=\clr]
                \draw[-,thick] (0,0.5)to[out=up,in=down](0,1.2);
    \end{tikzpicture})\circ(1_{\ob m} \lcup) $, for any $g\in B 1_{\ob {m+1}}$.
\end{proof}

\begin{Lemma}\label{sksjdd} Let $ \tilde E:= B_L\otimes _{B}?=B_R\otimes _B?$.
Then  $\tilde E$ is self-adjoint and exact.
\end{Lemma}
\begin{proof}
The unit $\beta:\text{Id}_{B\text{-mod}}\rightarrow\tilde E^2$ is  the bimodule homomorphism
$B\rightarrow B_R\otimes_B B_R$ sending $ f$ to $( 1_{\ob m}~\lcap)\otimes(f~\begin{tikzpicture}[baseline = 10pt, scale=0.5, color=\clr]
                \draw[-,thick] (0,0.5)to[out=up,in=down](0,1.2);
    \end{tikzpicture}~)$ for all  $f\in  1_{\ob m}B$, and
  the counit $\alpha:\tilde E ^2\rightarrow \text{Id}_{B\text{-mod}}$ is   the bimodule homomorphism
$B_R\otimes _B B_R\rightarrow B$ sending $ f\otimes g $ to  $f \circ( g~\begin{tikzpicture}[baseline = 10pt, scale=0.5, color=\clr]
                \draw[-,thick] (0,0.5)to[out=up,in=down](0,1.2);
    \end{tikzpicture})~\circ(1_{\ob m}\lcup)$, for all $ f\in B 1_{\ob {n+1}}$, and $g\in 1_{\ob n}B 1_{\ob {m+1}}$.
Using   \eqref{B2},   one can check that $  \alpha\tilde E  \circ  \tilde E   \beta =\text{Id}_{\tilde E}= \tilde E  \alpha \circ \beta  \tilde E $.
So, $\tilde E$ is self-adjoint and exact.

 \end{proof}
  Using arguments similar to those  in \cite[Theorem~IV.2.1]{Re} yields the following result. See also \cite[Lemma~6.7]{Br}. We give a  proof here since it is  crucial for the categorical action later on.
\begin{Lemma}\label{intailh}
There is a short exact sequence of functors from $B^0$-mod to $B$-mod:
\begin{equation}\label{diejdhd}
0\rightarrow \Delta\circ E\rightarrow \tilde E\circ \Delta\rightarrow\Delta\circ  F\rightarrow0.
\end{equation}
\end{Lemma}
\begin{proof}
It suffices to show the following is a short exact sequence of $(B,B^0)$-bimodules
\begin{equation}\label{shoretmor}
0\rightarrow B\otimes _{B^\sharp}B^0_L\overset {\varphi}\rightarrow B_L\otimes _{B^\sharp}B^0\overset{\psi}\rightarrow B\otimes _{B^\sharp}B^0_R\rightarrow0,
\end{equation}
where $\varphi(f\otimes g)= ( f~\begin{tikzpicture}[baseline = 10pt, scale=0.5, color=\clr]
                \draw[-,thick] (0,0.5)to[out=up,in=down](0,1.2);
    \end{tikzpicture})\otimes g$ and $\psi(f\otimes g)=((1_{\ob m}~\lcap)\circ (f~\begin{tikzpicture}[baseline = 10pt, scale=0.5, color=\clr]
                \draw[-,thick] (0,0.5)to[out=up,in=down](0,1.2);
    \end{tikzpicture}))\otimes ( g~\begin{tikzpicture}[baseline = 10pt, scale=0.5, color=\clr]
                \draw[-,thick] (0,0.5)to[out=up,in=down](0,1.2);
    \end{tikzpicture})$ for all admissible Brauer diagrams $f$ and $g$.

 First of all,  it is routine  to check that both  $\varphi $ and $\psi$  are well-defined $(B,B^0)$-homomorphisms. Further,  we have
 $$\psi\varphi(f\otimes g)=(f\lcap)\otimes( g~\begin{tikzpicture}[baseline = 10pt, scale=0.5, color=\clr]
                \draw[-,thick] (0,0.5)to[out=up,in=down](0,1.2);
    \end{tikzpicture})=0, \text{ for  $f\in B,g\in B^0_L$},  $$
  where the last equality   follows from the fact that the cap commutes past the tensor and acts on $B^0_R$ as zero.

  It remains to show the exactness of the sequence. Motivated by the argument in the proof of \cite[Lemma~6.7]{Br}, we consider certain bases of three bimodules as above. Note that $1_{\ob s}B1_{\ob m}=\Hom_{\B(\delta_0)}(\ob m,\ob s)$, which  has basis given by  $\mathbb B_{m,s}/\sim$, the set of equivalence classes of all $(m,s)$-Brauer diagrams. For all $m,s,t\in\mathbb N$, define   $\mathbb B^-_{m,m+2s}=\mathbb B_{m,m+2s}\cap B^-$ and $\mathbb B_{t}^0=\mathbb B_{t,t}\cap B^0$.
  Then $1_{\ob m+2s}B^-1_{\ob m}$ (resp., $1_\ob tB^0 1_{\ob t}$) has  basis  given by  $\mathbb B^-_{m,m+2s}/\sim$ (resp., $\mathbb B_{t}^0/\sim$).
  Since $B^-=\oplus_{m,s\in\mathbb N}1_{\ob m+2s}B^-1_{\ob m} $ and $B^0=\oplus _{t\in \mathbb N} 1_\ob tB^0 1_{\ob t}$, by Lemma~\ref{triangularde}, $J, H$ and $L$ are bases of $ B\otimes _{B^\sharp}B^0_L$, $ B_L\otimes _{B^\sharp}B^0$  and $ B\otimes _{B^\sharp}B^0_R$, respectively, where
  $$\begin{aligned}
  J&=\{f\otimes g\mid f\in \mathbb B^-_{m,m+2s}/\sim, g\in\mathbb B_{m+1}^0/\sim, \forall m,s\in \mathbb N\},\\
  H&=\{f\otimes g\mid f\in \mathbb B^-_{m,m+2s}/\sim, g\in\mathbb B_{m-1}^0/\sim, \forall m,s\in \mathbb N, \text{ and } m\geq 1\},\\
  L&=\{f\otimes g\mid f\in \mathbb B^-_{m,m+2s}/\sim, g\in\mathbb B_{m}^0/\sim,  \forall m,s\in \mathbb N, m\geq 1 \}.
  \end{aligned}$$

 Let $H_1$ be the subset of $H$ consists of all $f\otimes g\in H$ such that the rightmost  vertex on the top row  of $f$ is on a cup. Then
  $H=H_1\overset {.} \sqcup H_2$. Therefore,  the rightmost vertex on the top row of $f$ is on the rightmost vertical string if  $f\otimes g\in H_2$. So $f=\bar f ~ \begin{tikzpicture}[baseline = 10pt, scale=0.5, color=\clr]
                \draw[-,thick] (0,0.5)to[out=up,in=down](0,1.2);
    \end{tikzpicture}$ and $\bar f\otimes g\in J$ for any $f\otimes g\in H_2$, where $\bar f$ is obtained from $f$ by deleting the rightmost vertical string.
 By definition we see that $\varphi$ maps $J$ to $H_2$ and $\psi$ maps $H_1$ to $L$.
  Moreover, $\varphi: J\rightarrow H_2$ has an  inverse map given by $\varphi^{-1}: H_2\rightarrow J, f\otimes g\mapsto \bar f\otimes g$.

We want to show that $\psi:H_1\rightarrow L$ has an inverse map $\psi^{-1}$, too.
 Given any  $g\in\mathbb B_{m}^0 $ with $m\geq 1$ such that the $i$th vertex on the top row is connected with the rightmost  vertex on the bottom row, we
   define $$S_{i,m}=\left\{
                     \begin{array}{ll}
                       S_{i}\circ S_{i+1}\circ \cdots \circ S_{m-1}, & \hbox{if $i<m$;} \\
                       1_{\ob m}, & \hbox{if $i=m$}
                     \end{array}
                   \right.$$
 where $S_i$ is given in \eqref{jdijdijxs}. Then there is some $\tilde g\in \mathbb B^0_{m-1} $ such that $S_{i,m}^{-1}\circ g= \tilde g~\begin{tikzpicture}[baseline = 10pt, scale=0.5, color=\clr]
                \draw[-,thick] (0,0.5)to[out=up,in=down](0,1.2);
    \end{tikzpicture}$.   Moreover,
$$ (f~\begin{tikzpicture}[baseline = 10pt, scale=0.5, color=\clr]
                \draw[-,thick] (0,0.5)to[out=up,in=down](0,1.2);
    \end{tikzpicture})\circ (S_{i,m}~\begin{tikzpicture}[baseline = 10pt, scale=0.5, color=\clr]
                \draw[-,thick] (0,0.5)to[out=up,in=down](0,1.2);
    \end{tikzpicture}) \circ (1_{\ob {m-1}\lcup})\in \mathbb B^-_{m+1,m+2s+1} \text{ for any } f\in \mathbb B^-_{m,m+2s}.$$
 Hence  the required inverse of $\psi$ is
 $$\psi^{-1}: L\rightarrow H_1 , f\otimes g \mapsto ((f~\begin{tikzpicture}[baseline = 10pt, scale=0.5, color=\clr]
                \draw[-,thick] (0,0.5)to[out=up,in=down](0,1.2);
    \end{tikzpicture})\circ (S_{i,m}~\begin{tikzpicture}[baseline = 10pt, scale=0.5, color=\clr]
                \draw[-,thick] (0,0.5)to[out=up,in=down](0,1.2);
    \end{tikzpicture}) \circ (1_{\ob {m-1}\lcup})) \otimes \tilde g, \text{ for any} f\in \mathbb B^-_{m,m+2s}, g\in \mathbb B_{m}^0.$$
 We have shown that $\varphi $ maps $J$ bijectively onto $H_2$ and $\psi$ maps $H_1$ bijectively onto $L$. This completes the proof.
\end{proof}
 Let $\bar {(k,l)}\in \End_{\B(\delta_0)}(\ob m)$ be the morphism such that its diagram is the cap-cup joining the $k$th and $l$th point.
 For example,  $$\bar{(i,i+1)}=\begin{tikzpicture}[baseline = 25pt, scale=0.35, color=\clr]
\draw[-,thick](-2,1.1)to[out= down,in=up](-2,3.9);
         \draw(-1.5,1.1) node{$ \cdots$}; \draw(-1.5,3.9) node{$ \cdots$};
         \draw[-,thick](-0.5,1.1)to[out= down,in=up](-0.5,3.9);
        \draw[-,thick] (0,4) to[out=down,in=left] (0.5,3.5) to[out=right,in=down] (1,4);
         \draw[-,thick] (0,1) to[out=up,in=left] (0.5,1.5) to[out=right,in=up] (1,1);
         \draw(0,4.5)node{\tiny$i$}; \draw(1.3,4.5)node{\tiny$i+1$};
         \draw[-,thick](2,1.1)to[out= down,in=up](2,3.9);
         \draw(3.5,1.1) node{$ \cdots$}; \draw(3.5,3.9) node{$ \cdots$};
         \draw[-,thick](4.5,1.1)to[out= down,in=up](4.5,3.9);
           \end{tikzpicture}.$$
    For $1\leq l\leq m$, define
 \begin{equation}\label{jucyofb}
 L_1^m=0,\quad L_l^{ m}:= \sum_{1\leq k<l}(k,l)-\sum_{1\leq k<l}\bar{(k,l)}\in \End_{\B(\delta_0)}(\ob m) \text{ for $l>1$}.
 \end{equation}
 As stated in section 2, there is a functor $\mathcal F: \AB(\delta)\rightarrow \B(\delta_0)$
  sending $X_1  1_{\ob m}$ to $\frac{\delta_0-1}{2}1_{\ob m}$ for $m\in\mathbb N$ and fixing other generators $A, U$ and $S$. Therefore, we can interpret dotted Brauer diagrams (e.g. the dotted Brauer diagrams in  \eqref{jdijdijxs} and \eqref{AB1}--\eqref{AB2})  as morphisms in $\B(\delta_0)$ later on.

 \begin{Lemma}\label{ejihehudh} For any admissible $i$,
 $X_i 1_{\ob m}= \frac{\delta_0-1}{2}1_{\ob m} + L_i^m$.
\end{Lemma}
\begin{proof} For $i=1$ the result follows since the quotient functor $\mathcal F: \AB(\delta)\rightarrow \B(\delta_0)$ sends $X_11_{\ob m}$ to $\frac{\delta_0-1}{2} 1_{\ob m}$. In general, using \eqref{AB1} we have that
 \begin{equation}\label{cnceced}
\text{$X_{i+1}1_{\ob m}= S_i1_{\ob m}\circ  X_i1_{\ob m} \circ S_i 1_{\ob m} + S_i1_{\ob m} - \bar{(i,i+1)}$ in $\AB$.}
\end{equation}
Then  the result follows from \eqref{cnceced} and induction on $i$.
\end{proof}

 %\begin{Lemma}\label{murofb}
% For all  $a,b,c,d\in\mathbb N$, we have
% \begin{enumerate}
% \item $(g ~\begin{tikzpicture}[baseline = 10pt, scale=0.5, color=\clr]
%                \draw[-,thick] (0,0.5)to[out=up,in=down](0,1.2);
%    \end{tikzpicture} ~h)\circ L_{a+1}^{a+b+1}= L_{c+1}^{c+d+1}\circ (g~ \begin{tikzpicture}[baseline = 10pt, scale=0.5, color=\clr]
%                \draw[-,thick] (0,0.5)to[out=up,in=down](0,1.2);
%    \end{tikzpicture} ~h) $ for any morphisms $(g, h)\in\Hom_{\B}(\ob a,\ob c)\times  \Hom_{\B}(\ob c,\ob d)$,
% %\item $L_{a+2}^{a+b+2}\circ (1_{\ob a}S 1_{\ob b})= (1_{\ob a}S 1_{\ob b})L_{a+1}^{a+b+2}+ 1_{\ob {a+b+2}}- 1_{\ob a}E1_{\ob b}$.
% \end{enumerate}
%
% \end{Lemma}
% \begin{proof} Easy exercise.
%%By Definition~\ref{DefnB}, (a) can be verified for those $g$ which are tensor products of identity morphisms and a single $S$, $A$ or $U$.
%%(b) can be checked directly. $U_i 1_{\ob m}$, $A_i 1_{\ob m}$, $X_i 1_{\ob m}$, $S_i 1_{\ob m}$(also written $ 1_{\ob n} U_i$, $1_{\ob n}A_i$, $1_{\ob n}X_i$, $1_{\ob n}S_i$ respectively)
%%to be the following morphisms in $\B$, for all admissible
% \end{proof}

%The following result is the key difference from that in \cite{Re,Br}.
\begin{Lemma}\label{msikdjde}Let $\tilde X: B_L\rightarrow B_L$ be the linear homomorphism such that the restriction of $\tilde X$  to $ 1_{\ob {m+1}}B$ is given by  the left multiplication of $X_{m+1}1_{\ob{m+1}}$.  Recall $\varphi$ and $\psi$  in \eqref {shoretmor}. Then
\begin{enumerate} \item[(1)]  $\tilde X$ is a $(B,B)$-bimodule homomorphism, \item[(2)] $\varphi \circ (\text{id}\otimes X^0_L)=(\tilde X\otimes \text{id})\circ \varphi$,
 \item [(3)] $\psi\circ (\tilde X\otimes \text{id})=-(\text{id}\otimes X^0_R )\circ \psi$.\end{enumerate}
\end{Lemma}
\begin{proof} For any Brauer diagram $g\in \Hom_{\B(\delta_0)}(\ob m,\ob n)$, we have
$X_{n+1}1_{\ob {n+1}}\circ (g~\begin{tikzpicture}[baseline = 10pt, scale=0.5, color=\clr]
                \draw[-,thick] (0,0.5)to[out=up,in=down](0,1.2);
    \end{tikzpicture})=g~\xdot=(g~\begin{tikzpicture}[baseline = 10pt, scale=0.5, color=\clr]
                \draw[-,thick] (0,0.5)to[out=up,in=down](0,1.2);
    \end{tikzpicture}) \circ X_{m+1}1_{\ob {m+1}}$. This implies
(1).
Note that $\bar{(k,l)}\otimes _{B^\sharp} b=0 $ for any $b\in B^0$. Thanks to Lemma~\ref{ejihehudh}, we have
\begin{equation}\label{sssmndwuhdu}
X_{m+1}1_{\ob {m+1}}\otimes _{B^{\sharp}}b =1_{\ob {m+1}}\otimes _{B^\sharp}(L_{m+1}\circ b), \text{ for any } b\in B^0.
\end{equation}
For all admissible Brauer diagrams $f\in B1_{\ob m}$ and $g\in 1_{\ob {m+1}}B^0$, we have
$$\varphi \circ (\text{id}\otimes X^0_L)(f\otimes g)=\varphi(f\otimes (L_{m+1}\circ g))  \overset{~\eqref{sssmndwuhdu}} =\tilde X(f~\begin{tikzpicture}[baseline = 10pt, scale=0.5, color=\clr]
                \draw[-,thick] (0,0.5)to[out=up,in=down](0,1.2);
    \end{tikzpicture} )\otimes g=(\tilde X\otimes \text{id})\circ \varphi(f\otimes g),$$
  proving (2). Thanks to \cite[Lemma~3.1]{RS3},  we have
    \begin{equation}\label{changeminus}
    \begin{tikzpicture}[baseline = 5pt, scale=0.5, color=\clr]
        \draw[-,thick] (0,0) to[out=up,in=left] (1,2) to[out=right,in=up] (2,0);\draw( 0,0.5)\bdot;
        %\draw[fill=\clr,thick] (-0.1,0.1)--(0.05,0.2)--(0.1,0.1)--(0,0.1);
    \end{tikzpicture}
    ~=~-~\begin{tikzpicture}[baseline = 5pt, scale=0.5, color=\clr]
        \draw[-,thick] (0,0) to[out=up,in=left] (1,2) to[out=right,in=up] (2,0);\draw( 2,0.5) \bdot;
        %\draw[fill=\clr,thick] (0.84,0.2)--(1,0.1)--(1.1,0.2)--(0.84,0.2);
    \end{tikzpicture},   \end{equation}  in $\AB$, hence in both $\B$ and $\B(\delta_0)$. So,
    for all  $f\in 1_{\ob {m+1}}B1_{\ob n}, g\in 1_{\ob n}B^0$,
    $$ \begin{aligned}
    \psi\circ (\tilde X\otimes \text{id})(f\otimes g)&= ((1_{\ob m}~\lcap)\circ ((X_{m+1}1_{\ob {m+1}}\circ f)~ ~\begin{tikzpicture}[baseline = 10pt, scale=0.5, color=\clr]
                \draw[-,thick] (0,0.5)to[out=up,in=down](0,1.2);
    \end{tikzpicture}  ))\otimes (g~\begin{tikzpicture}[baseline = 10pt, scale=0.5, color=\clr]
                \draw[-,thick] (0,0.5)to[out=up,in=down](0,1.2);
    \end{tikzpicture}) \\
    &\overset{\eqref{changeminus}}=-((1_{\ob m}~\lcap)\circ (f  ~\begin{tikzpicture}[baseline = 10pt, scale=0.5, color=\clr]
                \draw[-,thick] (0,0.5)to[out=up,in=down](0,1.2);
    \end{tikzpicture}  )\circ1_{\ob{n+1}}X_{n+1})\otimes (g~\begin{tikzpicture}[baseline = 10pt, scale=0.5, color=\clr]
                \draw[-,thick] (0,0.5)to[out=up,in=down](0,1.2);
    \end{tikzpicture}) \\
    &\overset{\eqref{sssmndwuhdu}}= -((1_{\ob m}~\lcap)\circ (f  ~\begin{tikzpicture}[baseline = 10pt, scale=0.5, color=\clr]
                \draw[-,thick] (0,0.5)to[out=up,in=down](0,1.2);
    \end{tikzpicture}  ))\otimes (L_{n+1}\circ(g~ \begin{tikzpicture}[baseline = 10pt, scale=0.5, color=\clr]
                \draw[-,thick] (0,0.5)to[out=up,in=down](0,1.2);
    \end{tikzpicture} ))\\
    &= -((1_{\ob m}~\lcap)\circ (f  ~\begin{tikzpicture}[baseline = 10pt, scale=0.5, color=\clr]
                \draw[-,thick] (0,0.5)to[out=up,in=down](0,1.2);
    \end{tikzpicture}  ))\otimes ((g~ \begin{tikzpicture}[baseline = 10pt, scale=0.5, color=\clr]
                \draw[-,thick] (0,0.5)to[out=up,in=down](0,1.2);
    \end{tikzpicture} )\circ L_{n+1})\\
    &=-(\text{id}\otimes X^0_R )\circ \psi(f\otimes g),
    \end{aligned} $$
   proving (3).
\end{proof}

There is a  minus in  Lemma~\ref{msikdjde}(3), which   comes from \eqref{changeminus}. This is the essential  difference between our  paper and  \cite{Re}.
It will lead to a different categorification later on.

\begin{Lemma}\label{usuactifuc} For each $i\in K$, let  $ \tilde E_i:= B_{L, i}\otimes _{B}?$, where $B_{L,i}$ is the generalized $i$-eigenspace of $\tilde X$ on $B_L$. Then
\begin{itemize}\item [(1)]  $\tilde E_i$ is exact,
\item [(2)]there is a short exact sequence of functors
$ 0\rightarrow \Delta\circ E_i\rightarrow \tilde E_i\circ \Delta\rightarrow\Delta\circ  F_{-i}\rightarrow0$,
\item [(3)] $\tilde E=\bigoplus_{i\in I}\tilde E_i$, where $I:=I_0\cup -I_0$ and  $I_0=\frac{\delta_0-1}{2}+\mathbb Z1_ K \subset K$ (see Lemma~\ref{filteration}(2)).
\end{itemize}
\end{Lemma}

\begin{proof} We have (1) since  $\tilde E_i$ is a summand  of the exact functor $\tilde E$.  (2) follows from Lemmas~\ref{intailh},\ref{msikdjde} by passing to appropriate generalized eigenspaces.
Thanks to  Lemma~\ref{filteration} and the short exact sequences in \eqref{diejdhd} and (2), we see that (3) holds on any standard module $\Delta(\lambda)$.
Since    $\tilde E_i$'s are exact, by  Proposition~\ref{dektss}, (3) also holds on all indecomposable projective modules. So, (3) holds on any module.
\end{proof}

Thanks to \cite[Lemma~2.8]{LZ}, $1_{\ob m}B1_{\ob m}\cong B_m(\delta_0)$, which is known as the Brauer algebra with defining  parameter $\delta_0$\cite{B}.
For any positive integer $m$,  $\{X_i1_{\ob m}\mid 1\le i\le m\}$ are  Jucys-Murphy elements of $ 1_{\ob m}B1_{\ob m}$,  where
 $X_i1_{\ob m}$'s are given in Lemma~\ref{ejihehudh}. They generate a finite dimensional commutative subalgebra of  $1_{\ob m}B1_{\ob m}$.
 For any $\mathbf i=(i_1,i_2,\ldots, i_m)\in I^m$, there is an idempotent $1_{\mathbf i}\in 1_{\ob m}B1_{\ob m}$ which projects any $1_{\ob m}B1_{\ob m}$-module
onto the simultaneous generalized eigenspaces of $X_11_{\ob m} ,\ldots,X_{\ob m}1_{\ob m}$ with respect to the eigenvalues $i_1,\ldots,i_m$.

\begin{Defn}\label{defincsomega}
For any $\lambda\in\Lambda$ and any box $x$ in row $i$ and column $j$ of the Young diagram of $\lambda$, let $c_{\delta_0}(x)=\frac{\delta_0-1}{2}+c(x)\in I$, where $I$ is given in Lemma~\ref{usuactifuc}(3) and  $c(x)\equiv j-i(\text{mod } p)$,  the usual  content of $x$.
\end{Defn}

Let  $\mathbb B$ be  the graph such that the set of  vertices is  $\Lambda$ and  any edge is of form   $\lambda$---$\mu$
whenever $\mu$ is obtained from $\lambda$ by either  adding or   removing  a box $x$.
A path $\gamma:\lambda \rightsquigarrow\mu$ is of length $m$ if  there is a  sequence of partitions $\lambda=\lambda_0, \lambda_1,\ldots,\lambda_m=\mu$ such that $\lambda_{j-1}$---$\lambda_{j}$ is
  colored by  $i_j\in I$, for all $1\le j\le m$, where
    $$i_j=\begin{cases} c_{\delta_0}(x), & \text{if $\lambda_j$ is obtained from $\lambda_{j-1}$ by adding a box $x$,} \\
  -c_{\delta_0}(x), &\text{if $\lambda_{j}$ is obtained from $\lambda_{j-1}$ by removing a box $x$.}\\
  \end{cases}$$
In this case, we say that   $\gamma$ is of content  $\mathbf i=(i_1, i_2, \ldots, i_m)$ and write   $c(\gamma)=\mathbf i$.
When $m=0$, we set $1_\emptyset=1_{\ob 0}$ and we say there is a unique path from $\emptyset$ to $\emptyset$ with length $0$. The corresponding content  is  $ \emptyset$.

 %We need to compute $\dim(1_{\mathbf i})\tilde\Delta(\lambda)$, for any $m\in\mathbb N $ and $\mathbf i\in %I^m$.
\begin{Defn} For any $V\in B$-lfdmod, define the formal character of $V$ by
\begin{equation}
\text{ch}V=\sum_{m\in\mathbf N, \mathbf i\in I^m}(\dim 1_{\mathbf i}V)e^{\mathbf i},
\end{equation}
where $1_{\mathbf i}V$ is the   simultaneous generalized eigenspace of $X_11_{\ob m} ,\ldots,X_m 1_{\ob m}$ corresponding to $\mathbf i$.
\end{Defn}
%It is known that $\text{ch} V$ is called  the formal character of $V$.
The strategy for the proof of the following result is similar to the case treated in \cite{Br,Re}.
\begin{Prop}\label{charcter}
For any $\lambda\in\Lambda$,
$\text{ch}\tilde \Delta(\lambda)=\sum_{\gamma}e^{c(\gamma)}$
where $\gamma$ ranges over all paths  from $\emptyset$ to $\lambda$.
\end{Prop}
\begin{proof}It suffices to show that $\dim 1_{\mathbf i}\tilde \Delta(\lambda)$ is equal to the number of all  paths from $\emptyset$ to $\lambda$ with content $\mathbf i$, for all $\mathbf i\in I^m$ and $m\in \mathbb N$.
Suppose that $\lambda\in\Lambda(n)$. As in the proof of Theorem~\ref{hd1}, we see that
\begin{equation}\label{zmmzzm}
\tilde\Delta(\lambda)=\oplus_{d\in\mathbb N}1_{\ob {n+2d}}\tilde\Delta(\lambda) \text{ and } 1_{\ob n} \tilde\Delta(\lambda)=1_{\ob n}\otimes S(\lambda).
\end{equation}
Suppose that  $m=0$. If $\lambda\neq\emptyset$, then $1_\emptyset \tilde\Delta(\lambda)=0$. If $\lambda=\emptyset$, then  $1_\emptyset \tilde\Delta(\lambda)=K$. So, the result holds for $m=0$.
Next we
assume  that  $m>0$ and we prove the result by induction on $m$ as follows.

By the definition of $\tilde E$ in Lemma~~\ref{sksjdd}, we have $1_{\ob {m-1}}\tilde E (\tilde \Delta(\lambda))=1_{\ob {m-1}}B_L\otimes _{B}\tilde \Delta(\lambda)=1_{\ob m}B\otimes_B \tilde \Delta(\lambda) $. So, we have a $1_{\ob {m-1}}B1_{\ob{m-1}}$-isomorphism
$$f: 1_{\ob {m-1}}\tilde E (\tilde \Delta(\lambda))=1_{\ob m}B\otimes_B \tilde \Delta(\lambda)\rightarrow 1_{\ob m}  \tilde \Delta(\lambda),b\otimes v\mapsto bv, \text{  for all $b\in 1_{\ob m}B$ and $v\in  \tilde \Delta(\lambda)$ }. $$
Note that $1_{\ob {m-1}}\tilde E_{i} (\tilde \Delta(\lambda))=1_{\ob {m-1}} (B_L)_i\otimes _B \tilde \Delta(\lambda)= (1_{\ob m} B)_i\otimes _B \tilde \Delta(\lambda)$, where
$ (1_{\ob m} B)_i$ is the generalized $i$-eigenspace of $X_m1_{\ob m}$ on $1_{\ob m}B$.
So,  restricting  $f$ to $1_{\ob {m-1}}\tilde E_{i} (\tilde \Delta(\lambda)) $ gives an isomorphism between $ 1_{\ob {m-1}}\tilde E_{i} (\tilde \Delta(\lambda))$ and the
generalized $i$-eigenspace of $X_m1_{\ob m}$ on $1_{\ob m} \tilde \Delta(\lambda)$ and hence
\begin{equation}\label{cdcedccde}
\dim(1_{\mathbf i}\tilde\Delta(\lambda))=\dim(1_{\mathbf i'}\tilde E_{i_m}\tilde\Delta(\lambda)),
\end{equation}
for  any $\mathbf i\in I^m$, where $\mathbf i'=(i_1,\ldots,i_{m-1})$.

Thanks to Lemma~\ref{usuactifuc}(2) and  Lemma~\ref{filteration}(3)-(4), $\tilde E_i\tilde \Delta(\lambda)$ has a
$\tilde\Delta$-filtration such that $\tilde\Delta(\mu)$ appears as a section if and only if $\mu$ is obtained from $\lambda $ by either  removing a box $x$ satisfying  $i=c_{\delta_0}(x)$ or
adding a box $x$ satisfying $i=-c_{\delta_0}(x)$.
Now the result follows from \eqref{cdcedccde} and induction on $m$.

\end{proof}

\begin{rem}Suppose $\lambda\in\Lambda(n)$. As in \eqref{zmmzzm},   $1_{\ob m}\tilde \Delta(\lambda)\neq 0$
only if $m=n+2f$ for some $f\in\mathbb N$. In this case,  $1_{\ob m}\tilde \Delta(\lambda)$ is isomorphic to the cell module of $1_{\ob m}B1_{\ob m}$ with respect to $\lambda$ (see Proposition~\ref{xeijdiec}).
Then Proposition ~\ref{charcter} can also be obtained by using the fact (see \cite[10.7]{Eny}) that the cell module $1_{\ob m}\tilde \Delta(\lambda)$
has a Murphy  basis indexed by  $\{\gamma: \emptyset\rightsquigarrow\lambda\mid \gamma \text{ is of  length $m$}\}$, and the simultaneous generalized eigenvalue of  corresponding basis element is   $c(\gamma)$.
\end{rem}
\begin{Cor}\label{twopath}
Suppose  $\lambda\in\Lambda, \mu\in \Lambda_p(m)$.  If $[\tilde \Delta(\lambda):L(\mu)]\neq0$, then there are two  paths $\gamma:\emptyset \rightsquigarrow \lambda$
and $\delta: \emptyset\rightsquigarrow\mu$  of length $m$ such that  $c(\gamma)=c(\delta)$.
\end{Cor}
\begin{proof}Thanks to Theorem~\ref{hd1}(2),
we have    $1_{\ob m}L(\mu)=L(\mu)^s=D(\mu)\neq 0$ since $\mu\in \Lambda_p(m)$.  Pick an $ \mathbf i\in I^m$ such that $1_{\mathbf i}L(\mu)\neq0$. Since we are assuming that $[\tilde \Delta(\lambda):L(\mu)]\neq0$,  we have $1_{\mathbf i}\tilde\Delta(\lambda)\neq0$.
Thanks to Proposition~\ref{dektss}(2) and Corollary~\ref{xijixs}(2), $[\bar \Delta(\mu): L(\mu)]=1$. Since $\Delta$ is exact, $[\tilde \Delta(\mu):L(\mu)]\neq0$, forcing    $1_{\mathbf i}\tilde\Delta(\mu)\neq0$. Now, the result follows from Proposition~\ref{charcter}, immediately.
\end{proof}
\begin{Cor}\label{sjdsdh}
Suppose $\delta_0\notin\mathbb Z 1_K$. Then \begin{itemize}\item [(1)] $P(\mu)=\Delta(\mu)$ for all  $\mu\in \Lambda_p$,
\item [(2)]  $\Delta: B^0\text{-mod}\rightarrow B\text{-mod} $ is an equivalence of categories. \end{itemize}
\end{Cor}
\begin{proof}
By Proposition~\ref{dektss} and Corollary~\ref{xijixs}, we have (1) if   $[\bar\Delta(\lambda):L(\mu)]=0$ for any two distinct
 $\mu, \lambda\in \Lambda_p$. If it were false, then
$[\bar\Delta(\lambda):L(\mu)]\neq0$, for some  $\mu\in \Lambda_p(m)$ and some $\lambda\neq \mu$. So,
 $[\tilde \Delta(\lambda):L(\mu)]\neq0$. By Corollary~\ref{twopath}, there are two paths   $\gamma:\emptyset \rightsquigarrow \lambda$  and $\delta: \emptyset\rightsquigarrow\mu$  such that
$c(\gamma)=c(\delta)\in I^m$.
%Moreover, both of them have lengths $m$.
Since $\mu\in \Lambda_p(m)$, we have
$c(\delta)=(c_{\delta_0}(x_1),\ldots,c_{\delta_0}(x_m))$, where $(x_1, \ldots,x_m)$ is a permutation of all boxes in $\mu$.
Note  that $\delta_0\notin \mathbb  Z1_K $. So,   $-\frac{\delta_0-1}{2} +a\neq  c_{\delta_0}(x_i)$ for all admissible $i$ and all $a\in \mathbb Z 1_K$, forcing   $\lambda\in \Lambda_p(m)$.
So, $1_{\ob m}\bar\Delta(\lambda)=D(\lambda)$ (see Theorem~\ref{hd1}). Since $[\bar\Delta(\lambda):L(\mu)]\neq0$, there is a non-zero $B$-homomorphism from
 $L(\mu)$ to a quotient module of $\bar \Delta(\lambda)$. Acting $1_{\ob m}$ yields  a non-zero $B^0$-homomorphism from $D(\mu)$ to $D(\lambda)$, forcing  $\lambda=\mu$, a contradiction.

  By (1),  the exact functor $\Delta$ sends  projective $B^0$ modules $Y(\lambda)$'s  to  projective $B$-modules $P(\lambda)$'s. Recall that $B^0\cong \oplus_{m\in\mathbb N} K\mathfrak S_m$. It is well-known that
$\{Y(\lambda)\mid \lambda\in\Lambda_p \}$ gives a complete set of representatives of all indecomposable objects of  $  B^0\text{-pMod} $.
 Thanks to the general  result on locally unital  algebras in  \cite[Corollary~2.5]{BRUNDAN}, we see that $\Delta$ is an equivalence of categories, proving (2).
\end{proof}

\begin{Defn}\label{ksiijejeddd}
Recall that $ 1_{\ob m}B1_{\ob m}$   is isomorphic to the Brauer algebra $B_m(\delta_0)$ with the defining parameter $\delta_0$.  Define
\begin{enumerate}

\item  $E^0:=1$,  $E^k:=U^{\otimes k}\circ A^{\otimes k}\in 1_{\ob {2k}}B1_{\ob {2k}}$ for  $1\leq k\leq \lceil m/2\rceil$, where $U=\lcup$ and $ A=\lcap$ as in section~2.

%$U^{\otimes k}=\overset{k \text{ times}}{\overbrace{U\otimes U\otimes \cdots\otimes U}}\in 1_{\ob {2k}}B1_{\ob {0}}$ and $A^{\otimes k}=\overset{k \text{ times }}{\overbrace{A\otimes A\otimes \cdots\otimes A}}\in 1_{\ob {0}}B1_{\ob {2k}}$.
%

%$$
 %U^{\otimes k}=\overbrace{U\otimes U\otimes \cdots\otimes U}\in 1_{\ob {2k}}B1_{\ob {0}}%\begin{tikzpicture}[baseline = 25pt, scale=0.35, color=\clr]
%        \draw[-,thick] (0,4) to[out=down,in=left] (2.5,1) to[out=right,in=down] (5,4);
%        \draw[-,thick] (0.5,4) to[out=down,in=left] (2.5,1.5) to[out=right,in=down] (4.5,4);
%        \draw(1.5,3.5)node{$\cdots$};\draw(3.5,3.5)node{$\cdots$};
%        \draw[-,thick] (2,4) to[out=down,in=left] (2.5,3) to[out=right,in=down] (3,4);
%        \draw (1,4.5)node {$\ob k$}; \draw (4,4.5)node {$\ob k$};
%           \end{tikzpicture}
%\quad \text{ and }\quad

%         A_{\ob k}= \begin{tikzpicture}[baseline = 25pt, scale=0.35, color=\clr]
%           \draw[-,thick] (0,1.5) to[out=up,in=left] (2.5,4.5) to[out=right,in=up] (5,1.5);
%           \draw[-,thick] (0.5,1.5) to[out=up,in=left] (2.5,4) to[out=right,in=up] (4.5,1.5);
%           \draw[-,thick] (2,1.5) to[out=up,in=left] (2.5,2.5) to[out=right,in=up] (3,1.5);
%           \draw(1.5,1.8)node{$\cdots$};\draw(4,1.8)node{$\cdots$};
%          \draw (1,1)node {$\ob k$}; \draw (4,1)node {$\ob k$};
%           \end{tikzpicture}
%\in 1_{\ob {0}}B1_{\ob {2k}},  \text{  for $1\leq k\leq \lceil m/2\rceil$},$$
\item $D_{k,m}$, a complete set of left coset representative for $\mathcal B_f\times \mathfrak S_{m-2k}$ in $\mathfrak S_m$,
where  $\mathcal B_0=\langle1\rangle$, $\mathcal B_1=\langle S_1\rangle$, and
$\mathcal B_k=\langle S_{m-2i+1}, S_{m-2i}S_{m-2i-1}S_{m-2i+1}S_{m-2i}\mid  1\leq i\leq k\rangle$,
\item  $B^k$,  the two-sided ideal of $B_m(\delta_0)$ generated by $1_{\ob {m-2k}}\otimes E^k$ (which will be denoted by  $1_{\ob {m-2k}} E^k$ by abusing of notation in \eqref{com1}),
    %\item  $L_m(\lambda):=1_{\ob m}L(\lambda)$, for all $\lambda\in\Lambda_p $,
    %\item $C_m(\mu):=1_{\ob m}\tilde \Delta(\mu) $ for all  $\mu\in \Lambda$,
    \item  $\Lambda^+(m)=\coprod_{0\leq d\leq \lceil m/2\rceil}\Lambda(m-2d)$, and $\Lambda_p^+(m)=\coprod_{0\leq d\leq \lceil m/2\rceil}\Lambda_p(m-2d)$, \end{enumerate}\end{Defn}

%The following result describe  cell modules of $B_m(\delta_0)$ are given in
It is proved in \cite{GL} that $B_m(\delta_0)$ is a cellular algebra in the sense of \cite{GL}  and it has  cell modules as follows.

\begin{Lemma} \cite{DHW, CVM}
For any $\lambda\in\Lambda(m-2k)$,  $\Delta_m(\lambda):= B^k/B^{k+1}\otimes _{K\mathfrak S_{m-2k}}S(\lambda)$ is the   (left) cell module of $B_m(\delta_0)$. It  has a basis
$\{d \circ  (1_{\ob {m-2k}} E^k)\otimes x\mid d\in D_{k,m}, x\in \mathcal S_\lambda  \}$,
where $\mathcal S_\lambda$ is a basis of $S(\lambda)$.\end{Lemma}

\begin{Prop}\label{xeijdiec} Suppose $\mu\in\Lambda$ and $\lambda\in \Lambda_p$.
\begin{itemize}
\item [(1)] $1_{\ob m}\tilde\Delta(\mu)\neq 0$ only if $\mu\in\Lambda^+(m)$, in which case
$1_{\ob m}\tilde\Delta(\mu)\cong \Delta_m(\mu)$ as $B_m(\delta_0)$-modules.
\item [(2)]$1_{\ob m } L(\lambda)\neq0$ if and only if $ \lambda\in\tilde\Lambda_p^+(m)$, where
$$\tilde\Lambda_p^+(m)=\left\{
                      \begin{array}{ll}
                         \Lambda_p^+(m)\setminus\{\emptyset\}, & \hbox{ if $m$ is positive even and $\delta_0=0$;} \\
                         \Lambda_p^+(m), & \hbox{otherwise.}
                      \end{array}
                    \right.
 $$
\item [(3)] If $\lambda\in \tilde\Lambda_p^+(m)$ and $\mu \in\Lambda^+(m) $,  then $[\tilde \Delta(\mu):L(\lambda)]=[1_{\ob m} \tilde \Delta(\mu): 1_{\ob m} L(\lambda)]$, and   $1_{\ob m}L(\lambda)$ is the simple head of $1_{\ob m} \tilde \Delta(\lambda)$.
\end{itemize}
\end{Prop}
\begin{proof}Suppose $\mu\in\Lambda(n)$.
Recall  $\tilde \Delta(\mu)=B\otimes _{B^\sharp} S(\mu)$. By \eqref{zmmzzm},
$1_{\ob m}\tilde \Delta(\mu)\neq 0$  only if $ m=n+2k$.
When  $m=n+2k$,
$1_{\ob m}\tilde \Delta(\mu)= 1_{\ob m}B^-\otimes _{\mathbb K} S(\mu)=1_{\ob m}B^-[k]1_{\ob n}\otimes _{\mathbb K} S(\mu)$. This together with
 Theorem~\ref{basisofb} yields that  $1_{\ob m}\tilde \Delta(\mu)$
has a basis $ \{d \circ  (1_{\ob {m-2k}} U^{\otimes k})\otimes x\mid d\in D_{k,m}, x\in \mathcal S_\mu  \}$, where $1_{\ob {m-2k}} U^{\otimes k}= 1_{\ob {m-2k}}\otimes  U^{\otimes k}\in 1_{\ob m}B^- 1_{\ob {m-2k}}$ (see \eqref{com1}).
Then one can check that  the required isomorphism in (1) is the
 $K$-linear isomorphism sending
$d \circ  (1_{\ob {m-2k}} U^{\otimes k})\otimes x$ to  $d \circ  (1_{\ob {m-2k}} E^k)\otimes x$ for all $d\in D_{k,m}, x\in \mathcal S_\mu$.

Since $1_{\ob m}$ is an idempotent, the set $\{1_{\ob m}L(\lambda)\neq0\mid \lambda\in \Lambda_p\}$ consists of a complete set of all non-isomorphic simple modules of $1_{\ob m}B1_{\ob m}$.
Suppose $\lambda\in\Lambda_p(n)$. If  $1_{\ob m}L(\lambda)\neq 0$, then $1_{\ob m}\bar \Delta(\lambda)\neq 0$, which is available only if  $m=n+2d$ for some $d\in \mathbb N$.
So,   $\lambda \in\Lambda_p^+(m)$.
Thanks to Lemma~\ref{omehga0}, we have $1_{\ob m}L(\emptyset)=0$ if $\delta_0=0$ and $m>0$.
Now, (2) follows from a well-known result that there is a bijection between  $\tilde\Lambda_p^+(m)$ and the set of pairwise non-isomorphic simple modules for Brauer algebra $B_m(\delta_0)$.

  There is an epimorphism  $\tilde\Delta(\lambda)\twoheadrightarrow L(\lambda)$. If $1_{\ob m} L(\lambda)\neq 0$, then  it induces a non-zero  homomorphism from $1_{\ob m}\tilde\Delta(\lambda)$ to  $1_{\ob m}L(\lambda)$. This proves the last assertion in  (3). The first one follows  immediately from (1)--(2) since  $1_{\ob m}$ is an idempotent element.
\end{proof}

\begin{Theorem}\label{semi}
$B$ is semisimple   if and only if $\delta_0\notin\mathbb Z1_K $ and $p=0$.
\end{Theorem}
\begin{proof} Suppose  $\delta_0\notin\mathbb Z1_K $.
Thanks to Corollary ~\ref{sjdsdh},  $B$-mod is equivalent to $B^0$-mod. So, $B$  is semisimple   if and only if  $p=0$.
If  $\delta_0=n\in \mathbb Z$, thanks to the results on semisimplicity of  Brauer algebras in \cite{Wen, Rui},   $B_{m}(n)$ is not semisimple for some  $m$. In this case, one can find a cell module, say $\Delta_m(\mu)$,  for  $B_{m}(n)$,  which is not completely reducible. By Proposition~\ref{xeijdiec}, $\tilde\Delta(\mu)$ can not be written as a direct sum of irreducible modules. So,
 $B$ is not semisimple.
\end{proof}

\section{Lie algebra action  on the Grothendieck Group}
In this section, we keep the assumption that $K$ is a field with characteristic $p\neq 2$.  Recall that $I_0=\frac{\delta_0-1}{2}+\mathbb Z1_K$.  Consider the  Cartan matrix $(a_{i,j})_{i,j\in I_0}$
such that \begin{equation}
a_{i,j}=\begin{cases}
            2, & \text{if $i=j$. } \\
            -1, & \text{if $|i-j|= 1$,} \\
            0, & \text{otherwise.}
          \end{cases}
\end{equation}
 Let $\mathfrak {sl}_{K}$ be the  Kac-Moody algebra associated to the  Cartan matrix $(a_{i,j})_{i,j\in I_0}$ over $\mathbb C$. Then  \begin{equation} \label{lial} \mathfrak{sl}_K\cong\begin{cases} \mathfrak {sl}_{\infty}, & \text{if $p=0$,}\\
 \widehat{\mathfrak {sl}}_{p}, &\text{if $2\nmid p$.}\\ \end{cases}\end{equation}
  Let $\mathfrak h$ be the Cartan subalgebra of $\mathfrak{sl}_K$. Define $h_i=[e_i,f_i]$, where $\{e_i, f_i\mid i\in I_0\}$ is  the set of usual Chevalley generators of  $\mathfrak {sl}_{K}$.  Throughout, we fix the  notations as follows:
\begin{itemize}
\item $P= \{\lambda\in\mathfrak h^*\mid \langle h_i,\lambda \rangle\in \mathbb Z,  \forall i\in I_0\}$,   weight lattice,
 \item $\Pi=\{\alpha_i\mid i\in I_0\}$, the set of simple roots,
 \item $Q=\sum_{i\in I_0}\mathbb Z \alpha_i$, root lattice,
 \item $ \varpi_i$, $i\in I_0 $,   fundamental dominant weights,
 \item $\leq$ ,   the usual dominance order on $P$ such that $\lambda\leq \mu$ if $\mu-\lambda\in \mathbb N \Pi$.
\end{itemize}

 Suppose $\delta_0\in \mathbb Z1_K$.
 If $p=0$, then $I_0= \mathbb Z$ (resp., $ \mathbb Z+\frac{1}{2}$) if $\delta_0$ is odd (resp., even).
 If $p>0$, we say $\delta_0$ is even (resp., odd) if $\delta_0\equiv 2a~(\text{mod $p$})$ (resp., $\delta_0\equiv 2a+1~(\text{mod $p$})$)
 for some $a\in \mathbb Z$ such that  $2a\in\mathbb Z_p=\{0,1,\ldots,p-1\}$ (resp., $2a+1\in \mathbb Z_p$).
   Then
$I_0=\mathbb Z_p 1_K +\frac{1}{2} \bar\delta_0$, where  $\bar\delta_0$ is $1$ if $\delta_0$ is even and $0$ if  $\delta_0 $ is odd. If $p\neq 0$ we  write $p=2r+1$ since we are assuming that $p\neq 2$.
Consider the graph automorphism $\theta$ of Dynkin diagram associated to $I_0$ as in Figures 1--4 as follows.
 Then   $\theta$ sends the $i$th vertex to $j$th vertex if $j\equiv -i\pmod p$. Following \cite[(4.19)]{KW}, it
 induces  an automorphism of $\mathfrak {sl}_{K} $, denoted by $\theta$  such that
$\theta(e_i)=e_{\theta(i)}$, $\theta(f_i)=f_{\theta(i)}$, $\theta(h_i)=h_{\theta(i)}$ and
 $\langle\theta(h),\alpha_{\theta(i)}\rangle=\langle h, \alpha_{i} \rangle$.
Moreover, $\theta$ induces a linear map on $\mathfrak h^*$ such that $\theta(\alpha_i)=\alpha_{\theta(i)}$.

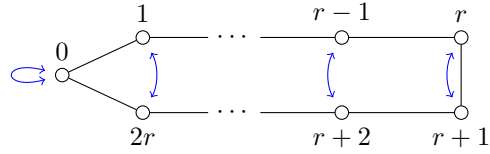
\begin{figure}[ht!]
\caption{Dynkin diagram of type $A^{(1)}_{2r}$ with involution   $\theta$ when  $\delta_0$ is odd and $p>0$.}
   \label{figure:ij}
\begin{tikzpicture}
\matrix [column sep={0.6cm}, row sep={0.5 cm,between origins}, nodes={draw = none,  inner sep = 3pt}]
{
	&\node(U1) [draw, circle, fill=white, scale=0.6, label = 1] {};
	&\node(U3) {$\cdots$};
	&\node(U4)[draw, circle, fill=white, scale=0.6, label =$r-1$] {};
	&\node(U5)[draw, circle, fill=white, scale=0.6, label =$r$] {};
\\
	\node(L)[draw, circle, fill=white, scale=0.6, label =0] {};
	&&&&&
\\
	&\node(L1) [draw, circle, fill=white, scale=0.6, label =below:$2r$] {};
	&\node(L3) {$\cdots$};
	&\node(L4)[draw, circle, fill=white, scale=0.6, label =below:$r+2$] {};
	&\node(L5)[draw, circle, fill=white, scale=0.6, label =below:$r+1$] {};
\\
};
\begin{scope}
\draw (L) -- node  {} (U1);
\draw (U1) -- node  {} (U3);
\draw (U3) -- node  {} (U4);
\draw (U4) -- node  {} (U5);
\draw (U5) -- node  {} (L5);
\draw (L) -- node  {} (L1);
\draw (L1) -- node  {} (L3);
\draw (L3) -- node  {} (L4);
\draw (L4) -- node  {} (L5);
\draw (L) edge [color = blue, loop left, looseness=40, <->, shorten >=4pt, shorten <=4pt] node {} (L);
\draw (L1) edge [color = blue,<->, bend right, shorten >=4pt, shorten <=4pt] node  {} (U1);
\draw (L4) edge [color = blue,<->, bend left, shorten >=4pt, shorten <=4pt] node  {} (U4);
\draw (L5) edge [color = blue,<->, bend left, shorten >=4pt, shorten <=4pt] node  {} (U5);
\end{scope}
\end{tikzpicture}
\end{figure}

\begin{figure}[ht!]
\caption{Dynkin diagram of type $A^{(1)}_{2r}$ with involution   $\theta$ when $\delta_0$ is even and $p>0$.}
   \label{figure:ji2}
\begin{tikzpicture}
\matrix [column sep={0.6cm}, row sep={0.5 cm,between origins}, nodes={draw = none,  inner sep = 3pt}]
{
	\node(U1) [draw, circle, fill=white, scale=0.6, label = $\hf$] {};
	&\node(U2)[draw, circle, fill=white, scale=0.6, label =$\frac{3}{2}$] {};
	&\node(U3) {$\cdots$};
	&\node(U5)[draw, circle, fill=white, scale=0.6, label =$r-\hf$] {};
\\
	&&&&
	\node(R)[draw, circle, fill=white, scale=0.6, label =$r+\hf$] {};
\\
	\node(L1) [draw, circle, fill=white, scale=0.6, label =below:$2r+\hf$] {};
	&\node(L2)[draw, circle, fill=white, scale=0.6, label =below:$2r-\hf$] {};
	&\node(L3) {$\cdots$};
	&\node(L5)[draw, circle, fill=white, scale=0.6, label =below:$r+\frac{3}{2}$] {};
\\
};
\begin{scope}
\draw (U1) -- node  {} (U2);
\draw (U2) -- node  {} (U3);
\draw (U3) -- node  {} (U5);
\draw (U5) -- node  {} (R);
\draw (U1) -- node  {} (L1);
\draw (L1) -- node  {} (L2);
\draw (L2) -- node  {} (L3);
\draw (L3) -- node  {} (L5);
\draw (L5) -- node  {} (R);
\draw (R) edge [color = blue,loop right, looseness=40, <->, shorten >=4pt, shorten <=4pt] node {} (R);
\draw (L1) edge [color = blue,<->, bend right, shorten >=4pt, shorten <=4pt] node  {} (U1);
\draw (L2) edge [color = blue,<->, bend right, shorten >=4pt, shorten <=4pt] node  {} (U2);
\draw (L5) edge [color = blue,<->, bend left, shorten >=4pt, shorten <=4pt] node  {} (U5);
\end{scope}
\end{tikzpicture}
\end{figure}

\begin{figure}[ht!]
\caption{Dynkin diagram of type $A_{\infty}$ with involution   $\theta$ when $\delta_0$ is odd and $p=0$.}
   \label{figure:ji3}
\begin{tikzpicture}
%\draw (-1.5,0) node {$A_{2r}:$};
 \draw[dotted] (-1,0) --(0,0);\draw[dotted] (6,0) --(7,0);
 \draw[dotted]  (0.5,0) node[below]  {$ -m+\hf$} -- (2.5,0) node[below]  {$-\hf$} ;
 \draw (2.5,0)
 -- (3.5,0) node[below]  {$\hf$};
 \draw[dotted] (3.5,0) -- (5.5,0) node[below] {$m-\hf$} ;
\draw (0.5,0) node (-m) {$\bullet$};
 \draw (2.5,0) node (-1) {$\bullet$};
%\draw (3,0) node (0) {$\bullet$};
\draw (3.5,0) node (1) {$\bullet$};
\draw (5.5,0) node (m) {$\bullet$};
\draw[<->][color = blue] (-m.north east) .. controls (3,1) .. node[above] {$\theta$} (m.north west) ;
\draw[<->] [color = blue](-1.north) .. controls (3,0.5) ..  (1.north) ;
%\draw[<->] (0) edge[<->, loop above] (0);
\end{tikzpicture}
\end{figure}

\begin{figure}[ht!]
\caption{Dynkin diagram of type $A_{\infty}$ with involution   $\theta$ when $\delta_0$ is  even and $p=0$.}
   \label{figure:ji4}
\begin{tikzpicture}
%\draw (-2,0) node {$A_{2r+1}:$};
 \draw[dotted] (-1,0) --(0,0);\draw[dotted] (6,0) --(7,0);
 \draw[dotted]  (0,0) node[below] {$-m$} -- (2,0) node[below] {$-1$} ;
 \draw (2,0) -- (3,0) node[below]  {$0$}
 -- (4,0) node[below] {$1$};
 \draw[dotted] (4,0) -- (6,0) node[below] {$m$} ;
\draw (0,0) node (-m) {$\bullet$};
 \draw (2,0) node (-1) {$\bullet$};
\draw (3,0) node (0) {$\bullet$};
\draw (4,0) node (1) {$\bullet$};
\draw (6,0) node (m) {$\bullet$};
\draw[<->][color = blue] (-m.north east) .. controls (3,1.5) .. node[above] {$\theta$} (m.north west) ;
\draw[<->] [color = blue](-1.north) .. controls (3,1) ..  (1.north) ;
\draw[<->] (0) edge[color = blue,<->, loop above] (0);
\end{tikzpicture}
\end{figure}

%Following \cite[(4.19)]{KW}, we have the following result.
%The following result is well-known.

\begin{Defn}\label{prec} Define the partial order $\preceq $ on $P$ such that  $\lambda\preceq \mu$
if  $\bar \lambda=\bar \mu $  and $ \lambda\leq \mu$, where
 $\bar \lambda$ is the image of $\lambda$ in $P_\theta$,  $P_\theta=P/Q^\theta$ and
$Q^\theta=\{ \theta(\mu)+\mu\mid \mu\in Q\}$ (resp.,  $0$) if $\delta_0\in \mathbb Z 1_K$ (resp., $\delta_0\notin \mathbb Z 1_K$).
\end{Defn}

Recall $I=I_0\cup-I_0$ in Lemma~\ref{usuactifuc}(3). Suppose that $\delta_0\in\mathbb Z1_K$. Then  $I_0=-I_0$ and $I=I_0$. Further, for  any $i\in I_0$,   $f_{-i}=f_{j}$ for some $j\in I_0$ such that  $-i\equiv j (\text{mod } p)$.
\begin{Defn} \label{liesub}  Let  $\mathfrak g$ be  the Lie  subalgebra  of $\mathfrak {sl}_{K}$ such that
\begin{itemize} \item [(1)]  $ \mathfrak g=\mathfrak {sl}_{K}$ if  $\delta_0\notin\mathbb Z1_K$,
\item [(2)]  $\mathfrak g$ is generated by  $\{e_i+f_{-i}\mid i\in I_0\}$ if $\delta_0\in \mathbb Z1_K$. \end{itemize}
\end{Defn}

Let $K_0(B^0\text{-pmod})$ be the split Grothendieck group of the category of finitely generated projective $B^0$-modules.
Let $[K_0(B^0\text{-pmod}) ]=\mathbb C\otimes_{\mathbb Z}K_0(B^0\text{-pmod}) $ and $ [K_0(B\text{-mod}^{\Delta})]=\mathbb C\otimes_{\mathbb Z}K_0(B\text{-mod}^{\Delta}) $.
Recall functors $E_i, F_i, \tilde E_i$ in Lemmas~\ref{filteration},\ref{usuactifuc}. $E_i, F_i$ (resp., $ \tilde E_i$) induce linear maps on $ [K_0(B^0\text{-pmod}) ]$ (resp.,
$ [K_0(B\text{-mod}^{\Delta})]$). They are also denoted by $E_i, F_i$ and $ \tilde E_i$, respectively.

\begin{Theorem}\label{cateofg} As $\mathfrak g$-modules,  $ [K_0(B\text{-\rm mod}^{\Delta})]\cong V( \varpi_{\frac{\delta_0-1}{2}})$, where $V( \varpi_{\frac{\delta_0-1}{2}})$ is  the integrable highest weight $\mathfrak {sl}_K$-module of highest weight $ \varpi_{\frac{\delta_0-1}{2}}$. It sends  $[\Delta(\emptyset)]$ to  the unique highest weight vector with weight $\varpi_{\frac{\delta_0-1}{2}} $.  Moreover, $\tilde e_i$ acts on  $ [K_0(B\text{-\rm mod}^{\Delta})]$ via
 $\tilde E_i$ for all  $i\in I$, where
 \begin{equation} \tilde e_i =\begin{cases}
                   e_i+f_{-i}, & \text{if $\delta_0\in \mathbb Z1_K $ and  $i\in I_0$, } \\
                   e_i, & \text{if $\delta_0\notin \mathbb Z1_K$ and  $i\in I_0$,}\\
                 f_{-i}, & \text{if $\delta_0\notin \mathbb Z1_K$ and  $i\in -I_0$.} \end{cases}\end{equation}

\end{Theorem}
\begin{proof}Let $v_{ \varpi_{\frac{\delta_0-1}{2}}}$ be the nonzero highest weight vector of $V( \varpi_{\frac{\delta_0-1}{2}})$ with weight $\varpi_{\frac{\delta_0-1}{2}} $.
Thanks to  Lemma~\ref{filteration}(1) and \cite[Theorem~9.5.1]{Kle}, there is an $\mathfrak{sl}_K$-isomorphism
\begin{equation}\label{nxjsnx}
\phi: [K_0(B^0\text{-pmod}) ]\rightarrow  V( \varpi_{\frac{\delta_0-1}{2}}),
\end{equation} such that $\phi  ( [Y(\emptyset)])= v_{ \varpi_{\frac{\delta_0-1}{2}}}$, and   $e_i$ and $f_i$ act on $[K_0(B^0\text{-pmod}) ]$
 via $E_i$ and  $F_i$  for all $i\in I_0$, respectively.
 The exact functor $\Delta$ gives a linear isomorphism $$\Delta: [K_0(B^0\text{-pmod}) ]\rightarrow [K_0(B\text{-mod}^\Delta) ]$$ which sends $[Y(\lambda)]$ to $[\Delta(\lambda)]$ for all $\lambda\in \Lambda_p$. Define $$e_i[\Delta(\lambda)]=[\Delta(E_i Y(\lambda))] \text{  and
  $f_i[\Delta(\lambda)]=[\Delta(F_i Y(\lambda))]$ for all admissible $i$ and $\lambda$.}$$
 It induces an $\mathfrak {sl}_K$-structure on $ [K_0(B\text{-mod}^\Delta) ]$ and  $\Delta$ is an $\mathfrak{sl}_K$-homomorphism.
 So, $\phi\circ \Delta^{-1}$ is an $\mathfrak{sl}_K$-isomorphism.
 Restricting $\phi\circ \Delta^{-1}$  to $\mfg$ yields the required $\mfg$-isomorphism.
 Thanks to  Lemma~\ref{usuactifuc}, $$ [{\tilde E_i} \Delta(\lambda)]  =[\Delta(E_i Y(\lambda))]+[\Delta(F_{-i} Y(\lambda))] =\begin{cases} (e_i+f_{-i})[\Delta(\lambda)], &\text{if $\delta_0\in \mathbb Z 1_K$ and $i\in I_0$,} \\
 e_i[\Delta(\lambda)], &\text{if $\delta_0\not\in \mathbb Z 1_K$ and $i\in I_0$,} \\
f_{-i}[\Delta(\lambda)], &\text{if $\delta_0\not\in \mathbb Z 1_K$ and $i\in -I_0$.} \\
\end{cases}\\
$$
 In any case, $\tilde e_i$ acts on  $ [K_0(B\text{-mod}^{\Delta})]$ via the endofunctor  $\tilde E_i$ for any $i\in I$.
  %(resp., $I_0$, $-I_0$) depending on whether $\delta_0\in \mathbb Z1_K$ or not.
%The remaining assertion is trivial.
\end{proof}
\begin{rem} Although $ [K_0(B\text{-mod}^{\Delta})]\cong V( \varpi_{\frac{\delta_0-1}{2}})$ as $\mathfrak{sl}_K$-modules, we cannot say that $ [K_0(B\text{-mod}^{\Delta})]$ categorifies   $\mathfrak{sl}_K$-module  $V( \varpi_{\frac{\delta_0-1}{2}})$. The reason is that the actions of Chevalley generators $e_i$'s and $f_i$'s on  $ [K_0(B\text{-mod}^{\Delta})]$ cannot be explained as certain endofunctors on $B\text{-mod}^{\Delta}$.\end{rem}
\begin{Defn}\label{wt123} For any  $\mu\in \Lambda$, define $
\text{wt}(\mu)=\varpi_{\frac{\delta_0-1}{2}}-\sum_{x\in \mu}\alpha_{c_{\delta_0}(x)}\in P$ and $\text{wt}_0(\mu)=\bar{\text{wt}(\mu)}\in P_\theta$, where
$c_{\delta_0}(x)$ is given in Definition~\ref{defincsomega}.
\end{Defn}

Again the strategy for the following result is to use induction on the length of the path from $\emptyset $ to a partition $\lambda$. This is similar to the case treated for the oriented Brauer category \cite{Re} and oriented skein category \cite{Br}.
\begin{Prop}\label{dedji}
Suppose  $\lambda\in\Lambda$ and $\mu\in\Lambda_p$. If
$[\tilde \Delta(\lambda):L(\mu)] \neq 0$, then   $\text{wt}(\mu) \preceq \text{wt}(\lambda)$.
\end{Prop}

\begin{proof}
Suppose that $\mu\in\Lambda_p(m)$. Thanks to  Corollary ~\ref{twopath}, there are two  paths $\gamma:\emptyset\rightsquigarrow \lambda$ and  $\delta: \emptyset \rightsquigarrow\mu$ of length
$m$ such that  $c(\gamma)=c(\delta)$.
 We are going to prove the result  by induction on $m$.

 If $m=0$, then $\lambda=\mu=\emptyset$ and there is nothing to prove.  Otherwise, we have $m>0$. Removing the last edge in  both $\gamma$  and $\delta$ yields two shorter paths $\gamma': \emptyset\rightsquigarrow \lambda'$
and $\delta': \emptyset \rightsquigarrow \mu'$ such that $c(\gamma')=c(\delta')$.
There are two cases we need to consider.

First, we assume that  $\lambda$ is obtained from $\lambda'$ by adding a box $x$. Thanks to Definition~\ref{wt123}, $\text{wt}(\lambda')=\text{wt}(\lambda)+\alpha_{c_{\delta_0(x)}}$. Since $\mu\in \Lambda(m)$,
$\text{wt}(\mu')=\text{wt}(\mu)+\alpha_{c_{\delta_0(x)}}$.
By induction assumption, we have  $\text{wt}(\mu') \preceq \text{wt}(\lambda')$, forcing   $\text{wt}(\mu) \preceq \text{wt}(\lambda)$.

Finally we assume that   $\lambda$ is obtained from $\lambda'$ by removing a box $y$.
Then $\text{wt}(\lambda')=\text{wt}(\lambda)-\alpha_{c_{\delta_0(y)}}$ and $\text{wt}(\mu')=\text{wt}(\mu)+\alpha_{-c_{\delta_0(y)}}$.
 By induction assumption, we have
 $$\text{wt}(\lambda)-\text{wt}(\mu)=\text{wt}(\lambda')-\text{wt}(\mu')+\alpha_{c_{\delta_0(y)}}+\alpha_{-c_{\delta_0(y)}}\in \mathbb N\Pi \text{ and }\overline{\text{wt}(\lambda)}=\overline{\text{wt}(\mu)}.$$
 Hence  $\text{wt}(\mu) \preceq \text{wt}(\lambda)$.
We remark that the second case happens only when  $\delta_0\in\mathbb Z1_K$.
\end{proof}
\begin{Cor}\label{hlink}
Suppose  $\lambda,\mu\in\Lambda_p$. If  $L(\lambda)$ and $L(\mu)$ are in the  same block of $B$-mod,  then $\text{wt}_0(\lambda)=\text{wt}_0(\mu)$.
\end{Cor}
\begin{proof}
If $\Hom_B(P(\lambda),P(\mu))\neq0$, then $ [P(\mu):L(\lambda)]\neq0$. Thanks to
Corollary ~\ref{ijxxexeu}, $P(\mu)$ has  a $\tilde \Delta$-flag and
$$ [P(\mu):L(\lambda)]=\sum_{\nu\in \Lambda} [\tilde \Delta(\nu):L(\mu)] [\tilde \Delta(\nu): L(\lambda)]. $$
So, there is  $\nu\in\Lambda$
such that $ [\tilde \Delta(\nu):L(\mu)] [\tilde \Delta(\nu): L(\lambda)]\neq 0$.
Thanks to  Proposition~\ref {dedji},  $\text{wt}(\lambda)\preceq \text{wt}(\nu)$ and $\text{wt}(\mu)\preceq \text{wt}(\nu)$, forcing
 $\text{wt}_0(\lambda)=\text{wt}_0(\mu)$ (see Definitions~\ref{prec} and \ref{wt123}).
\end{proof}

\begin{Cor}\label{jdiejdiej}
 Suppose $\lambda,\mu\in\Lambda_p$ and $\lambda\neq\mu$.  If  $[\bar \Delta(\lambda):L(\mu)]\neq0$, then $\text{wt}(\mu)\prec\text{wt}(\lambda)$.
\end{Cor}

\begin{proof}If  $[\bar \Delta(\lambda):L(\mu)]\neq0$, then $[\tilde \Delta(\lambda):L(\mu)]\neq0$ and hence $\text{wt}(\mu)\preceq\text{wt}(\lambda) $, by Proposition~\ref{dedji}.
Suppose $\lambda\in\Lambda_p(m)$ and $\mu\in\Lambda_p(n) $.
Thanks to Proposition~ \ref{dektss} and Corollary~\ref{xijixs}(2),
 $n=m+2d$ such that $d>0$ if  $\lambda\neq\mu$.. In particular,  $ \text{wt}(\lambda)\neq\text{wt}(\mu)$ and
  the result follows.
\end{proof}
 Recall two functions $ \text{wt}: \Lambda\rightarrow P$ and $ \text{wt}_0: \Lambda\rightarrow P_\theta$. We have
\begin{equation}\label{decompopar}
\Lambda=\coprod_{\rho\in P} \Lambda^\rho, \text{ and }\Lambda=\coprod_{\rho\in P_\theta} \Lambda^{0,\rho}
\end{equation}
where $\Lambda^\rho= \text{wt}^{-1}(\rho)$ and $\Lambda^{0,\rho}=\text{wt}^{-1}_0(\rho) $.
Restricting $ \text{wt}$ and $ \text{wt}_0$  to $\Lambda_p$ yields  $\Lambda^\rho_p$ and $\Lambda^{0,\rho}_p$, respectively. We consider some Serre subcategories as follows:
\begin{itemize} \item For any  $\rho \in P$, $B^0\text{-mod}_\rho\subset  B^0\text{-mod}$ such that its objects are
 $M$  satisfying $\text{wt}(\lambda)=\rho$ if $\Hom_{B^0}(Y(\lambda),M)\neq0$,
 \item For any  $\rho \in P_\theta$, $B\text{-mod}_\rho\subset  B\text{-mod}$
such that its objects are $M$ satisfying $\text{wt}_0(\lambda)=\rho$ if $\Hom_{B}(P(\lambda),M)\neq0$,
\item For any  $\rho \in P$, $B\text{-mod}_{\preceq\rho}\subset B\text{-mod}$ such that its
objects are $M$ satisfying $\text{wt}(\lambda)\preceq\rho$  if  $\Hom_{B}(P(\lambda),M)\neq0$,
\item For any  $\rho \in P$, $B\text{-mod}_{\prec\rho}\subset B\text{-mod}$ such that its
objects are $M$ satisfying $\text{wt}(\lambda)\prec\rho$, if  $\Hom_{B}(P(\lambda),M)\neq0$.
 \end{itemize}

%Thanks to the  well-known results on the blocks of symmetric groups,  we see   that the first equation in %\eqref{decompopar} give  the block decomposition of $B^0$-Mod.
\begin{Lemma} \label{duehdueh} Keep the notations above. \begin{itemize}\item [(1)]
$B^0\text{-mod} =\prod _{\rho\in P} B^0\text{-mod}_\rho$ gives a  block decomposition of $ B^0$-mod,
\item [(2)] $
B\text{-mod}=\prod _{\rho\in P_\theta} B\text{-mod}_{\rho}$  gives a partial  block decomposition of  $ B$-mod.
\end{itemize}\end{Lemma}

\begin{proof} (1) follows well-known results on block decomposition of module category for symmetric groups and (2) follows from
Corollary~\ref{hlink}. We remark that $B\text{-mod}_{\rho}$ may not be a single block.   \end{proof}

Restricting to $B^0\text{-lfdmod}$ and $B\text{-lfdmod}$  we get $ B^0\text{-lfdmod}_\rho$ and $B\text{-lfdmod}_\rho$ for any $\rho$.
\begin{Defn} For any $\rho\in P$, let
$\pi_{\rho}: B\text{-lfdmod}_{\preceq\rho}\rightarrow B^0\text{-lfdmod}_\rho$
be the functor defined first by restriction to $B^0$ then projection onto the block $B^0\text{-lfdmod}_\rho$.
Composing the inclusion of $B^0\text{-lfdmod}_\rho$ into $B^0$-lfdmod with $\Delta$ (resp., $\nabla$) yields  two functors
$\Delta_\rho: B^0\text{-lfdmod}_\rho\rightarrow B\text{-lfdmod}$ and $\nabla_\rho: B^0\text{-lfdmod}_\rho\rightarrow B\text{-lfdmod}$.
\end{Defn}
Obviously, $\pi_{\rho}$, $\Delta_\rho$ and $\nabla_\rho$ are exact functors.
\begin{Lemma}Both  $\Delta_\rho$ and  $\nabla_\rho$ are exact functors from $ B^0\text{-\rm lfdmod}_\rho$ to $ B\text{-\rm lfdmod}_{\preceq\rho}$.
 Moreover,  $\Delta_\rho$ and $\nabla_\rho$ are left and right adjoint to $\pi_\rho$, respectively.\end{Lemma}
\begin{proof} Suppose that $M$ is an object in $B^0$-lfdmod${_\rho}$. Then $\wt(\mu)=\rho$ if $\Hom_{B^0} (Y(\mu), M)\neq 0$. In other words,
 $\text{wt}(\mu)=\rho$ if  $[M:D(\mu)]\neq0$.  Since $\Delta$ is an exact functor, it gives a $\bar\Delta$-flag of $\Delta(M)$ such that each section is of form  $\bar \Delta(\mu)$ satisfying  $[M:D(\mu)]\neq0$.
Thanks to  Corollary~\ref{jdiejdiej}, $\text{wt}(\lambda)\preceq \text{wt}(\mu)=\rho$ if $ [\bar\Delta(\mu):L(\lambda)]\neq0$.
So, $\bar \Delta(\mu)\in \text{B-lfdmod}_{\preceq}\rho$ and hence $\Delta_\rho(M)=\Delta(M)\in\text{B- lfdmod}_{\preceq}\rho$. We actually have a functor
$\Delta_\rho: B^0\text{-lfdmod}_\rho\rightarrow B\text{-lfdmod}_{\preceq\rho}$. By Lemma~\ref{isodual}(2), we have the result for $\nabla_\rho$.
The remaining statements  follow from standard arguments (e.g. induction functor and restriction functor  form an adjoint pair).
\end{proof}
\begin{Lemma}\label{njsjxs}
Suppose  $\rho\in P$ and $\lambda\in\Lambda_p$.\begin{itemize} \item[(1)]  If  $\text{wt}(\lambda)\prec\rho$, then $\pi_\rho(L(\lambda))=0$.
\item [(2)]  The induced exact functor $\bar\pi_\rho:B\text{-\rm lfdmod}_{\preceq\rho}/B\text{-\rm lfdmod}_{\prec\rho}\rightarrow B^0\text{-\rm lfdmod}_\rho$
is an equivalence of categories.\end{itemize}
\end{Lemma}
\begin{proof}
Suppose $\Lambda^\rho\subset \Lambda(m)$ for some positive integer $m$. Thanks to  Definition~\ref{prec},  $\Lambda^{\sigma} \subset \Lambda(n)$ for some $n>m$ provided that  $\sigma\prec \rho $.
If $\lambda\in \Lambda^{\sigma}$, then $ L(\lambda)^s=1_{\ob n}L(\lambda)$   and  each weight space   $ 1_{\ob k} L(\lambda)$ of $L(\lambda)$ satisfies $k=n+2d$ for some $d\in\mathbb N$.
Therefore,  the restriction of $L(\lambda)$ to $B^0$ belongs to $\coprod_{d\in\mathbb N}K\mathfrak S_{n+2d}\text{-mod}$. Since $n>m$, we have  $\pi_\rho(L(\lambda))=0$.
So,   the  exact functor $\pi_\rho$ induces the required exact  functor  $\bar \pi_\rho$ and the simple objects in $B\text{-lfdmod}_{\preceq\rho}/B\text{-lfdmod}_{\prec\rho}$ are represented by the set
$\{L(\lambda)\mid \lambda \in \Lambda_p^\rho\}$.

Thanks to Proposition~\ref{dektss}, $P(\lambda)$ has a $\Delta$-flag such that  $\Delta(\lambda)$ appears as the top section.
By Corollaries~\ref{xijixs} and ~\ref{jdiejdiej}, $\mu\succ\rho$ if  $(P(\lambda):\Delta(\mu))\neq 0$ and  $\mu\neq \lambda$.
This proves that  $\Delta(\lambda)$ is the largest quotient of $P(\lambda)$ which belongs to $B\text{-lfdmod}_{\preceq\rho}$.
So, $\{\Delta(\lambda)\mid \lambda\in\Lambda_p^\rho\}$ is a complete set of all non-isomorphic  indecomposable projective objects in $B\text{-lfdmod}_{\preceq\rho}/B\text{-lfdmod}_{\prec\rho}$.
Suppose that $\lambda\in\Lambda_p^\rho\subset \Lambda_p(m)$.  As in the proof of Theorem~\ref{hd1}, $\Delta(\lambda)=\oplus_{d\in\mathbb N} 1_{\ob{m+2d}}\Delta(\lambda)$.
Then $\bar \pi_\rho( 1_{\ob {m+2d}}\Delta(\lambda))=0 $ for $d>0$ and   $ \bar\pi_\rho(\Delta(\lambda))=1_{\ob m}\Delta(\lambda)\cong Y(\lambda)$.
So, $ \bar\pi_\rho$ induces a bijection between isomorphism classes of all indecomposable projective objects in $B\text{-lfdmod}_{\preceq\rho}/B\text{-lfdmod}_{\prec\rho}$ and $ B^0\text{-lfdmod}_\rho$, proving (2).
\end{proof}
%Recall the notion of standardly stratified category from \cite{LW}.

\begin{Theorem}\label{fullystratified}
The category $B$-{\rm ldfmod} is an upper finite fully  stratified category in the sense of \cite[Definition~3.36]{BS}.  The corresponding  stratification is given by the function $\text{wt}:\Lambda_p\rightarrow P$
and the order $ \preceq$. The associated graded category is $B^0$-{\rm lfdmod}. Moreover, if $p=0$, then $B$-{\rm lfdmod} is an upper finite highest weight  category in the sense of \cite[Definition~3.36]{BS}.
\end{Theorem}

\begin{proof}Suppose that $\Lambda_p^{\rho}\subset\Lambda_p(m)$ and $\sigma\succeq \rho$. By Definition~\ref{prec}, $\text{wt}(\mu)\geq \text{wt}(\lambda)$ for any
$\mu\in \Lambda_p^{\sigma}$ and $\lambda\in \Lambda_p^{\rho}$. So, $\mu\in \Lambda_p(n)$ for some $n\leq m$ and    $ \coprod_{\sigma\succeq\rho}\Lambda^{\sigma}_p $ is a finite set.
Thanks to  Proposition~\ref{dektss} and  Corollaries~\ref{xijixs} and~\ref{jdiejdiej}, $P(\lambda)$ has a finite $\Delta$-flag such that  $\Delta(\lambda)$ appears as   the top section and moreover, other sections $\Delta(\mu)$ satisfying  $\text{wt}(\mu)\succ \text{wt}(\lambda)$.
Combining this with Lemma~\ref{njsjxs},  we see that $B$-lfdmod is an upper finite $+$-stratified category in the sense of \cite[Definition~3.36]{BS}.
 By Lemma~\ref{duehdueh}(1) and the definitions of $\Delta(\lambda)$ and $\bar\Delta(\mu)$ in \eqref{stahhss},  $\Delta(\lambda)$ has a $\bar \Delta$-flag with sections $\bar\Delta(\mu)$ such that  $\text{wt}(\lambda)=\text{wt}(\mu)$. Then $B$-lfdmod is also fully  stratified by the arguments in \cite[Lemma~3.22]{BS}.

If $p=0$, then $\Delta(\lambda)=\bar\Delta(\lambda)$ for all $\lambda\in\Lambda$ and $B$-lfdmod is   highest weight. The claim for the associated graded category follows from   Lemma~\ref{njsjxs}(2) directly.
\end{proof}

\begin{rem}\label{highestweight}
Theorem~\ref{fullystratified} can also be obtained  conceptually via Brundan and Stroppel's  recent work in \cite{BS} as follows (although one still needs Corollary~\ref{jdiejdiej} in order to obtain the precise stratification from Theorem~\ref{fullystratified} that incorporates the optimal partial order $\preceq$). Thanks to  Lemmas~\ref{useifcats}--\ref{triangularde}, $(B^-,B^0,B^+)$ is an upper finite triangular decomposition of $B$ in the sense of \cite[Definition~5.24]{BS} (hence a Cartan decomposition in the sense \cite[Definition~5.23]{BS} by \cite[Lemma~5.24]{BS}). By \cite[Theorem~5.30]{BS}, $B$
is an upper finite based stratified algebra in the sense of \cite[Definition~5.17]{BS}. Moreover, it is also an upper finite based quasi-hereditary algebra is the sense of \cite[Definition~5.1]{BS} if $p=0$ since $B^0$ is semisimple in this case.
 Thanks to \cite[Corollary~5.31]{BS}, $B\text{-lfdmod}$ is an upper finite fully stratified category with stratification $\text{wt}:\Lambda_p\rightarrow P$ and $B\text{-lfdmod}$ is highest weight if  $p=0$.
This implies  Propositions~\ref{ext}, \ref{dektss} directly. Thanks the referee for his/her explanation here.
\end{rem}
\section{Kazhdan-Lusztig theory  in the case $p=0$ and $\delta_0\in \mathbb Z$}
In this section,  we  assume that  $p=0$ and $\delta_0\in \mathbb Z $. Thanks to  Theorem~\ref{semi}, $B$-mod is completely reducible when $p=0$ and $\delta_0\notin \mathbb Z$. In this case, there is nothing to be discussed.
 We reformulate some data  for $\mathfrak{sl}_K$ in section 4 as follows:
\begin{itemize}\item $ I_0=\mathbb Z$  if $\delta_0$ is odd, and
    $I_0=  \mathbb Z+\frac{1}{2}$,  if $\delta_0$ is even,
   \item $P=\sum_{i\in  I_0}\mathbb Z\varpi_i$, where $\varpi_i$'s are fundamental weights,
    \item   $\alpha_i=\varepsilon_{i-\frac{1}{2}}-\varepsilon_{i+\frac{1}{2}}$, where  $\varepsilon_{i-1/2}=\varpi_{i}-\varpi_{i-1}$, for all $i\in I_0$,
    \item  $\theta$, the involution on   $P$ satisfying
$\theta(\varepsilon_i )=-\varepsilon_{-i}$, for all $ i\in \mathbb I:=I_0+1/2$.
\end{itemize}
Then  $\theta$  induces an automorphism of $\mathfrak {sl}_K$.
Consider the automorphism $\phi$ of $\mathfrak {sl}_K$ such that $\phi$ switches $e_i$ and $f_i$ for all admissible $i$ and sends $h$ to $-h$ for any $h$ in the Cartan subalgebra of $\mathfrak {sl}_{K}$. Then $\phi \circ \theta$  is an automorphism of $\mathfrak{sl}_K$ such that
$\phi \circ \theta (e_i)=f_{-i}$, $\phi \circ \theta (f_i)=e_{-i}$ and $\phi \circ \theta(h_i)=-h_{-i}$, $\forall  i\in I_0$.
Following \cite{BW,Bao}, we have    $\mathfrak g=\{g\in\mathfrak {sl}_K\mid \phi \circ \theta (g)=g\}$, a Lie subalgebra of $\mathfrak{sl}_K$,
 and $(\mathfrak{sl}_K$, $\mathfrak g)$ forms a classical symmetric pair. It is the infinite rank limit of the quantum symmetric pair (in the sense of \cite{Le}) $(U_q(\mathfrak {sl}_\infty), U^i_q(\mathfrak{sl}_\infty))$ at the limit $q\mapsto1$ in \cite{Bao}.

  Let $\mathbb V$ be the  natural representation  of $\mathfrak {sl}_K$ over $\mathbb C$.
It is the  $\mathbb C$-space with basis
$\{v_i\mid i\in\mathbb I \}$ such that
\begin{equation}\label{eexheu}
e_i v_a=\delta_{i+\frac{1}{2}, a}v_{a-1}, \quad f_i v_a=\delta_{i-\frac{1}{2},a}v_{a+1}, \forall (i, a)\in I_0 \times \mathbb I.
\end{equation}
Consider  the restricted dual $\mathbb W$ of $\mathbb V$. As a vector space, it  has  dual  basis $\{w_a\mid a\in\mathbb I\}$ such that
 $$e_i w_a=\delta_{i-\frac{1}{2}, a}w_{a+1}, \quad f_i w_a=\delta_{i+\frac{1}{2},a}w_{a-1}, \forall (i, a ) \in I\times \mathbb I.$$
 Restricting  $\mathbb W$ to $\mathfrak g$ yields  a $\mathfrak g$-module. Bao-Wang~\cite{BW} considered $\mathfrak{sl}_K$-module (and hence $\mfg$-module)
  $\bigwedge_d^{\infty}\mathbb V$ and
 $\bigwedge_d^{\infty}\mathbb W$,  the $d$th sector of  semi-infinite wedge space $\bigwedge^{\infty}\mathbb V$  and $\bigwedge^{\infty}\mathbb W$ such that
$$\textstyle\bigwedge^{\infty}\mathbb V=\bigoplus_{d\in \mathbb I}\bigwedge_d^\infty \mathbb V \text{ and }\bigwedge^{\infty}\mathbb W=\bigoplus_{d\in \mathbb I}\bigwedge_d^\infty \mathbb W.$$
 More precisely, $\bigwedge_d^{\infty}\mathbb V$ has  basis $\{v_{\mathbf i}=v_{i_1}\wedge v_{i_2}\wedge\ldots\mid \mathbf i\in \mathbb I_d^+\}$ and  the space $\bigwedge_d^{\infty}\mathbb W$  has basis $ \{w_{\mathbf i}=w_{i_1}\wedge w_{i_2}\wedge\ldots \mid \mathbf i\in \mathbb I_d^-\}$, where
$$ \begin{aligned}\mathbb I_d^{+}&=\{\mathbf i\in\mathbb I^\infty\mid i_1>i_2>\ldots, i_k=d-k+1  \text{ for } k\gg0\}, \text {and}\\
\mathbb I_d^{-}&=\{\mathbf i\in\mathbb I^\infty\mid i_1<i_2<\ldots, i_k=d+k  \text{ for $k\gg0$}\}.
\end{aligned}$$
 These bases are known as monomial bases in the literature.
 Moreover, these two bases  can be expressed as  $\{v_{\lambda, d}\mid \lambda\in\Lambda\}$ and $\{ w_{\lambda, d}\mid \lambda\in\Lambda\}$, where
$ v_{\lambda, d}:=v_{\lambda_1+d}\wedge v_{\lambda_2+d-1}\wedge v_{\lambda_3+d-2}\wedge\ldots$ and
$ w_{\lambda, d}:=w_{d+1-\lambda_1}\wedge w_{d+2-\lambda_2}\wedge w_{d+3-\lambda_3}\wedge\ldots$.

Recall that $  [K_0(B\text{-mod}^{\Delta})]$ is $\mathfrak{g}$-module such that $\tilde e_i$'s act via $\tilde E_i$'s in Theorem~\ref{cateofg}.
\begin{Theorem}\label{isoophi} As $\mathfrak{g}$-modules, $ [K_0(B\text{-\rm mod}^{\Delta})] \cong \bigwedge_{\frac{\delta_0}{2}-1}^{\infty}\mathbb W$, and the   required isomorphism
$\varphi$ sends $ [\Delta(\lambda)]$ to $w_{\lambda', \frac{\delta_0}{2}-1}$,  $\forall \lambda\in \Lambda$, where $\lambda'$ is the conjugate  partition of $\lambda$.
\end{Theorem}
\begin{proof} Thanks to \cite[Proposition~10.6]{BW},  $\bigwedge_d^{\infty}\mathbb V\cong\bigwedge_d^{\infty}\mathbb W$ as $\mathfrak {sl}_K$-modules and the required isomorphism, say $f$, sends  $v_{\lambda,d}$ to $w_{\lambda',d}$. It is well-known that there are  $\mathfrak {sl}_K$-isomorphisms
 $f_1:\bigwedge_d^{\infty}\mathbb V\cong V(\varpi_{d+\frac{1}{2}})$ and $f_2:\bigwedge_d^{\infty}\mathbb W\cong V(\varpi_{d+\frac{1}{2}})$ such that $f_1(v_{\emptyset,d})=v_{\varpi_{d+\frac{1}{2}} }$ and $f_2(w_{\emptyset,d})= v_{\varpi_{d+\frac{1}{2}} }$, where $v_{\varpi_{d+\frac{1}{2}} } $ is the nonzero highest weight vector with weight $\varpi_{d+\frac{1}{2}}$.
  Recall the $\mathfrak{sl}_K$-isomorphism $\phi$ in \eqref{nxjsnx}. Then $f_1^{-1}\circ \phi\circ \Delta^{-1}:[K_0(B\text{-mod}^\Delta)]\cong \bigwedge_d^{\infty}\mathbb V$, an $\mathfrak{sl}_K$-isomorphism sending  $ [\Delta(\emptyset)]$ to $v_{\emptyset,d}$, where $d=\frac{\delta_0}{2}-1$.
 Since $p=0$, we have $Y(\lambda)=S(\lambda)$ for any $\lambda\in\Lambda$.
Using Lemma~\ref{filteration}(3)--(4) and \eqref{eexheu}, we see that this isomorphism sends $ [\Delta(\lambda)]$ to $v_{\lambda,d}$ (see also \cite[Section~7]{Br}).
 So,
$f\circ f_1^{-1} \circ  \phi\circ \Delta^{-1}$ is the required  $\mfg$-isomorphism $[K_0(B\text{-mod})^\Delta]\cong \bigwedge_d^{\infty}\mathbb W$  sending  $ [\Delta(\lambda)]$ to $ w_{\lambda',d}$.
\end{proof}

%As mentioned before, the $\mfg$-module $ [K_0(B\text{-mod}^{\Delta})]$ categorifies $V(\varpi_{\frac{\delta_0-1}{2}})$ although there is an $\mathfrak{sl}_K$-structure on $ [K_0(B\text{-mod}^{\Delta})]$. In the later case, the Chevalley generators $e_i$ and $f_i$ can not be explained as some endofunctors on $B\text{-mod}^{\Delta}$.

For $n\geq2$, consider the free abelian group with basis $\{\zeta_1,\zeta_2,\ldots,\zeta_n\}$ with a symmetric bilinear form $(,)$ given by
$$(\zeta_i,\zeta_j)=\delta_{i,j}, \text{ for }i,j\in\{1,2,\ldots,n\}.$$
Set $\Pi_n=\{\beta_0,\beta_1, \beta_2, \ldots, \beta_{n-1}\}$, where
  $\beta_i=\zeta_i-\zeta_{i+1}$,  $ 1\le i\le n-1$, and $\beta_0=-\zeta_1-\zeta_2$.
Let $\mathfrak{so}_{2n}$ be the special orthogonal Lie algebra corresponding to the Dynkin diagram of type $D_n$ together with the set of simple roots $\Pi_n$ listed as follows :

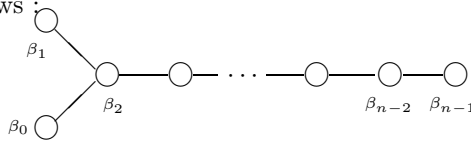
\begin{center}
\hskip 3cm \setlength{\unitlength}{0.16in}
\begin{picture}(24,3.5)
\put(8,2){\makebox(0,0)[c]{$\bigcirc$}}
\put(10.4,2){\makebox(0,0)[c]{$\bigcirc$}}
\put(14.85,2){\makebox(0,0)[c]{$\bigcirc$}}
\put(17.25,2){\makebox(0,0)[c]{$\bigcirc$}}
\put(19.4,2){\makebox(0,0)[c]{$\bigcirc$}}
\put(6,3.8){\makebox(0,0)[c]{$\bigcirc$}}
\put(6,.3){\makebox(0,0)[c]{$\bigcirc$}}
\put(8.4,2){\line(1,0){1.55}} \put(10.82,2){\line(1,0){0.8}}
\put(13.2,2){\line(1,0){1.2}} \put(15.28,2){\line(1,0){1.45}}
\put(17.7,2){\line(1,0){1.25}}
\put(7.6,2.2){\line(-1,1){1.3}}
\put(7.6,1.8){\line(-1,-1){1.3}}
\put(12.5,1.95){\makebox(0,0)[c]{$\cdots$}}
%\put(-.5,2){\makebox(0,0)[c]{$\mf{d}$:}}
\put(5.1,0.3){\makebox(0,0)[c]{\tiny$\beta_{0}$}}
\put(5.7,2.8){\makebox(0,0)[c]{\tiny$\beta_{1}$}}
\put(8.2,1){\makebox(0,0)[c]{\tiny$\beta_{2}$}}
\put(17.2,1){\makebox(0,0)[c]{\tiny$\beta_{n-2}$}}
\put(19.3,1){\makebox(0,0)[c]{\tiny$\beta_{n-1}$}}
\end{picture}.
\end{center}
There is a triangular decomposition   $\mathfrak {so}_{2n}=\mathfrak n_n^{-}\oplus \mathfrak t_n\oplus \mathfrak n_n^+$ corresponding the simple roots $\Pi_n$.
Let $\mathfrak t_n^*$ be the linear dual  of $\mathfrak t_n$ with  dual basis $\{\zeta_i\mid 1\le i\le n\}$.
Each element $\lambda\in \mathfrak t_n^*$, called a weight, is of form   $\sum_{i=1}^n \lambda_i\zeta_i$.
A weight $\lambda\in\mathfrak t_n^*$  is integral if it is the $\mathbb Z$-span or the $(\mathbb Z+\frac{1}{2})$-span of the $\zeta_i$'s.

Let $\mathfrak {so}_{2\infty}=\bigcup_{n\in \mathbb N}\mathfrak {so}_{2n}$. We also have $\mathfrak t_\infty$ and the simple system $\Pi_\infty=\{\beta_i\mid i\in\mathbb N\}$.
For $n\in\mathbb N\cup\{\infty\}$, let $ \bar \Pi_n=\Pi_n\setminus{\beta_0}$. Let $\mathfrak l_n$ (resp.,  $\mathfrak p_n$) be the Levi (resp.,  parabolic) subalgebra of $\mathfrak {so}_{2n}$ with respect to the set $ \bar \Pi_n$.
For $n\in\mathbb N$, define
\begin{equation}\begin{aligned}
X (n)&=\{\lambda\in \mathfrak t_n^* \text{ integral}\mid \lambda=\sum _{i=1}^n \lambda_i\zeta_i, \text { where $\lambda_1\geq \lambda_2\geq\ldots\geq \lambda_n$}\}\\
&=\{\lambda\in \mathfrak t_n^* \text{ integral}\mid \lambda+\rho_n=\sum _{i=1}^n \mu_i\zeta_i, \text { where $\mu_1> \mu_2>\ldots> \mu_n$}\},
\end{aligned}
\end{equation}
  where $\rho_n$ denotes the half-sum of positive roots. More explicitly,  $\rho_n=\sum_{i=1}^n(1-i)\zeta_i$. Define
 \begin{equation}\begin{aligned}
X(\infty)&=\{\lambda\in \mathfrak t_\infty^* \text{ integral}\mid \lambda=\sum _{i=1}^\infty \lambda_i\zeta_i, \text {$\lambda_1\geq \lambda_2\geq\ldots$ and $\lambda_i=\lambda_{i+1}$ for $i\gg0$}\}\\
&=\{\lambda\in \mathfrak t_\infty^* \text{ integral}\mid \lambda+\rho_\infty=\sum _{i=1}^\infty \mu_i\zeta_i, \text {$\mu_1> \mu_2>\ldots$ and $\mu_i=\mu_{i+1}+1$ for $i\gg0$}\},
\end{aligned}
\end{equation}
where  $\rho_\infty=\sum_{i=1}^\infty(1-i)\zeta_i$.
Let $L_{\mathfrak l_n}(\lambda)$ be the simple $\mathfrak l_n$-module with highest weight $\lambda$. Similarly, we have the parabolic Verma module
$\mathrm M_n(\lambda):= \text{Ind}_{\mathfrak p_n}^{\mathfrak {so}_{2n}}L_{\mathfrak l_n}(\lambda)$.
Let $ \mathrm {L}_{n}(\lambda)$ be the  simple quotient of $\mathrm M_n(\lambda)$. When $n=\infty$, we simply drop the subscript.   For example $\mathrm M(\lambda)=\mathrm M_\infty(\lambda)$ and $\mathrm L(\lambda)=\mathrm L_{\infty}(\lambda)$.
\begin{Defn}For $n\in\mathbb N\cup\{\infty\}$, let $\mathcal O(n)$ be the category of $\mathfrak t_n$-semisimple $\mathfrak {so}_{2n}$-modules $M$
such that
\begin{enumerate}\item $M=\oplus_\mu M_\mu$ and $\dim M_\mu<\infty$,
\item   $M$ decomposes into a direct sum of $\mathfrak p_n$-modules $L_{\mathfrak l_n}(\lambda)$ for $\lambda\in X(n)$;
\item there exist finitely many weights $\lambda^1,\lambda^2,\ldots,\lambda^k\in X(n)$ such that
$\mu\in \lambda^i-\sum_{\alpha\in\Pi_n}\mathbb N\alpha $ for some $i$  if $\mu$ is a weight of $M$.
\end{enumerate}\end{Defn}

%Let $ \mathcal O(n)^\Delta$ be the full subcategory of $\mathcal O(n)$ consisting of all modules possessing a finite parabolic Verma flag.
% Similarly we have $\mathcal O^{d}(n)^\Delta$.

\begin{Defn}\cite{CW}
The truncation functor $\text{tr}: \mathcal O (\infty)\rightarrow \mathcal O (n)$ is an exact functor  defined by
$$\text{tr}(M)=\oplus _\nu M_\nu, $$
where $M_\nu$ is the wight space of $M$ and the sum is over $\nu$ such that $(\nu,\zeta_j-\zeta_{j+1})=0$ for all $j\geq n+1$ and $j\in\mathbb N$.
\end{Defn}

For any $\lambda=\sum_{i=1}^\infty\lambda_i\zeta_i\in X(\infty)$, define $  \lambda^n=\sum_{i=1}^n\lambda_i\zeta_i\in X(n)$.
For any  $\lambda\in \Lambda$ and $d\in \mathbb I$,  let $\tilde\lambda=\sum _{i=1}^\infty \lambda_i\zeta_i +d\zeta_\infty$ and $ X^{d}(\infty)=\{\tilde \lambda\mid \lambda\in \Lambda\}$,  where $ \zeta_\infty=\sum_{i\geq1}\zeta_i$.
  Then $ X^{d}(\infty)\subset X(\infty)$ and
$\tilde\lambda^n\in X(n)$ for any $\lambda\in\Lambda$.

\begin{Lemma}\cite[Prop.~6.9]{CW}\label{turnca}
If $\lambda,\mu\in\Lambda$ and $ \ell(\lambda)\leq n$ and  $ \ell(\mu)\leq n$, then
$[\mathrm M^{\mathfrak p}(\tilde\lambda): \mathrm L(\tilde \mu)]=[\mathrm M_n(\tilde\lambda^n): \mathrm L_n(\tilde\mu^n)]  $, where $\ell(\lambda)$ is the number of non-zero parts of $\lambda$.
\end{Lemma}
\begin{proof} Thanks to \cite[Prop.~6.9]{CW}, we have $\text{tr}(Z(\tilde\lambda))=Z_n(\tilde\lambda^n)$  if $\ell(\lambda)\leq n$ and $0$, otherwise, for all $Z\in \{\mathrm M, \mathrm L\}$, and $ \lambda\in \Lambda$.
The result follows immediately since $\text{tr} $ is an exact functor. \end{proof}

\begin{Defn} For any $d\in\mathbb I$, define  \begin{itemize}
     \item
$\xi: \Lambda \rightarrow \mathbb I^{ -}_d,  \lambda \mapsto (d+1-\lambda_1,  d+2-\lambda_2 ,  d+3-\lambda_3,\ldots)$,
\item $\gamma: \mathbb I_d^{ -}\rightarrow X^{b}(\infty)$ such that  $\gamma(\mathbf i)=-\sum_{k\geq 1}i_k\zeta_k-\rho_\infty$, where $b=-d-1$.
\item $P_\lambda=\sum_{\mu\in\Lambda^+(m)}[\mathrm M(\gamma\circ\xi(\mu)):\mathrm L(\gamma\circ\xi(\lambda))]w_{\mu,d}$, where  $\lambda\in \Lambda(m)$ and $\Lambda^+(m)$ is given in Definition~\ref{ksiijejeddd}.
\end{itemize} \end{Defn}
Then both $\gamma$ and $\xi$ are bijections and  $w_{\lambda,d}=w_{\xi(\lambda)}$.

\begin{Theorem}\label{amisd} Let $\varphi:  [K_0(B\text{-\rm mod}^{\Delta})] \cong \bigwedge_{\frac{\delta_0}{2}-1}^{\infty}\mathbb W$ be the isomorphism
in Theorem~\ref{isoophi}. Then $\varphi([P(\lambda)])=P_{\lambda'}$ for all $\lambda\in \Lambda$, where $\lambda'$ is the conjugate  partition of $\lambda$. \end{Theorem}

\begin{proof} Suppose that $\lambda\in\Lambda(k)$.
For any $ \mu\in\Lambda^+(k)$, there is a positive integer $m$ such that $ \lambda,\mu\in\Lambda^+(m)$.
Let $\bar\mu=\gamma\circ\xi(\mu')$ and $\bar\lambda=\gamma\circ\xi(\lambda')$.
Suppose  either $\delta_0\neq0$ or $\lambda\neq\emptyset$.  Thanks to Corollary~\ref{xijixs}, Proposition~\ref{xeijdiec} and Lemma~\ref{turnca}, we have the first, second and last equalities as follows:
$$(P(\lambda):\Delta(\mu))=[\Delta(\mu):L(\lambda)]=[C_m(\mu):L_m(\lambda)]
=[\mathrm M_n({\bar\mu}^n): \mathrm L_n(\bar\lambda^n)]
=[\mathrm M ( \bar\mu  ):\mathrm L( \bar\lambda) ],$$
where $n\in \mathbb Z$ and  $n\gg0$. Finally,   the third equation follows from  \cite[Theorem~5.6]{ES1}.

Suppose that $\delta_0=0$. Then $[\mathrm M(\bar\mu):\mathrm L(\bar\emptyset)]=\delta_{\mu,\emptyset}$ since $\bar\emptyset^n $ is dominant integral for   $n\geq2$.
By Proposition~\ref{dektss}(3),  $P(\emptyset)=\Delta(\emptyset)$.

In any case, we have proved that  $\varphi([P(\lambda)])=P_{\lambda'} $ for all $\lambda\in \Lambda$, where $\varphi$ is the isomorphism given in Theorem~\ref{isoophi}.
\end{proof}

By Theorem~\ref{amisd},  we see that $\{P_\lambda\mid \lambda\in \Lambda\}$ is another basis of $\bigwedge_{\frac{\delta_0}{2}-1}^{\infty}\mathbb W$.
By Lemma~\ref{turnca}, $ [\mathrm M(\gamma\circ\xi(\mu)):\mathrm L(\gamma\circ\xi(\lambda))]$ in the definition of $P_\lambda$ can be computed by
parabolic Kazhdan-Lusztig polynomials of type $D$ with maximal parabolic subgroup of type $A$\cite[Prop.~7.5]{Sor}.
  This is the reason why we call
 $\{P_\lambda\mid \lambda\in \Lambda\}$ is the quasi-canonical  basis  of $\bigwedge_{\frac{\delta_0}{2}-1}^{\infty}\mathbb W$.

In Remark~\ref{highestweight}, we see that $B$-lfdmod is an upper finite highest weight category when $p=0$.
Then $B$-lfdmod has a semi-infinite Ringel dual $ B\text{-lfdmod}^{\text{R}}  $ in the sense of \cite{BS}. Further, the indecomposable tilting modules in $ B\text{-lfdmod}^{\text{R}} $ have finite $\Delta$-flags in the same way as the projectives for $B$ with different combinatorial labeling since Ringel duality reverses the order on the underlying poset.
 When $p=0$ and $\delta_0\in\mathbb Z$, we expect that the indecomposable tilting module in $B\text{-lfdmod}^{\text{R}} $ matches exactly the canonical basis of $\bigwedge_{\frac{\delta_0}{2}-1}^{\infty}\mathbb W$ studied in \cite{Bao,BW} by using the Schur-Weyl duality between the coideal subalgebra of $U_q(\mathfrak {sl}_n)$ and Hecke algebra of type $D$.

%In Remark~\ref{highestweight}, we see that $B$-Mod is a highest weight category when $p=0$. So there is an indecomposable tilting module $T(\lambda)$ in $B$-Mod and
%$T(\lambda)\in B\text{-mod}^\Delta$,
%for all $\lambda\in\Lambda$. We expect that $\{\varphi([T(\lambda)])\mid \lambda \in\Lambda\}$ is the canonical basis of $\bigwedge_{\frac{\delta_0}{2}-1}^{\infty}\mathbb W$ studied in \cite{Bao,BW} by using the Schur-Weyl duality between the coideal subalgebra of $U_q(\mathfrak {sl}_n)$ and Hecke algebra of type $D$.

\end{document}